\documentclass[10pt]{amsart}
\usepackage{a4wide}

\usepackage{amsmath,amssymb,amsfonts,amsthm,hyperref,cite,graphicx}

\usepackage{color,marginnote,mathtools,mathrsfs}
\usepackage[normalem]{ulem}

\def\R{\mathbb{R}}

\def\Div{\textup{div}}

\def\e{\varepsilon}
\def\de{\delta}
\newcommand{\medint}{-\kern -,375cm\int}
\renewcommand{\H}{\mathcal{H}}
\newcommand{\beq}{\begin{equation}}
\newcommand{\eeq}{\end{equation}}
\newcommand{\pa}{\partial}
\newcommand{\Sp}{\mathfrak M}
\newcommand{\noop}[1]{}
\theoremstyle{plain}
\newtheorem{theorem}{Theorem}[section]
\newtheorem{proposition}[theorem]{Proposition}
\newtheorem{corollary}[theorem]{Corollary}
\newtheorem{lemma}[theorem]{Lemma}
\newtheorem{definition}[theorem]{Definition}
\theoremstyle{remark}
\newtheorem{remark}[theorem]{Remark}
\newcommand{\Z}{\mathbb Z}

\newcommand{\N}{\mathbb N}

\newcommand{\Om}{\Omega}
\newcommand{\wto}{\rightharpoonup}
\newcommand{\eps}{\varepsilon}

\numberwithin{equation}{section}

\title[Domain walls in ferromagnetic strips]{Transverse domain walls
  in thin ferromagnetic strips}

\author{M. Morini} \address[M. Morini]{Dipartimento di Scienze
  Matematiche, Fisiche e Informatiche, Universit\`a degli Studi di
  Parma, via Università 12, 43121 Parma, Italy}
\email{massimiliano.morini@unipr.it} 

\author{C. B. Muratov} \address[C. B. Muratov]{Department of
  Mathematical Sciences, New Jersey Institute of Technology, Newark NJ
  07102, USA} \email{muratov@njit.edu}

\author{M. Novaga} \address[M. Novaga]{Dipartimento di Matematica,
  Universit\`a di Pisa, Largo Bruno Pontecorvo 5, 56127 Pisa, Italy}
\email{matteo.novaga@unipi.it}

\author{V. V. Slastikov} \address[V. V. Slastikov]{School of
  Mathematics, University of Bristol, Bristol BS8 1UG, United Kingdom}
\email{valeriy.slastikov@bristol.ac.uk}

\begin{document}

\maketitle

\begin{abstract}
  We present a characterization of the domain wall solutions arising
  as minimizers of an energy functional obtained in a suitable
  asymptotic regime of micromagnetics for infinitely long thin film
  ferromagnetic strips in which the magnetization is forced to lie in
  the film plane. For the considered energy, we provide existence,
  uniqueness, monotonicity, and symmetry of the magnetization profiles
  in the form of 180$^\circ$ and 360$^\circ$ walls. We also
  demonstrate how this energy arises as a $\Gamma$-limit of the
  reduced two-dimensional thin film micromagnetic energy that captures
  the non-local effects associated with the stray field, and
    characterize its respective energy minimizers.
\end{abstract}

\tableofcontents

\section{Introduction}
\label{sec:introduction}

Advances in nanofabrication techniques have enabled an unprecedented
degree of precision and control in producing a wide variety of solid
state materials and devices in the form of atomically thin films and
multilayers \cite{stepanova}. For ferromagnetic materials, this
control offers opportunities to develop novel principles of
information processing and storage based on {\em spintronics} -- an
emergent discipline of electronics in which both the electric charge
and the quantum mechanical spin of an electron are harnessed
\cite{bader10}. In addition to the present day use of spin valves as
magnetic field sensors in hard-disk drive read heads \cite{zhu06rev},
some more recent applications of spintronic technology include domain
wall logic and computing \cite{allwood05,sharard12,manipatruni18},
magnetoresistive random access memory
\cite{apalkov16,fukami09,ross_patent05,zhu03,mo:ieeetm09} and
racetrack memory \cite{parkin08}.

In a typical domain wall device, a bit of information is encoded using
the position and polarity of a head-to-head wall along a thin, long
ferromagnetic nanostrip. By ``head-to-head'', one understands a
magnetization configuration in which the magnetization points along
the strip axis, but in the opposite directions at the opposite
extremes of the strip \cite{dennis02}. The structure of such a domain
wall in soft ferromagnets rather sensitively depends on the ratio of
the strip thickness and width to the characteristic length scale of
the ferromagnetic material (the exchange length
$\ell_\mathrm{ex} = \sqrt{2 A / (\mu_0 M_s^2)}$, where $A$ is the
exchange stiffness, $M_s$ is the saturation magnetization and $\mu_0$
is vacuum permeability \cite{hubert}). Depending on the film
thickness, one observes two basic types of walls -- the transverse and
the vortex wall -- for thinner and thicker films, respectively. This
picture was first established numerically by McMichael and Donahue via
micromagnetic simulations \cite{mcmichael97}, and later corroborated
by Kl\"aui {\em et al.}  through experimental studies in ferromagnetic
nanorings \cite{klaui04,laufenberg06} (for reviews, see
\cite{klaui08b,thiaville09}). Furthermore, as was shown numerically by
Nakatani, Thiaville and Miltat \cite{nakatani05}, there exist at least
two types of transverse domain walls: symmetric and asymmetric
walls. Finally, winding domain walls in which the magnetization
rotates by a 360-degree angle in the film plane are also known to
exist in ferromagnetic nanostrips \cite{kunz09,jang12,zhang15}. These
types of domain wall profiles, obtained numerically using the method
from \cite{mo:jcp06}, are illustrated in Fig. \ref{f:dw}.

\begin{figure}
  \centering
  \includegraphics[width=12cm]{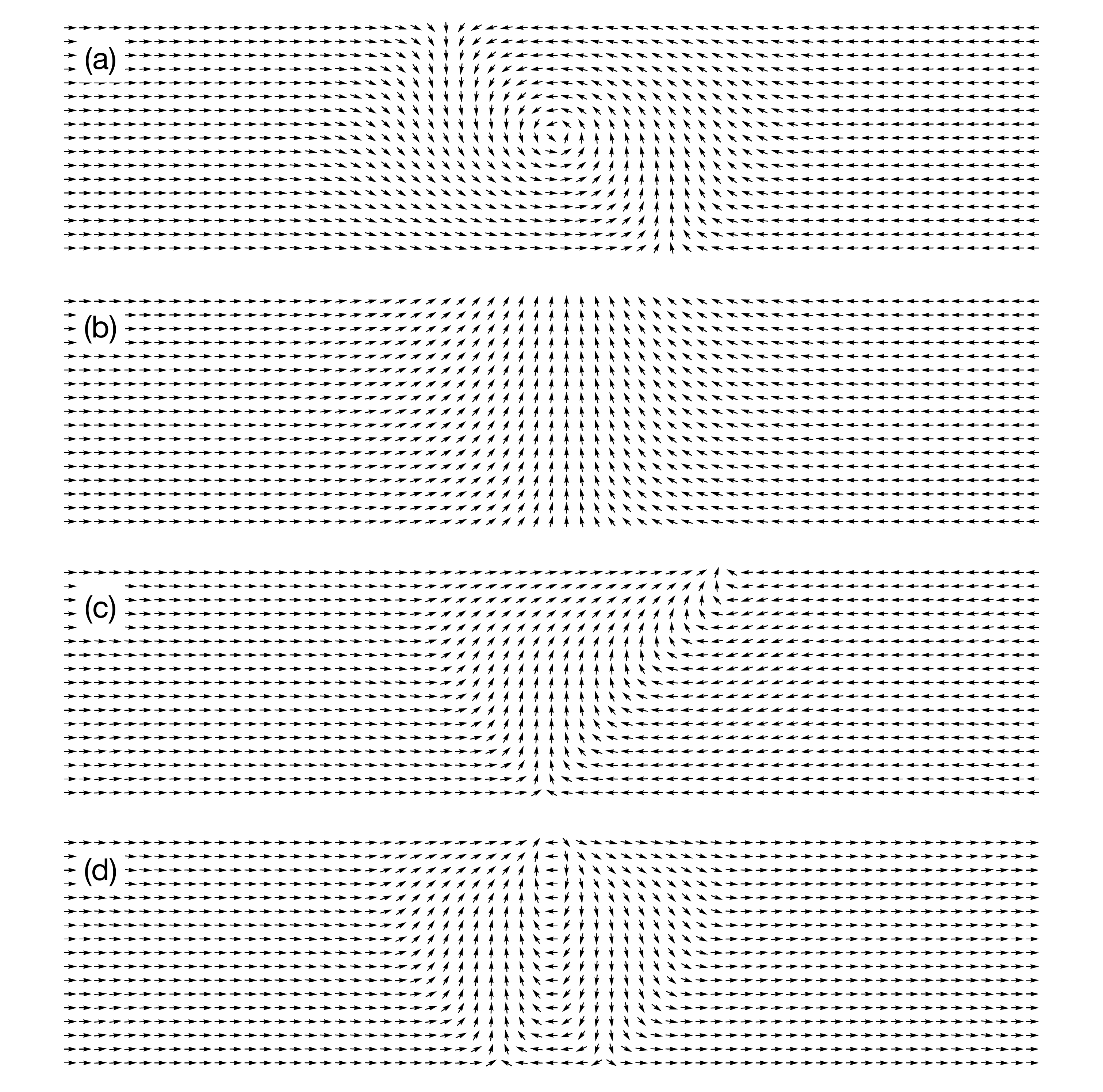}
  \caption{Domain wall profiles in the numerical simulations of
    amorphous cobalt nanostrips: (a) vortex head-to-head wall in a 100
    nm wide and 5 nm thick strip; (b) symmetric transverse
    head-to-head wall in a 50 nm wide and 2 nm thick strip; (c)
    asymmetric head-to-head wall in a 400 nm wide and 5 nm thick
    strip; (d) a winding transverse domain wall in a 400 nm wide and 5
    nm thick strip. The material parameters are: exchange constant
    $A = 1.4 \times 10^{-11}$ J/m, saturation magnetization
    $M_s = 1.4 \times 10^6$ A/m, and zero magnetocrystalline
    anisotropy or applied magnetic field \cite{li01}. For this
    material, the exchange length is $\ell_\mathrm{ex} = 3.37$ nm.}
  \label{f:dw}
\end{figure}

The mathematical understanding of domain wall profiles in ferromagnets
rests on the micromagnetic modeling framework, whereby the
magnetization configurations representing these profiles are viewed as
local or global minimizers of the micromagnetic energy functional
\cite{hubert,landau8}. This framework has been successfully used to
characterize a great variety of domain walls and other magnetization
configurations (for an overview, see \cite{desimone06r}; for some more
recent developments, see
\cite{ignat10,cm:non13,doering14,ignat17,ms:prsla17,lms:non18,
  kmn:arma19 ,lms:jns20}). However, head-to-head domain walls pose a
fundamental challenge to micromagnetic modeling and analysis, since
these magnetization configurations carry a non-zero magnetic charge,
which may lead to divergence of the wall energy in infinite samples
due to singular behaviors of the stray field \cite{ lms:non18}. To
date, there have been only a handful of micromagnetic studies of such
charged domain walls \cite{kuehn06, kuehn07, harutyunyan14,
  harutyunyan16, lms:non18, lms:jns20,knupfer20}.

In \cite{kuehn06,kuehn07}, K\"uhn studied head-to-head domain walls in
cylindrical nanowires of radius $R > 0$. These walls are viewed as
global minimizers of the energy
\begin{align}
  \label{eq:1}
  \mathcal E(m) := \frac12 \int_\Sigma |\nabla m|^2 d^3r + \frac12
  \int_{\mathbb R^3} |\nabla u|^2 d^3r,
\end{align}
where $m \in H^1_{loc}(\Sigma; \mathbb S^2)$,
$\Sigma = \Sigma_R := \mathbb R \times B_R(0) \subset \mathbb R^3$,
and $u \in \mathring{H}^1(\mathbb R^3)$ is the magnetostatic potential
solving
\begin{align}
  \label{eq:2}
  \Delta u = \nabla \cdot m
\end{align}
distributionally in $\mathbb R^3$, with $m$ extended by zero to
$\mathbb R^3 \backslash \Sigma$. The magnetization $m$ is subject to
the condition at infinity
\begin{align}
  \label{eq:3}
  m(x, y, z) \to (\pm 1, 0, 0) \qquad \text{as} \qquad x  \to \pm
  \infty,  
\end{align}
in some average sense (for a recent discussion of variational
principles of micromagnetics, see \cite{dmrs:sima20}). K\"uhn
considered existence of minimizers of $\mathcal E$ in a suitable class
of magnetizations $m$ for which \eqref{eq:3} holds, as well as a
number of their characteristics depending on $R$. In particular, she
showed that as $R \to 0$ the domain wall profile is expected to
converge, in an appropriate sense, to that of a one-dimensional
transverse wall, which is given explicitly, up to translations along
the $x$-axis and rotations in the $yz$-plane, by
\begin{align}
  \label{eq:4}
  m(x, y, z) = \left(  \tanh (x /
  \sqrt{2}), \text{sech} (x / \sqrt{2}), 0 \right).
\end{align}
Existence and convergence of minimizers were later established by
Harutyunyan for general cylindrical domains
$\Sigma = \R \times \Omega$, where $\Omega \subset \mathbb R^2$ is a
bounded domain with a $C^1$ boundary \cite{harutyunyan16} (see also
\cite{slastikov12}). In \cite{harutyunyan14}, Harutyunyan also studied
the behavior of the limit energy when $\Omega$ is a rectangle with a
large aspect ratio and obtained an additional logarithmic factor in
the scaling of the optimal energy (for sharp asymptotics, see
\cite{gaididei17}).

In the case of $\Sigma = \Sigma_R$ with $R > 0$ sufficiently small,
the analysis mentioned above is enabled by the fact that as $R \to 0$
the magnetization becomes essentially constant in the $yz$-plane,
allowing to asymptotically reduce the energy to
$\mathcal E(m) \simeq \mathcal E_0^\mathrm{1d}(\bar m)$, where
$\bar m(x, y, z) := \displaystyle\lim_{R \to 0} \left( {1 \over \pi
    R^2} \int_{B_R(0)} m(x, y', z') \, dy' dz' \right)$ and
\begin{align}
  \label{eq:5}
  \mathcal E_0^\mathrm{1d}(\bar m) := \int_{\Sigma_R} \left( \frac12
  |\nabla \bar m|^2 + \frac14 \left( 1 - \bar m_1^2 \right)
  \right) d^3r,  
\end{align}
whose minimizers among all
$\bar m \in \mathring{H}^1(\Sigma_R; \mathbb S^2)$ with
$\bar m = \bar m(x)$ satisfying \eqref{eq:3} are given by
\eqref{eq:4}, up to translations and rotations in the $yz$-plane. The
latter follows from the fact that the limit energy
$\mathcal E_0^\mathrm{1d}$ in \eqref{eq:5} is fully local, and its
minimizers satisfy a simple ordinary differential equation that can be
solved explicitly. The situation becomes much more complicated for
general values of $R \gtrsim 1$ or for general cross-sections
$\Omega$, since in that case the Euler-Lagrange equation for the
minimizers of $\mathcal E$ is a system of nonlinear partial
differential equations whose explicit solution is no longer
available. In particular, it is not known whether or not the
minimizers could exhibit winding, whereby the magnetization rotates by
an integer multiple of 360$^\circ$ along the axis of the wire, as,
e.g., in Fig. \ref{f:dw}(d).

In the absence of exact solutions and in view of the interest from
applications, one can alternatively focus on the case of
asymptotically thin films, i.e., for $\delta \ll 1$ to consider the
energy $\mathcal E_\delta(m)$ given by $\mathcal E(m)$ in
\eqref{eq:1}, in which
$\Sigma = \Sigma_\delta := \mathbb R \times (0, \delta) \times (0,
w_\delta)$. Here $\delta > 0$ is the film thickness and $w_\delta > 0$
is the film width, respectively, both in the units of the exchange
length, with the dependence of $w_\delta$ on $\delta$ as
$\delta \to 0$ to be specified. Notice that if $\Sigma_\delta$ were a
bounded domain with the lateral extent of order $w_\delta$, then from
the results of Kohn and Slastikov \cite{kohn05arma} one could conclude
that the full micromagnetic energy $\mathcal E$ asymptotically reduces
to $\mathcal E(m) \simeq \mathcal E_0^\mathrm{2d}(\bar m)$, where
$\bar m(x, y, z) := \displaystyle \lim_{\delta \to 0} \left( {1 \over
    \delta} \int_0^\delta m(x, y, z') \, dz' \right)$ such that
$\bar m_3 = 0$ and
\begin{align}
  \label{eq:6}
  \mathcal E_0^\mathrm{2d}(\bar m) := \frac12 \int_{\Sigma_\delta} 
  |\nabla \bar m|^2 d^3r + {\gamma \over w_\delta} \int_{\Gamma_\delta}
  (\bar m \cdot \nu)^2 d \mathcal H^2,
\end{align}
where $\Gamma_\delta$ is the portion of the boundary of
$\Sigma_\delta$ associated with the film edge and $\nu$ is the outward
unit normal to $\Gamma_\delta$, provided
\begin{align}
  \label{eq:7}
  w_\delta = {4 \pi \gamma \over \delta \ln \delta^{-1}},
\end{align}
for some $\gamma > 0$ fixed, as $\delta \to 0$.

Rescaling all lengths {in the film plane} with $w_\delta$ and
writing $\bar m = (\cos \theta, \sin \theta)$, we then formally have
$\mathcal E(m) \simeq \mathcal F(\theta) \delta$, where
\begin{align}
  \label{eq:8}
  \mathcal F(\theta) := \frac12 \int_{\Sigma_0} |\nabla
  \theta|^2 d^2r + \gamma \int_{\partial \Sigma_0}
  \sin^2 \theta \, d\mathcal H^1,
\end{align}
$\Sigma_0 := \mathbb R \times (0,1) $ denotes an infinite strip of
unit width, and $\theta \in C^1(\overline\Sigma_0)$, for example. As
expected, in this scaling regime the contribution of the stray field
to the energy localizes to become a nonlinear boundary penalty term,
greatly simplifying the otherwise highly nonlocal problem for the
domain wall profiles.

Note, however, that finding the profile in this case does not reduce
to solving an ordinary differential equation for the magnetization
angle, as in the case of thin ferromagnetic wires discussed
earlier. Instead, the problem may be reduced to a one-dimensional
fractional differential equation. To see this, let us formally reduce
the minimization problem for $\mathcal F$ to the problem for the trace
of $\theta$ on $\partial \Sigma_0$ (for details, see Appendix
\ref{s:app}). It is easy to see that any minimizer of $\mathcal F$ in
the form of a domain wall must be reflection-symmetric with respect to
the midline of $\Sigma$. Hence for a given trace
$\bar \theta \in C^\infty(\R)$ of $\theta$ on $\partial \Sigma_0$ such
that
\begin{align}
  \label{eq:36}
  \bar\theta(x) = k_1 \pi \quad \forall x < -R, \qquad \qquad \bar
  \theta(x) = k_2 \pi \quad \forall x > R,
\end{align}
for some $R > 0$ and $k_1, k_2 \in \mathbb Z$ we can minimize the
Dirichlet integral by choosing $\theta$ to be the harmonic extension
of $\bar \theta$. A direct computation then shows that
$\mathcal F(\theta) = 2 \bar{\mathcal F}(\bar \theta)$, where
\begin{align}
  \label{eq:9}
  \bar{\mathcal F}(\bar \theta) := \frac14 \int_{\mathbb R} \int_{\mathbb R}
  K(x - x') (\bar \theta(x) - \bar \theta(x'))^2 dx \, dx' + \gamma
  \int_{\mathbb R} \sin^2 \bar\theta(x) \, dx,
\end{align}
in which the symmetric, positive definite kernel
\begin{align}
  \label{eq:10}
  K(x) := {\pi \cosh ( \pi x) \over \sinh^2 (\pi x)} 
\end{align}
has the same singularity at the origin as the kernel generating
$(-d^2/dx^2)^{1/2}$ \cite{dinezza12} and decays exponentially at
infinity.

The Euler-Lagrange equation corresponding to $\bar {\mathcal F}$ reads
\begin{align}
  \label{eq:11}
  \frac12 \int_{\mathbb R} \big( 2 \bar \theta(x) - \bar \theta(x - \xi)
  - \bar \theta (x + \xi) \big) K(\xi) d \xi + \gamma \sin 2 \bar
  \theta(x) = 0  \qquad \forall x \in \mathbb R.
\end{align}
This equation is reminiscent of the fractional Ginzburg-Landau
equation studied in \cite{cabre05,palatucci13,cmy:mmnp17}, which is
known to exhibit transition layer profiles connecting the limits at
infinity that differ by $\pm \pi$ corresponding to the adjacent minima
of the wells of the potential appearing in the last term in
\eqref{eq:9}. Contrary to the problem in
\cite{cabre05,palatucci13,cmy:mmnp17}, however, the infimum of the
energy in \eqref{eq:9} is finite, making it amenable to analysis via
direct minimization. Note that when $\gamma \gg 1$, minimizers of
$\bar {\mathcal F}$ are expected to concentrate on the
$O(\gamma^{-1})$ length scale (for a closely related problem, see
\cite{kurzke06}). In this case one can approximate
$K(x) \simeq {1 \over \pi x^2}$, for which all domain wall type
solutions of \eqref{eq:11} are \cite{toland97}
\begin{align}
  \label{eq:12}
  \bar \theta(x) = \pm \arctan 2 \gamma x + {\pi \over 2},
\end{align}
up to translations and additions of integer multiples of $\pi$. Thus,
the head-to-head domain wall profiles minimizing $\mathcal E$ with
$\Sigma = \mathbb R \times (0, \delta) \times (0, w_\delta) $ in the
regime of $\delta \ll 1$ and $w_\delta$ given by~\eqref{eq:7} with
$\gamma \gg 1$ are expected to consist of magnetizations rotating in
the film plane in the form of two symmetric boundary vortices on the
opposite sides of the strip, consistently with the heuristics
presented in \cite{tchernyshyov05}. Alternatively, when
$\gamma \ll 1$, one would expect the minimizers of $\bar {\mathcal F}$
to vary on an $O(\gamma^{-1/2})$ scale, for which one can approximate
$x^2 K(x) \simeq \delta(x)$, where $\delta(x)$ is the Dirac
delta-function (cf. also \cite{bourgain01}). In this case
\eqref{eq:11} would reduce to an ordinary differential equation
\begin{align}
  \label{eq:26}
  {d^2 \bar \theta(x) \over dx^2} = 2 \gamma \sin 2 \bar \theta(x)
  \qquad \forall x \in \mathbb R, 
\end{align}
whose all domain wall type solutions are
$\bar \theta(x) = \pm 2 \, \arctan \left( e^{2 \sqrt{\gamma} \, x}
\right)$, up to translations and additions of integer multiples of
$\pi$. After a suitable rescaling and a possible reflection, these
correspond to the profile in \eqref{eq:4}.

The minimization of the energy \eqref{eq:9} could in principle be
carried out directly, yielding existence and properties of minimizers
for \eqref{eq:8}. The situation becomes more complicated, however, in
the presence of an applied external field $h > 0$ along the strip,
which amounts to an extra Zeeman term \cite{hubert} added to the
energy in \eqref{eq:1}:
\begin{align}
  \label{eq:23}
  \mathcal E(m) := \frac12 \int_\Sigma |\nabla m|^2 d^3r + h
  \int_\Sigma \left( 1 - m_1 \right) d^3 r + \frac12
  \int_{\mathbb R^3} |\nabla u|^2 d^3r,
\end{align}
after subtracting a suitable additive constant. At the level of the
limit thin film energy in \eqref{eq:8}, this translates into
\begin{align}
  \label{eq:24}
  \mathcal F(\theta) := \frac12 \int_{\Sigma_0} |\nabla
  \theta|^2 d^2r + h \int_{\Sigma_0} (1 - \cos \theta) \, d^2 r + \gamma
  \int_{\partial \Sigma_0} \sin^2 \theta \, d\mathcal H^1,
\end{align}
and clearly one could no longer explicitly minimize the first two
terms in the energy for a given trace $\bar\theta$, as this would
involve solving a nonlinear partial differential equation for $\theta$
in $\Sigma_0$. Instead, we will work directly with the energy in
\eqref{eq:24} and study its minimizers for $h \geq 0$ that connect
distinct equilibrium solutions $\theta = \mathrm{const}$ as
$x \to \pm \infty$.

We first focus on \eqref{eq:24} and establish existence of energy
minimizers that connect distinct equilibria at $x = \pm \infty$, using
the direct method of calculus of variations. The difficulty here is
the fact that the problem is posed on an unbounded domain and,
therefore, a priori minimizing sequences may fail to converge to a
function that has the right behavior at infinity. We overcome this
difficulty by proving monotonicity of the minimizers on larger and
larger truncated domains with prescribed Dirichlet data at the left
and the right ends of the truncated strip. Taking the limit of the
sequence of truncated minimizers, after suitable translations, we
obtain a limiting monotone function. Combining this fact with the
knowledge of the behavior at infinity for functions with bounded
energy \eqref{eq:24} (see Lemma~\ref{lm:zagara}), we show that this
limiting function is non-trivial and has the appropriate behavior at
infinity. By lower semicontinuity of the energy, we subsequently
conclude that the obtained limit is the desired minimizer.

Notice that the Euler-Lagrange equation form the energy in
\eqref{eq:24} is reminiscent of problems arising in the studies of
front solutions in infinite cylinders, on which there exists an
extensive literature. For example, when $\gamma = 0$ and $h > 0$ the
existence and qualitative properties of such solutions were
established in \cite{berestycki92}. A novel aspect of the considered
problem is the fact that the bistable nonlinearity enabling existence
of the front solutions is concentrated on the domain boundary (for
several studies of problems of this kind, see
e.g. \cite{cabre05,consul96,arrieta99,heinze}; this list is certainly
not exhaustive). Our contribution here is to develop a set of tools to
address the problems with boundary nonlinearities based on maximum and
comparison principles and the sliding method.  Using these tools, we
completely classify the critical points corresponding to domain wall
solutions and establish regularity, symmetry, uniqueness, monotonicity
and decay properties of the domain wall profiles. In particular, we
show that after reflections, translations and shifts in $\theta$, all
domain wall solutions associated with \eqref{eq:24} are the energy
minimizers that connect two distinct equilibria at infinity with no
winding for $h = 0$ (symmetric $180^\circ$ walls) or the same
equilibrium at infinity with exactly one rotation for $h > 0$
(symmetric $360^\circ$ walls). We also establish the explicit limit
behaviors of the minimizers in the limiting regimes of $\gamma \to 0$
and $\gamma \to \infty$ when $h = 0$.

We finally relate the minimization problem associated with
\eqref{eq:24} with that of the original micromagnetic problem
associated with \eqref{eq:23}. To this end, we introduce a reduced
thin film micromagnetic energy functional that is appropriate for
modeling ultrathin ferromagnetic films in which the ferromagnetic
layer has thicknesses down to a few atomic layers and, strictly
speaking, the macroscopic energy functional in \eqref{eq:23} is no
longer applicable. This two-dimensional reduced thin film energy
functional retains the nonlocal character of the micromagnetic energy
in \eqref{eq:23} in the ultrathin film regime and was introduced by us
earlier in the studies of exchange biased films \cite{lms:jns20}. It
represents an intermediate level of modeling between the full
three-dimensional micromagnetic energy in \eqref{eq:23} and the
two-dimensional thin film limit energy in \eqref{eq:24}. Notice that
the latter formally coincides with the one identified by Kohn and
Slastikov in their studies of thin film limits of ferromagnets of
finite lateral extent \cite{kohn05arma}. However, their analysis is no
longer applicable in our setting due to the loss of compactness
associated with the unbounded domain occupied by the ferromagnet.  For
this reason, we had to develop a series of new tools to tackle these
issues in order to be able to prove the $\Gamma$-convergence of the
reduced thin film energy (to be introduced in the following section,
see \eqref{epsilone}) to the limit thin film energy in \eqref{eq:24},
together with compactness and convergence of the respective energy
minimizers as the film thickness goes to zero. Importantly, we also
prove that at small but finite film thickness the non-trivial energy
minimizers of the reduced thin-film energy \eqref{epsilone} remain
close in a certain sense to the unique minimizers of the limit problem
associated with \eqref{eq:24}. In particular, they exhibit the same
head-to-head (for $h=0$) or winding (for $h>0$) behavior.


  Our paper is organized as follows. In
  Sec. \ref{sec:statement-results}, we state precisely the variational
  problems to be analyzed and the main results of the paper. In
  particular, the basic existence and qualitative properties of the
  domain wall profiles for the limit thin film problem are presented
  in Theorem \ref{t:exist}, a complete characterization of all domain
  wall profiles of the limit problem is given in Theorem
  \ref{th:uniqueness}, and convergence of the minimizers in the
  regimes of large and small values of $\gamma$ for $h = 0$ is
  presented in Theorem \ref{th:glargesmall}. Finally, a
  characterization and the asymptotic behavior of minimizers of the
  reduced thin film energy as the film thickness vanishes is presented
  in Theorem \ref{th:micromin}. In Sec. \ref{sec:thin}, we present the
  treatment of the limit thin film energy, in which the existence
  result for the minimizers is given by Theorem \ref{th:k=2} and the
  rest of the section is devoted to the proofs of Theorem
  \ref{t:exist}, Theorem \ref{th:uniqueness} and Theorem
  \ref{th:glargesmall}.  We also characterize the infimum energy for
  the limit thin film energy in the classes of configurations with
  prescribed winding in Corollary \ref{cor:infk0}. Finally, in
  Sec. \ref{sec:mag} we prove a $\Gamma$-convergence result for the
  reduced micromagnetic thin film energy to the limit energy analyzed
  in Sec. \ref{sec:thin} in Theorem \ref{th:gamma}, and then establish
  Theorem \ref{th:micromin} via a sequence of corollaries.

  \paragraph{\bf Acknowledgements} The work of CBM was supported, in
  part, by NSF via grants DMS-1614948 and DMS-1908709.  MN was
  supported by the PRIN Project 2019/24. MM and MN are members of the
  INDAM/GNAMPA.  VS acknowledges support by Leverhulme grant
  RPG-2018-438 and would like to thank the Max Planck Institute for
  Mathematics in the Sciences in Leipzig for support and hospitality.

\section{Statement of results}
\label{sec:statement-results}

We now turn to the precise statements of the main results of our
paper. We begin by simplifying some of the notation. For the limit
thin film energy, we drop the subscript ``0'' from the definition of
the two-dimensional strip domain and simply write
$\Sigma := \R \times (0,1)\subset \R^2$. By
$\mathbf r = (x, y) \in \Sigma$ we denote a generic point in the
strip, with $x \in \R$ and $y \in (0,1)$. On the strip $\Sigma$ we
introduce a local space $H^1_{l}(\Sigma)$ consisting of functions
whose restrictions to truncated strips $Q_R := (-R, R) \times (0,1)$
belong to $H^1(Q_R)$ for any $R > 0$. We equip $H^1_{l}(\Sigma)$ with
the notion of convergence corresponding to the $H^1(Q_R)$ convergence
of the restrictions to $Q_R$. This space plays the role of the local
space $H^1_{loc}(\Sigma)$ that allows to make sense of the traces of
functions on $\partial \Sigma$ in the $L^2_{loc}(\partial \Sigma)$
sense.

For $h \geq 0$, $\gamma > 0$ and $\theta \in H^1_{l}(\Sigma)$ the thin
film limit energy
\begin{align}
  \label{eq:25}
  F(\theta):=\int_\Sigma\left(\frac12|\nabla
  \theta|^2+h(1-\cos\theta)\right)\, d^2 r +\gamma\int_{\pa
  \Sigma}\sin^2\theta\, d\H^1
\end{align}
defines a map $F: H^1_{l}(\Sigma) \to [0, +\infty]$, provided the last
term in \eqref{eq:25} is understood in the sense of trace. Notice that
the Euler-Lagrange equation associated with \eqref{eq:25} is
\beq\label{limitEL2}
\begin{cases}
  \Delta\theta=h\sin\theta & \text{in }\Sigma\,,\\
  \pa_\nu\theta=-\gamma \sin(2\theta)& \text{on }\pa \Sigma\,,
\end{cases}
\eeq where $\partial_\nu \theta$ denotes the derivative of $\theta$ in
the direction of the outward normal $\nu$ to $\partial \Sigma$. The
weak form of \eqref{limitEL2} is \beq\label{wEL2} \int_{\Sigma}(\nabla
\theta\cdot \nabla \varphi+ h\sin(\theta)\, \varphi)\, d^2
r+\gamma\int_{\pa\Sigma}\sin(2\theta) \, \varphi\, d\mathcal H^1=0
\qquad \forall \varphi\in H^1_l({\Sigma}) \text{ with bounded
  support}.  \eeq
\begin{remark}\label{rm:wvss}
  By Lemma~\ref{lm:elem} below, any {bounded} weak solution to
  \eqref{limitEL2}, i.e., any
  $\theta\in H^1_l(\Sigma) {\cap L^\infty(\Sigma)}$ satisfying
  \eqref{wEL2} belongs to $C^\infty(\overline\Sigma)$ and thus is a
  classical solution{ of \eqref{limitEL2}}. Therefore, throughout
  the paper we will not distinguish between weak and
  strong formulations of the problem.
\end{remark}
Next, for $k \in \mathbb Z$ we introduce a class of functions 
\begin{multline}
  \mathcal{A}_k:=\Bigl\{\theta\in H^1_{l}(\Sigma) :\,
   \lim_{x\to+\infty}\|\theta(x,
  \cdot)\|_{L^2(0,1)}=0, \lim_{x\to-\infty}\|\theta(x,
  \cdot)-k\pi\|_{L^2(0,1)}=0 \Bigr\}\,,
\end{multline}
{where $\theta(x, \cdot)$ is understood as a
  trace. These functions} correspond to the in-plane magnetization
profiles $m = (\cos \theta, \sin \theta)$ connecting
$\theta(x, y) = 0$ at $x = +\infty$ with $\theta(x, y)= k \pi$ at
$x = -\infty$ in an average sense. For the limit energy $F$, we are
then interested in the following variational problem: \beq\label{mink}
\text{minimize} \ F(\theta)\ \text{over} \ \theta \in \mathcal A_k \
\text{with} \ k \not=0 \ \text{fixed}. \eeq

\begin{remark}\label{rm:mink}
Note that if $\theta\in \mathcal{A}_k$, then $-\theta\in
\mathcal{A}_{-k}$ with  
$F(\theta)=F(-\theta)$. In particular, for every $k\in \N$ we have
$$
\inf_{\theta\in \mathcal{A}_k}F(\theta)= 
\inf_{\theta\in \mathcal{A}_{-k}}F(\theta)\,.
$$
\end{remark}

\noindent In view of the previous remark, we may restrict ourselves to
the case $k\in \N$ in \eqref{mink}.

Our first result concerns existence, uniqueness and qualitative
properties of the minimizers of $F$ in $\mathcal A_k$.

\begin{theorem}
  \label{t:exist}
  Let $\gamma > 0$, $h \geq 0$ and $k \in \N$. Then a minimizer
  $\theta_{min}$ of $F$ over $\mathcal A_k$ exists if and only if
  $k = 1$ for $h = 0$, or if and only if $k = 2$ for $h > 0$. The
  minimizer is unique up to translations along the $x$ direction,
  belongs to $C^\infty(\overline\Sigma)$ with derivatives of all
  orders bounded and satisfies \eqref{limitEL2} classically.  In
  addition, for all $(x, y) \in \overline \Sigma$ the minimizer
  $\theta_{min}$ satisfies:
  \begin{itemize}    
  \item [a)] (strict monotone decrease)
    $\partial_x \theta_{min}(x, y) < 0$;
  \item [b)] (symmetry) $\theta_{min}(x,y)=\theta_{min}(x,1-y)$ and
    $\theta_{min}(x,y)=k \pi-\theta_{min}(a-x,y)$ for some $a \in \R$;
  \item [b)] (exponential decay at infinity) for every $m\in \N$ there
    exist positive constants $\alpha_m$, $\beta_m$ such that
    $$
    \|\theta_{min}-k \pi\|_{C^m((-\infty, -t] \times [0,1])}\leq
    \alpha_m \mathrm{e}^{-\beta_m t}\quad\text{and}\quad
    \|\theta_{min}\|_{C^m([t, +\infty) \times [0,1])}\leq \alpha_m
    \mathrm{e}^{-\beta_m t}
    $$ 
    for all $t>0$ sufficiently large.
  \end{itemize}
\end{theorem}

Our next result characterizes all domain wall type solutions {for
  the limit thin film model}, i.e., {all} bounded solutions of
\eqref{limitEL2} that attain distinct pointwise limits as
$x \to \pm \infty$. More precisely, we introduce the following
definition.

 \begin{definition}\label{def:dws}
   Let
   $\theta\in C^2(\Sigma)\cap C^1(\overline\Sigma) \cap
   L^\infty(\Sigma)$ be a solution of \eqref{limitEL2}.  We say that
   $\theta$ is a \emph{domain wall solution} if there exist
   $\ell^-, \ell^+\in \R$, $\ell^->\ell^+$, such that
   \beq\label{binfty} \lim_{x\to-\infty}\theta(x,
   y)=\ell^-\quad\text{and}\quad \lim_{x\to+\infty}\theta(x, y)=\ell^+
   \qquad\text{ for all } y\in (0,1)\,.  \eeq
 \end{definition}

 \begin{remark}\label{rm:dws}
   We make several observation regarding the above definition:
 \begin{itemize}
 \item [a)] The condition $\ell^->\ell^+$ is assumed without loss of
   generality, as otherwise we can replace $\theta(x, y)$ with
   $\theta(-x, y)$ in all the statements.
 \item[b)] If $\theta$ is a domain wall solution in the sense of
   Definition \ref{def:dws} and $k$ any integer, then $\theta+2k\pi$
   is a domain wall solution as well. If additionally $h=0$, so is
   also $\theta+k\pi$.
 \item[c)] By Lemma~\ref{lm:elem}, any bounded weak solution
   \eqref{wEL2} is smooth up to the boundary with derivatives of all
   orders bounded, and, therefore, it solves \eqref{limitEL2}
   classically. In particular, domain wall solutions in the sense of
   Definition~\ref{def:dws} belong to $C^\infty(\overline \Sigma)$, and
   their derivatives of all orders are bounded. Moreover, by the same
   lemma, the convergence to $\ell^\pm$ in \eqref{binfty} holds in
   fact in a much stronger sense, namely uniformly with respect to the
   $C^m$-norm, for every $m\in \N$, see \eqref{inftypm}.
 \item[d)] If $h=0$, or if $h>0$ and $F(\theta)<+\infty$, then
   condition \eqref{binfty} can be replaced (see Lemma~\ref{lm:equiv})
   by the following  one: \beq\label{wbinfty}
   \lim_{x\to-\infty}\theta(x, 0)= \lim_{x\to-\infty}\theta(x,
   1)=\ell^-\quad\text{and}\quad \lim_{x\to+\infty}\theta(x, 0)=
   \lim_{x\to+\infty}\theta(x, 1) =\ell^+\,. \eeq 
 \end{itemize}
\end{remark}

\noindent {We also note that in view of Remark \ref{rm:dws}-c) the
  functions $\theta(x, y) = \ell^\pm$ must themselves solve
  \eqref{limitEL2}. Hence, a priori we should have
  $\ell^\pm \in {\pi \over 2} \mathbb Z$ when $h = 0$ and
  $\ell^\pm \in \pi \mathbb Z$ when $h > 0$. }

 We now state the theorem about domain wall type solutions. In
 essence, our next result shows that the only domain wall type
 critical points of $F$ are the minimizers obtained in Theorem
 \ref{t:exist}, up to a reflection and an addition of a multiple of
 $\pi$.

\begin{theorem}\label{th:uniqueness}
  Let {$\gamma > 0$ and $h \geq 0$, let} $\theta$ be a domain wall
  solution in the sense of Definition~\ref{def:dws}, and let
  $\theta_{min}$ be as in Theorem~\ref{t:exist}. Then the following
  uniqueness properties hold true:
  \begin{itemize}
  \item[a)] If $h=0$, then there exist $k\in \Z$ and $\lambda\in \R$
    such that $\ell^+=k\pi$, $\ell^-=(k+1)\pi$, and {for every $(x,
      y) \in \overline \Sigma$}
    $$
    \theta{(x, y)}=\theta_{min}({x + \lambda, y})+k\pi\,;
    $$
  \item [b)] If $h>0$, then there exist $k\in \Z$ and $\lambda\in \R$
    such that $\ell^+=2k\pi$, $\ell^-=(2k+2)\pi$, and {for every $(x,
      y) \in \overline \Sigma$} 
    $$
    \theta{(x, y)} =\theta_{min}({x + \lambda, y})+2k\pi\,.
    $$ 
  \end{itemize}
\end{theorem}

Before turning to the relation between the thin limit model in
\eqref{eq:25} and the micromagnetic energy, we also consider the
asymptotic behavior of the domain wall solutions for both
$\gamma \ll 1$ and $\gamma \gg 1$. In view of Theorem
\ref{th:uniqueness}, it is sufficient to consider the minimizers of
$F$ in the appropriate function classes. For simplicity of
presentation, we will only consider the most interesting case $h = 0$,
as the case $h > 0$ may be treated analogously, albeit without an
explicit limiting solution when $\gamma \to \infty$.

\begin{theorem}
  \label{th:glargesmall}
  For $\gamma > 0$ and $h = 0$, let {$\theta_{min,\gamma}$} be the
  unique minimizer of $F$ over $\mathcal A_1$ satisfying
  ${\theta_{min,\gamma}}(0,\cdot) = {\pi \over 2}$. Then
  \begin{itemize}
  \item [a)]
    ${\theta_{min,\gamma}}(x / \sqrt{\gamma},y) \to \pi - 2 \arctan (e^{2x})$
    as $\gamma \to 0$; 

  \item [b)]
    ${\theta_{min,\gamma}}(x, y) \to {\pi \over 2} - \arctan \left(
      {\sinh (\pi x) \over \sin (\pi y)} \right)$ as
    $\gamma \to \infty$,
  \end{itemize}
  { locally uniformly in $\Sigma$.}
\end{theorem}

\begin{remark}
  {As may be seen from the proof, the result in part a) of Theorem
    \ref{th:glargesmall} also holds with respect to the
    $H^1_{l}(\Sigma)$ convergence. However, the latter does not hold
    for part b), as the limit function fails to be in
    $H^1_{l}(\Sigma)$. Finally, in the case $h > 0$ and $\gamma \to 0$
    the limit solution is easily seen to be that of \eqref{limitEL2}
    with $\gamma = 0$ and is, once again, one-dimensional, while as
    $\gamma \to \infty$ the solution is expected to converge to a
    solution of the first equation in \eqref{limitEL2} with Dirichlet
    boundary condition in the form of a piecewise-constant function
    taking values $0$ and $2 \pi$.}
\end{remark}

Notice that the result in part b) of Theorem \ref{th:glargesmall}
provides a rigorous basis for the physical picture presented in
\cite{tchernyshyov05}. Also, Theorem \ref{th:glargesmall} provides a
rigorous counterpart for the discussion in the introduction regarding
the limiting behavior of the magnetization in the strip in the limits
of large and small values of $\gamma$.

We finally turn to the relationship of the results obtained by us for
the limit thin film energy in \eqref{eq:25} with those for the
micromagnetic energy. Notice that in the regime of interest the film
thickness reaches an order of only a few atomic layers, making the use
of the full three-dimensional micromagnetic energy problematic. As was
argued previously, a model that is more appropriate for such {\em
  ultrathin} films is the reduced micromagnetic thin film energy (for
a detailed discussion, see \cite{lms:jns20,dms21}).

Let
{$d_{\Sigma}(\mathbf r) := \mathrm{dist}(\mathbf r, \R^2 \backslash
  \Sigma)$.}  For $\e > 0$ sufficiently, small we consider the family
of cutoff functions
\begin{equation}
  \eta_{\e} (\mathbf r) = \eta \left( d_{\Sigma} (\mathbf r) / \e
  \right), 
  \label{mdelta}
\end{equation}
where {$\eta \in C^1([0, +\infty))$ is such that $\eta (0) = 0$,
  $\eta'(t) \geq 0$ for all $t \geq 0$ and $\eta (t) = 1$ for
  $t \geq 1$.}
{Then for \beq m: \Sigma \to \mathbb S^1, \quad m = (m_1 (x,y),
  m_2(x,y)), \eeq such that $m \in C^\infty(\overline\Sigma; \R^2)$
  and $m_2$ vanishes outside a compact set,} we define the following
reduced micromagnetic energy: \beq\label{epsilone} E_\e(m) =
\frac{1}{2} \int_{\Sigma} |\nabla m|^2\, d^2r + \frac{\gamma}{2 |\ln
  \e|} \int_{\Sigma} \int_{\Sigma} \frac{\Div (\eta_\e m)
  (\mathbf r) \, \Div (\eta_\e m) (\mathbf r')}{| \mathbf r - \mathbf
  r'|} \, d^2r \, d^2 r' + h \int_{\Sigma} (1- m_1)\, d^2r, \eeq where
$\gamma>0$ is a fixed parameter, which may be obtained from the
  full three-dimensional micromagnetics via a formal asymptotic
  reduction and a suitable rescaling of the strip width
  \cite{lms:jns20,dms21}.  The conditions on $m$, which we are going
  to relax shortly, are needed to ensure convergence of all the
  integrals in \eqref{epsilone}. In particular, it ensures that
  $m_1(x, y) = \pm 1$ for all $|x|$ large enough, corresponding to the
  head-to-head or winding domain wall configurations.

  In \eqref{epsilone}, the parameter $\eps$ represents the effective
dimensionless film thickness measured relative to the strip width, and
$\gamma$ is an effective stray field strength normalized by
$|\ln \eps|$ (compare with \eqref{eq:6}). {As was already
  mentioned,} this energy is somewhat intermediate in the hierarchy of
multiscale micromagnetic energies between the full three-dimensional
micromagnetic energy in \eqref{eq:1} (with the Zeeman term added) and
the limit thin film energy in \eqref{eq:25}.

{The assumptions about $m$ above are clearly too restrictive for
  the existence of unconstrained minimizers of $E_\e$. To find a more
  appropriate functional setting to seek the energy minimizers in the
  form of head-to-head or winding domain walls, we pass to the Fourier
  space in the nonlocal term and introduce the transform
  $\mathscr F(\mathrm{div}(\eta_\eps m))$ of
  $\mathrm{div}(\eta_\eps m) \in C^\infty_c(\R^2)$:
\begin{align}
  \label{eq:38}
  \mathscr F(\mathrm{div}(\eta_\eps m)) (k_1, k_2) = \int_0^1 \int_\R
  e^{-i k_1 x - i k_2 y} \mathrm{div} (\eta_\eps(y) m(x, y)) \, dx \, dy,
\end{align}
where $\mathrm{div}(\eta_\eps m)$ was extended by zero outside
$\Sigma$.  Clearly, under our assumption we have {\cite[Theorem
  5.9]{lieb-loss}}
\begin{align}
  \label{eq:39}
  \int_\Sigma \int_\Sigma  \frac{\Div (\eta_\e m)
  (\mathbf r) \, \Div (\eta_\e m) (\mathbf r')}{2 \pi | \mathbf r -
  \mathbf r'|} \, d^2r \, d^2 r' = \int_{\R^2} {|\mathscr
  F(\mathrm{div}(\eta_\eps m)) |^2 \over |\mathbf k|} {d^2 k \over (2
  \pi)^2},
\end{align}
which is nothing but the $\mathring{H}^{-1/2}(\R^2)$ norm squared of
$\Div (\eta_\e m)$.  Thus, under the above assumptions about $m$ the
energy $E_\e(m)$ may be alternatively written in the form
\beq\label{epsiloneF} E_\e(m) = \frac{1}{2} \int_{\Sigma} |\nabla
m|^2\, d^2r + \frac{\gamma}{2 |\ln \e|} \int_{\R^2} {|\mathscr
  F(\Div(\eta_\eps m)) |^2 \over 2 \pi |\mathbf k|} \, d^2 k + h
\int_{\Sigma} (1- m_1)\, d^2r. \eeq

We now wish to relax the assumptions of smoothness of $m$ and of $m_2$
having compact support and introduce a more natural class of
magnetizations for which the energy in \eqref{epsiloneF} remains
valid, taking advantage of positivity of the nonlocal energy term
written in the Fourier space. Clearly, for $m \in H^1_{l}(\Sigma)$ all
the local terms in the energy are well defined (possibly taking the
value $+\infty$). It remains to make sense of the nonlocal term. For
that purpose, observe that for $m \in H^1_{l}(\Sigma)$ we have
{(with a slight abuse of notation)} 
\begin{align}
  \label{eq:40}
  \Div(\eta_\eps m)(x, y) = \eta_\eps(y) \partial_x m_1(x, y) + \eta_\eps(y)
  \partial_y m_2(x, y) + \eta_\eps'(y) m_2(x, y)
\end{align}
distributionally. Therefore, under a natural condition that
$\nabla m \in L^2(\Sigma; \R^2)$ the first two terms in the right-hand
side of \eqref{eq:40}, extended by zero outside $\Sigma$, belong to
$L^2(\R^2)$ and thus have a well-defined Fourier transform in the
$L^2$-sense. To make sense of the third term, we additionally assume
that $m_2 \in L^2(\Sigma)$. Thus, we introduce the class
\begin{align}
  \label{eq:41}
  \Sp := \left\{ m \in H^1_{l}(\Sigma; \mathbb S^1) \, : \, \nabla m
  \in L^2(\Sigma; \R^2), \, m_2 \in L^2(\Sigma) \right\}, 
\end{align}
on which $E_\e : \Sp \to [0, +\infty]$ is now well defined for all
$\e \in (0, \frac12)$. Note that the assumption $m_2 \in L^2(\Sigma)$
for all $m \in \Sp$ forces $m_1(x, \cdot)$ to approach $\pm 1$ in some
average sense as $x \to \pm \infty$, thus selecting the magnetization
profiles in the form of head-to-head or winding walls. }

We will show the $\Gamma$-convergence {as $\e \to 0$ of the energy
  $E_\e$ defined on $\Sp$} to the following reduced energy (see
Sec. \ref{sec:mag}): \beq\label{limitingE} E_0(m) = \frac{1}{2}
\int_{\Sigma} |\nabla m|^2\, d^2 r + h \int_{\Sigma} (1- m_1)\, d^2 r
+ \gamma \int_{\partial \Sigma} m^2_2\, d\mathcal H^1.  \eeq {With
  a slight abuse of notation, when talking about the limit $\e \to 0$
  we will always imply taking a sequence of $\e_k \to 0$ as
  $k \to \infty$.}

Associated with the energy in \eqref{limitingE}, we have the following
minimization problem:
\begin{multline}\label{mineps0}
  \text{minimize} \ E_0(m) \ \text{among} \\ \ m=(\cos\theta,
  \sin\theta)\in H^1_{l}(\Sigma; \mathbb S^1) \text{ with }\theta
  \text{ satisfying \eqref{lemma1} for some }k_1,\, k_2\in \Z,\,
  k_1\neq k_2.
\end{multline}
Notice that for $m = (\cos \theta, \sin \theta)$, this energy
coincides precisely with that in \eqref{eq:25}, and such a lifting is
always possible for any $m \in H^1_{l}(\Sigma; \mathbb S^1)$ (see, for
instance, \cite{bourgain05}), making the energies $E_0(m)$ and
$F(\theta)$ equivalent. The $\Gamma$-convergence result, stated in
Theorem \ref{th:gamma}, is with respect to the strong
$L^2_{loc}(\Sigma)$ convergence of maps
$m_\e : \Sigma \to \mathbb S^1$. Using this $\Gamma$-convergence
result, we can then establish {existence and a} characterization of
the minimizers of $E_\e$ in the form of domain walls in terms of those
of $E_0$ {for all small enough $\eps$}. Note that the existence and
properties of the latter are established by Theorem
\ref{t:exist}. {Also note that by Theorem \ref{t:exist} the
  minimizers of $E_0$ over $H^1_{l}(\Sigma; \mathbb S^1)$ with
  suitable behaviors at infinity belong to $\Sp$.}

\begin{theorem}
  \label{th:micromin}
  Let $\gamma > 0$, $h \geq 0$ and $k \in \mathbb N$. Then there
  exists $\eps_0 > 0$ such that for all $\eps \in (0, \eps_0)$ there
  exists a minimizer $m = (\cos \theta, \sin \theta)$ of $E_\eps$ over
  all $m \in \Sp$ {with} $\theta \in \mathcal A_k$ if and only if
  $k = 1$ when $h = 0$, or if and only if $k = 2$ when $h > 0$. As
  $\eps \to 0$, every minimizer of $E_\eps$ above converges {in
    $H^1_{l}(\Sigma; \R^2)$}, after a suitable translation, to
  {the} corresponding minimizer of $E_0$.
\end{theorem}

{The above result shows that, in the considered regime of ultrathin
  ferromagnetic films, the domain wall-like ground states of the
  micromagnetic energy are head-to-head walls with no winding
  ($180^\circ$ walls) in the absence of the applied field
  ($h=0$). When an applied field is present ($h>0$), the only domain
  wall-like ground states are winding domain walls with a single
  rotation ($360^\circ$ walls). Furthermore, as the film thickness
  tends to zero these ground state profiles converge to the uniquely
  defined energy minimizing profiles for the limit energy $E_0$ (up to
  translations). Thus, in particular our results provide a
  mathematical understanding for the symmetric head-to-head domain
  wall profiles in the absence of the applied field observed in
  experiments and numerical simulations of sufficiently thin
  nanostrips (see Fig. \ref{f:dw}(b)) and the discussion in
  Sec. \ref{sec:introduction}. At the same time, our analysis does not
  capture the asymmetric head-to-head walls observed in wider
  nanostrips (see Fig. \ref{f:dw}(c)). The analysis of the latter
  would require to consider a regime in which the stray field effect
  does not reduce to a purely local penalty term at the sample
  boundary, and is outside the regime studied in this
  paper. Similarly, our regime excludes the appearance of the vortex walls
  shown in Fig. \ref{f:dw}(a).  }

\section{Analysis of the thin film limit model} \label{sec:thin} 

We start by recalling that for every
$m\in H^1_{l}(\Sigma; \mathbb S^1)$ there exists
$\theta\in H^1_{l}(\Sigma)$ such that $m=(\cos\theta, \sin\theta)$
(see, for instance, \cite{bourgain05}), and the energy
\eqref{limitingE} may be rewritten as \beq\label{limitingtheta}
E_0(m)=F(\theta).  \eeq In what follows, we identify any
$\theta\in H^1_{l}(\Sigma)$ with the precise representative such that
for every $x_0\in \R$, $\theta(x_0, \cdot)$ coincides a.e. with the
trace of $\theta$ on the {vertical} line $ x=x_0$.
\begin{lemma}\label{lm:zagara}
  Let $\theta\in H^1_{l}(\Sigma)$ be such that
  $F(\theta)<+\infty$. Then there exist $k_1, k_2\in \Z$ such that
  \beq\label{lemma1}
  \lim_{x\to-\infty}\|\theta(x,\cdot)-k_1\pi\|_{L^{2}(0,1)}=0 \text{
    and }
  \lim_{x\to+\infty}\|\theta(x,\cdot)-k_2\pi\|_{L^{2}(0,1)}=0\,.  \eeq
  Furthermore, if $h>0$ we have $k_1, k_2\in 2\Z$.
\end{lemma}
\begin{proof}
Set 
$$
\bar\theta(x):=\int_0^1\theta(x, y)\, dy\,.
$$
and note that $\bar\theta\in H^1_{loc}(\R)$ and thus, in particular,
it is continuous.  We claim that \beq\label{claim}
\frac12\int_{\R}|\bar\theta'(x)|^2\,
dx+2\gamma\int_{\R}\sin^2\bar\theta(x)\, dx\leq (5+8\gamma)
F(\theta)\,.  \eeq

We start by observing that
\begin{align*}
  2\gamma\int_{\R}\sin^2\bar\theta(x)\, dx
  &=
    2\gamma\int_{\R}\sin^2\theta(x,0)\,
    dx+2\gamma
    \int_{\R}(\sin^2\bar\theta(x)-
    \sin^2\theta(x,0))\, dx\\ 
  &\leq  4\gamma\int_{\R}\sin^2\theta(x,0)\, dx+
    4\gamma\int_{\R}|\sin\bar\theta(x)- \sin\theta(x,0)|^2\, dx\\ 
  &\leq   4\gamma\int_{\R}\sin^2\theta(x,0)\, dx+
    4\gamma\int_{\R}|\bar\theta(x)- \theta(x,0)|^2\, dx\\ 
  &\leq   4\gamma\int_{\R}\sin^2\theta(x,0)\, dx+
    4\gamma\int_{\R}\int_0^1|\partial_y \theta(x, y)|^2\, dy\, dx\leq
    (4+8\gamma) F(\theta)\,. 
\end{align*}
{Equation \eqref{claim} then follows.}

Note that for every $\alpha<\beta$ we have \beq\label{mmtrick}
\frac12\int_{\alpha}^\beta|\bar\theta'|^2\,
dx+2\gamma\int_{\alpha}^\beta\sin^2\bar\theta\, dx\geq
2\sqrt{\gamma}\int_{\alpha}^\beta |\sin\bar\theta||\bar\theta'|\,
dx\geq
2\sqrt{\gamma}\big|\cos(\bar\theta(\beta))-\cos(\bar\theta(\alpha))\big|\,.
\eeq In particular, recalling \eqref{claim}, $\cos\bar\theta$
satisfies the Cauchy condition for $x\to+\infty$, that is
$$
\lim_{\alpha,
  \beta\to+\infty}\big|\cos(\bar\theta(\beta))-\cos(\bar\theta(\alpha))\big|=0\,, 
$$
and thus $\cos\bar\theta$, and in turn $\sin^2\bar\theta$, admit a
limit as $x\to+\infty$. Clearly the same is true for
$x\to-\infty$.

Recalling \eqref{claim}, we conclude that \beq\label{idonno}
\sin\bar\theta(x)\to 0 \text{ and } \cos^2\bar\theta(x)\to1 \text{ as
} |x|\to+\infty\,.  \eeq We now claim that there exist
$k_1, k_2\in \Z$ such that \beq\label{limbartheta}
\lim_{x\to-\infty}\bar\theta(x)=k_1\pi \text{ and }
\lim_{x\to+\infty}\bar \theta(x)=k_2\pi\,.  \eeq Let us show only the
second limit. We argue by contradiction assuming that there exist two
sequences $x_n<x'_n$ both diverging to $+\infty$ such that
$\liminf_{n \to \infty} |\bar\theta(x_n)- \bar\theta(x'_n)|\geq\pi$.
But then, by the continuity of $\bar\theta$ it is clear that we may
also find $x''_n\in (x_n, x'_n)$ such that
$ \cos^2\bar\theta(x''_n)\to 0$, which contradicts
\eqref{idonno}. Thus, \eqref{limbartheta} holds.

Denote $Q^t:=(t-\frac12,t+\frac12)\times(0,1)$ and note that
$\lim_{t\to\pm \infty}\|\nabla \theta\|_{L^2(Q^t)}=0$. In turn, by a
Poincar\'e-type inequality we have
$$
\|\theta- \bar\theta(t)\|^2_{H^1(Q^t)}\leq C \|\nabla
\theta\|^2_{L^2(Q^t)}
$$
and thus, taking into account \eqref{limbartheta} we conclude that
$$
\lim_{t\to+\infty}\|\theta-k_2\pi \|^2_{H^1(Q^t)}=0 \text{ and }
\lim_{t\to-\infty}\|\theta-k_1\pi \|^2_{H^1(Q^t)}=0\,.
$$
By an application of the Trace Theorem we obtain \eqref{lemma1}. If
$h>0$, then the fact that
  $$
  \int_{\Sigma}(1-\cos\theta)\, dx<+\infty
  $$
  implies that $k_1$, $k_2\in 2\Z$.
\end{proof}

Note that given $m\in H^1_{l}(\Sigma; \mathbb S^1)$, the corresponding
phase function $\theta$ is determined up to an additive constant of
the form $k\pi$, where $k\in \Z$ if $h=0$ or $k\in2\Z$ if $h>0$. In
view of Lemma~\ref{lm:zagara} we may additionally require that
\beq\label{thetau}
\lim_{x\to+\infty}\|\theta(x,\cdot)\|_{L^2(0,1)}=0\,.  \eeq Clearly by
enforcing such a condition the phase function $\theta$ is uniquely
determined.

In the next two subsections we address the existence of minimizers and
the classification of domain wall solutions in the sense of
Definition~\ref{def:dws}, respectively.

\subsection{Existence of minimizers}

We prove the following existence result.

\begin{theorem}\label{th:k=2}
  If $h=0$ then the minimization problem \eqref{mink} admits a
  solution for $k=1$. If $h>0$ then the minimization problem
  \eqref{mink} admits a solution for $k=2$. In both cases, a solution
  $\theta_{min}$ can be found satisfying
  $\displaystyle\int_0^1\theta_{min}(0,y)\,dy=\frac{k\pi}2$. Moreover,
  $\theta_{min}\in C^{\infty}(\overline\Sigma)$, with derivatives of
  all order bounded, and $\partial_x \theta_{min} <0$ in
  $\overline{\Sigma}$.
\end{theorem}
\begin{proof} 
  We provide the proof only in the case $h>0$, as the case $h = 0$ can
  be treated analogously and is simpler.
 To this end, for $M>0$ let
 $$
 \mathcal{A}_{2,M}:=\big\{\theta\in \mathcal{A}_2:\, \theta=0 \text{
   in }\{(x,y)\in \Sigma:\, x\geq M\}\text{ and } \theta=2\pi \text{
   in }\{(x,y)\in \Sigma:\, x\leq -M\}\big\},
 $$ 
 and note that by standard arguments there exists a minimizer
 $\theta_M$ of $F$ over $\mathcal{A}_{2,M}$.  Throughout the proof for
 every $M>0$ we set $\mathcal R_M:=(-M, M)\times (0,1)$.

 We claim that \beq\label{claim1} \theta_M(x, y) \in
 (0,2\pi)\quad\text{ for all } (x, y) \in \mathcal R_M\,.  \eeq This
 follows by first observing that by an easy truncation procedure we
 may conclude that $\theta_M$ satisfies \beq\label{claim1.1} 0\leq
 \theta_M\leq 2\pi\,.  \eeq Moreover, by a standard first variation
 argument $\theta_M$ is a weak solution to the following
 Euler-Lagrange problem \beq\label{EL0}
\begin{cases}
  \Delta \theta_M=h\sin\theta_M & \text{in }  \mathcal R_M\,,\\
  \pa_{\nu}\theta_M=-\gamma \sin(2\theta_M) & \text{on } \pa \mathcal
  R_M\cap \pa \Sigma \,,\\ 
  \theta_M=0      &\text{on }  \{M\}\times (0,1)\,,\\
  \theta_M=2\pi &\text{on } \{-M\}\times (0,1)\,,
\end{cases}
\eeq that is, \beq\label{wthetaM} \int_{ \mathcal R_M}(\nabla
\theta_M\cdot \nabla \varphi+ h\sin(\theta_M)\, \varphi)\, d^2
r+\gamma\int_{ \pa \mathcal R_M\cap \pa
  \Sigma}\sin(2\theta_M)\varphi\, d\mathcal H^1=0 \eeq for all {
  $\varphi\in H^1(\mathcal R_M)$ s.t. $\varphi=0$ on
  $\{-M, M\}\times (0,1)$. }

Consider now the reflected function $\widetilde \theta_M$ defined on
$\mathcal R_{3M}$ by
$$
\widetilde \theta_M(x,y):=
\begin{cases}
  -\theta_M(-x-2M,y)+4\pi & \text{if }x{\in (-3M, -M)},\\
  \theta_M(x,y) & \text{if } x\in (-M,M)\,,\\
  -\theta_M(-x+2M,y) & \text{if }x{ \in (M, 3 M)} \,.
\end{cases}
$$
Using the weak formulation \eqref{wthetaM}, one can immediately check
that $\widetilde \theta_M$ is in turn a weak solution; that is,
$$
\int_{ \mathcal R_{3M}}(\nabla \widetilde \theta_M\cdot \nabla
\varphi+ h\sin(\widetilde \theta_M)\, \varphi)\, d^2 r+\gamma\int_{
  \pa \mathcal R_{3M}\cap \pa \Sigma}\sin(2\widetilde
\theta_M)\varphi\, d\mathcal H^1=0
$$
for all $\varphi\in H^1(\mathcal R_{3M})$ s.t. $\varphi=0$ on
$\{-3M, 3M\}\times (0,1)$.  We may then apply the very same arguments
of Lemma~\ref{lm:elem}-a) below (clearly, we can, since the regularity
argument is local) to conclude that for every $0<M'<3M$,
$\widetilde \theta_M\in C^{\infty}(\overline{\mathcal R_{M'}})$. In
particular, $\theta_M\in C^{\infty}(\overline{\mathcal R_{M}})$,  and
\eqref{EL0} holds classically.

Note that we can write 
 $$
 \begin{cases}
 \Delta \theta_M= c(x,y)\theta_M & \text{in $(-M,M)\times (0,1)$,}\\
 \pa_\nu \theta_M=-\gamma \sin(2\theta_M) &\text{on }(-M, M)\times \{0,1\},
 \end{cases}
 $$
 where we set 
 $$c(x,y):=
 \begin{cases}
 h\frac{\sin\theta_M(x,y)}{\theta_M(x,y)} & \text{if } \theta_M(x,y)>0\,, \\
 h & \text{if }\theta_M(x,y)=0\,.
 \end{cases}
 $$
 In order to prove \eqref{claim1}, recall \eqref{claim1.1} and assume
 by contradiction that $\theta_M=0$ at some point in
 $(-M,M)\times (0,1)$. But then the Strong Maximum Principle
 \cite[Theorem 2.2]{pucci04} applies and yields that
 $\theta_M\equiv 0$ in $(-M,M)\times (0,1)$, a contradiction to the
 fact that $\theta_M\in \mathcal{A}_{2,M}$. If instead $\theta_M=0$ at
 some point of the {horizontal} boundary $ (-M, M)\times \{0,1\}$, then
 thanks to the Neumann condition in \eqref{EL0} also
 $\pa_\nu \theta_M$ vanishes at the same point and thus the
 contradiction follows from Hopf's Lemma \cite[Lemma
 3.4]{gilbarg}. Hence, we have shown that $ \theta_M>0$ in
 $ (-M, M)\times [0,1]$.  Replacing $\theta_M$ by $2\pi-\theta_M$ and
 arguing as before, we complete the proof of \eqref{claim1}.
 
 We now show that $\theta_M$ is monotone non-increasing in the
 $x$-direction. To this aim, we adapt the classical sliding 
   method of Berestycki and Nirenberg \cite{berestycki91} (see also
   \cite{berestycki92}) to the problem on the strip with nonlinear
   boundary conditions. Set
 $$
 \bar\lambda:=\inf\{\lambda>0:\, \theta_M(\cdot+\mu, \cdot)\leq
 \theta_M\text{ in $\Sigma$  for all }\mu\geq\lambda\},
 $$
 and observe that necessarily $\bar\lambda\in [0, 2M)$. Indeed,
 clearly $\theta_M(\cdot+\mu, \cdot)\leq \theta_M$ for all
 $\mu\geq 2M$. Moreover, since
 $ \theta_M(\cdot+2M, \cdot)=0<2\pi=\theta_M$ on $\{-M\}\times[0,1]$,
 by continuity we may find $\e>0$ so small that
 $ \theta_M(\cdot+2M-s, \cdot)<\theta_M$ on $[-M, -M+s]\times[0,1]$
 for all $s\in (0, \e]$, which in turn easily implies
 $ \theta_M(\cdot+2M-s, \cdot)\leq \theta_M$ for the same $s$. Thus
 $\bar \lambda\leq 2M-\e$.
 
 Note that $\bar \lambda=0$ if and only if $\theta_M$ is monotone
 non-increasing in the $x$-direction.  Assume by contradiction that
 $\bar\lambda>0$. This means that
 $\theta_M(\cdot+\bar \lambda, \cdot)\leq \theta_M$ and we claim that
 there exists $(\bar x, \bar y)\in [-M, M-\bar\lambda]\times [0,1]$
 such that
 $\theta_M(\bar x+\bar \lambda, \bar y)= \theta_M(\bar x, \bar
 y)$. Indeed, if not then we would have
 $\theta_M(\cdot+\bar \lambda, \cdot)< \theta_M$ in
 $[-M, M-\bar\lambda]\times [0,1]$ and in turn, arguing as above,
 $\theta_M(\cdot+\bar \lambda-\e, \cdot)\leq \theta_M$ in
 $[-M, M-\bar\lambda+\e]\times [0,1]$ for all $\e$ small enough,
 contradicting the minimality of $\bar\lambda$. We claim now that
 $\bar x\in ({-M}, M-\bar\lambda)$. Indeed, if $\bar x=-M$, then
 $\theta_M(\bar x+\bar \lambda, \bar y)= \theta_M(\bar x, \bar
 y)=2\pi$ which is impossible thanks to \eqref{claim1} since
 $\bar x+\bar \lambda\in (-M, M)$. If instead $\bar x=M-\bar\lambda$,
 then $ \theta_M(\bar x, \bar y)=\theta_M(M, \bar y)=0$, which is
 again impossible by \eqref{claim1} since $\bar x<M$.

 We now set $u:=\theta_M-\theta_M(\cdot+\bar\lambda, \cdot)$. Note
 that $u$ satisfies \beq\label{u}
  \begin{cases}
    \Delta u=\tilde c u & \text{in }(-M, M-\bar\lambda)\times (0,1)\,,\\
    \pa_\nu u=-\gamma
    (\sin(2\theta_M)-\sin(2\theta_M(\cdot+\bar\lambda, \cdot))) &
    \text{on } (-M, M-\bar\lambda)\times \{0,1\}\,,\\ 
    u\geq 0\,,\\
    u(\bar x, \bar y)=0\,,
  \end{cases}
  \eeq
  where 
  $$\tilde c:=
  \begin{cases}
    h\frac{\sin(\theta_M)-\sin(\theta_M(\cdot+\bar\lambda,
      \cdot))}{\theta_M-\theta_M(\cdot+\bar\lambda, \cdot)} & \text{in
    }\{\theta_M>\theta_M(\cdot+\bar\lambda, \cdot)\}\,,\\ 
    h & \text{in }\{\theta_M=\theta_M(\cdot+\bar\lambda, \cdot)\}\,.
  \end{cases}
  $$ 
  Now if $\bar y\in (0,1)$, then we can invoke again the Strong
  Maximum Principle \cite[Theorem 2.2]{pucci04} to conclude that
  $u\equiv 0$ in $[-M, M-\bar\lambda]\times [0,1]$, and in particular
  that $\theta_M(M-\bar\lambda, y)=\theta_M(M, y)=0$, which is a
  contradiction to \eqref{claim1}.  If instead $\bar y\in \{0,1\}$,
  then by Hopf's Lemma \cite[Lemma 3.4]{gilbarg} we have
  $\pa_\nu u(\bar x, \bar y)\neq0$, which contradicts the
   boundary condition in \eqref{u}. This concludes
  the proof of the fact that $\bar \lambda=0$ and thus that $\theta_M$
  is monotone non-increasing in the $x$-direction.
  
  We now set $\bar\theta_M(x):=\int_0^1\theta_M(x,y)\, dy$. Note that
  $\bar\theta_M$ is continuous on $\R$ and that $\bar\theta_M(x)=0$
  for $x\geq M$ and $\bar\theta_M(x)=2\pi$ for $x\leq-M$. Thus, we may
  find $x_M$ such that $\bar\theta_M(x_M)= \pi$.  We set
  $\tilde \theta_M:=\theta_M(\cdot+x_M, \cdot)$.  Observing that
  $F(\theta_M)$ is non-increasing in $M$, we easily see that
  $\{\tilde\theta_M\}_{M\geq 1}$ is equibounded in $H^1_{l}(\Sigma)$.
  Thus, we may find a {sequence} $M_n\to +\infty$ and a function
  $ \theta_\infty\in H^1_{l}(\Sigma)$ such that
  $\tilde \theta_{M_n}\rightharpoonup \theta_\infty$ weakly in
  $H^1_{l}(\Sigma)$, and \beq\label{semi} F(\theta_\infty)\leq
  \liminf_{n}F(\tilde\theta_{M_n})<+\infty\,.  \eeq Moreover,
  $0\leq\theta_\infty\leq2\pi$, $\theta_\infty$ is monotone
  non-increasing in the $x$-direction, satisfies \beq\label{pi2}
  \int_0^1\theta_\infty(0,y)\, dy=\pi \eeq and \beq\label{limitEL}
   \begin{cases}
     \Delta\theta_\infty=h\sin\theta_\infty & \text{in }\Sigma\,,\\
     \pa_\nu\theta_\infty=-\gamma \sin(2\theta_\infty)& \text{on }\pa
     \Sigma\,.
   \end{cases}
   \eeq in the weak sense.  Again by Lemma~\ref{lm:elem},
   $\theta_\infty\in C^\infty(\overline\Sigma)$, with derivatives of
   all orders bounded, and thus it satisfies \eqref{limitEL}
   classically.
   
   We claim that $\theta_\infty\in \mathcal{A}_2$. To this aim, in
   view of \eqref{semi} and Lemma \ref{lm:zagara}, and recalling that
   $0\leq\theta_\infty\leq2\pi$, we have
  $$ 
  \lim_{x\to-\infty}\|\theta_\infty(x,\cdot)-k_1\pi\|_{L^{2}(0,1)}=0
  \text{ and }
  \lim_{x\to+\infty}\|\theta_\infty(x,\cdot)-k_2\pi\|_{L^{2}(0,1)}=0\,,
  $$
  with $k_1, k_2\in \{0,2\}$. Now, by monotonicity and \eqref{pi2} we
  infer that necessarily $k_1=2$ and $k_2=0$.  This shows that
  $\theta_\infty\in \mathcal{A}_2$.
    
  In order to conclude that $\theta_\infty$ is a minimizer, in view of
  \eqref{semi} it remains to show that \beq\label{remains} \liminf_{n
    \to \infty}F(\tilde\theta_{M_n})=\inf_{\mathcal{A}_2}F\,.  \eeq To
  this aim, it is clearly enough to show that \beq\label{finalclaim}
  \text{for $\theta\in \mathcal{A}_2$ with $F(\theta)<+\infty$ and
    $\e>0$ there exists $M>0$ and $\tilde\theta\in \mathcal{A}_{2,M}$
    such that $F(\tilde\theta)\leq F(\theta)+\e$.}  \eeq In order to
  show this, we select two sequences, $x_n^+\to+\infty$ and
  $x_n^-\to-\infty$, such that \beq\label{1} \theta(x_n^+, \cdot)\to 0
  \quad\text{and}\quad \theta(x_n^-, \cdot)\to
  2\pi\qquad\text{uniformly in }[0,1]\,, \eeq and \beq\label{2}
  \limsup_{n \to \infty} \|\theta(x_n^\pm,
  \cdot)\|_{H^1(0,1)}<+\infty\,.  \eeq This is possible {by a
    simple slicing argument} thanks to the fact that
  $|\nabla\theta|\in L^2(\Sigma)$.  At this point, for every $n\in \N$
  we define 
  $$
  \theta_n(x,y):=
  \begin{cases}
    \theta(x,y) & \text{if }x\in (x_n^-, x_n^+)\,,\\
    \theta(x^+_n,y)\left[\left(1-\frac{x-x_n^+}{\|\theta(x_n^+,
          \cdot)\|_\infty}\right)\lor 0\right] & \text{if }x\geq
    x^+_n\,,\\ 
    2\pi-(2\pi-\theta(x_n^-,
    y))\left[\left(1-\frac{x_n^--x}{\|2\pi-\theta(x_n^-,
          \cdot)\|_\infty}\right)\lor 0\right] & \text{if }x\leq
    x^-_n\,,
  \end{cases}
  $$
  with the understanding that $\theta_n\equiv 0$ for $x\geq x_n$ if
  $\|\theta(x_n^+, \cdot)\|_\infty=0$, and $\theta_n\equiv 2\pi$ for
  $x\leq x_n$ if $\|2\pi-\theta(x_n^-, \cdot)\|_\infty=0$.  Clearly
  each $\theta_n$ belongs to $\mathcal{A}_{2,M_n}$ for some $M_n>0$
  sufficiently large. Moreover, using \eqref{1} and \eqref{2}, it is
  easy to check that $F(\theta_n)-F(\theta)\to 0$ as $n\to\infty$,
  thus establishing \eqref{finalclaim} and finishing the proof of
  existence.
  
  We are left with showing that $\partial_x \theta_\infty <0$ in
  $\overline{\Sigma}$. We already know that $\theta_\infty$ is a
  smooth function, with $\partial_x \theta_\infty \leq 0$
  everywhere. Differentiating \eqref{limitEL} with respect to $x$ we
  obtain \beq\label{limitELLE}
   \begin{cases}
     \Delta (\partial_x \theta_\infty)=h\cos\theta_\infty  \partial_x
     \theta_\infty& \text{in }\Sigma\,,\\ 
     \pa_\nu(\partial_x \theta_\infty)=-2\gamma \cos(2\theta_\infty)
     \partial_x \theta_\infty & \text{on }\pa \Sigma\,.
   \end{cases}
   \eeq Assume $\partial_x \theta_\infty(\bar x, \bar y) =0$ at some
   point $(\bar x, \bar y) \in \overline\Sigma$. If
   $(\bar x, \bar y) \in \Sigma$, then using the Strong Maximum
   Principle \cite[Theorem 2.2]{pucci04} we obtain a contradiction. If
   instead $(\bar x, \bar y) \in \partial\Sigma$, then also
   $\pa_\nu (\partial_x\theta_\infty)$ vanishes at the same point and
   thus the contradiction follows from Hopf's Lemma \cite[Lemma
   3.4]{gilbarg}.
 \end{proof}

 \begin{corollary}\label{cor:infk0}
If $h>0$ then for every $k \in 2\N$ we have 
$$
\inf_{\theta \in \mathcal A_k} F(\theta) = \frac{k}{2} \min_{\theta
  \in \mathcal A_2} F(\theta).
$$
If $h=0$ then for every $k \in \N$ we have 
$$
\inf_{\theta \in \mathcal A_k} F(\theta) = k \min_{\theta \in \mathcal
  A_1} F(\theta).
$$
\end{corollary}
\begin{proof}
  We provide the proof only for the case $h>0$, the other one being
  analogous. As in the previous proof, we fix $M>0$ and let
$$
\mathcal{A}_{k,M}:=\{\theta\in \mathcal{A}_k:\, \theta(x,y)=0 \text{
  if }x\geq M\text{ and } \theta(x,y)=k\pi \text{ if }x\leq -M\}.
 $$ 
 It is clear that there exists a minimizer \beq\label{pr:1}
 {\theta_M = \text{argmin}}_{\theta\in
   \mathcal{A}_{k,M}}F(\theta)\,.  \eeq By the same arguments and with
 the same notation used in the proof of Theorem~\ref{th:k=2} we obtain
\begin{itemize}
\item[i)] $ \theta_M\in (0,k\pi)\quad\text{ in } \mathcal R_M\,$;
\item[ii)] $\theta_M\in C^\infty(\overline{\mathcal R_M})$;
\item[iii)] $\theta_M$ has negative derivative in $x$-direction everywhere in 
  $(-M,M)\times[0,1]$.
\end{itemize}
Arguing as in the proof of Theorem~\ref{th:k=2}, we can show that
\beq\label{corik} \inf_{\theta\in \mathcal{A}_{k}}F(\theta) = \lim_{j
  \to \infty} F(\theta_{M_j}), \eeq where
$\theta_{M_j} \in \mathcal{A}_{k,M_j}$ is a minimizer of the
corresponding problem \eqref{pr:1} and $\{M_j\}$ is any sequence of
positive numbers such that $M_j \to \infty$.

Now observe that by the properties stated above, for every $j\in \N$
we may find smooth functions $g^j_i\in C^{\infty}([0,1])$,
$i=1,\dots, k/2-1$, such that
$$
M_j>g^j_1>g^j_2>\cdots>g^j_{k/2-1}>-M_j \qquad\text{and}\qquad
\theta_{M_j}(g^j_i(y), y)=2 \pi i \quad\text{for all }y\in [0,1]\,.
$$
Setting also $g^j_0:=M_j$, $g^{j}_{k/2}:=-M_j$ and
$\Sigma^j_i:=\{(x,y):\, g^j_{i-1}(y)>x>g^j_{i}(y)\}$, we clearly have
\beq\label{cutting}
\begin{split}
  F(\theta_{M_j})&=\sum_{i=1}^{k/2}
  \left(\int_{\Sigma^j_i}\left(\frac12|\nabla
      \theta_{M_j}|^2+h(1-\cos\theta)\right)\, d^2 r+\gamma\int_{\pa
      \Sigma\cap\pa \Sigma^j_i}\sin^2\theta_{M_j}\, d\H^1\right)\\
  &= \sum_{i=1}^{k/2}F(\zeta^j_{i})\,,
\end{split}
\eeq
where we set
$$
\zeta^j_i(x,y):=
\begin{cases}
  2(i-1)\pi & \text{if }x\geq 	g^j_{i-1}(y)\,,\\
  \theta_{M_j}(x,y) & \text{if } g^j_{i-1}(y)>x>g^j_{i}(y)\,,\\
  2\pi i & \text{if } x\leq g^j_i(y)\,.
\end{cases}
$$
Note that $\zeta^j_i- 2(i-1)\pi\in \mathcal{A}_2$ and
$F(\zeta^j_i- 2(i-1)\pi)=F(\zeta^j_i)$, and thus
$F(\zeta^j_i) \geq \min_{\theta \in \mathcal A_2} F(\theta)$. In turn,
by combining \eqref{corik} and \eqref{cutting}, we deduce that
$$
\inf_{\theta\in \mathcal{A}_{k}}F(\theta) = \lim_{j \to \infty}
F(\theta_{M_j}) \geq \frac{k}{2} \min_{\theta \in \mathcal A_2}
F(\theta)\,.
$$

In order to obtain the reverse inequality, we start from the minimizer
$\theta_{2,M_j}$ of the problem \eqref{pr:1}, with $k=2$ and $M=M_j$,
and define the function $\xi_j \in \mathcal A_{k}$ as
$$
\xi_j(x,y): =\sum_{i=0}^{k/2-1} \theta_{2,M_j}(x+2iM_j,y)\,,
$$ 
so that $F(\xi_j) = \frac{k}{2} F(\theta_{2,M_j})$. Then, we have
$$
\inf_{\theta\in \mathcal{A}_{k}}F(\theta) \leq \lim_{j \to \infty}
F(\xi_j) = \frac{k}{2} \lim_{j \to \infty}F(\theta_{2,M_j}) =
\frac{k}{2} \min_{\theta \in \mathcal A_2} F(\theta),
$$
where the last equality follows from the proof of
Theorem~\ref{th:k=2}.
\end{proof}

\subsection{Uniqueness of minimizers and classification of critical
  points} \label{sec:unique}
Next we address uniqueness of minimizers for the problem \eqref{mink}.
In fact, we will classify all the critical points subject to constant
boundary conditions at infinity; i.e., domain wall solutions to the
boundary reaction-diffusion {type} problem of the form in
\eqref{limitEL2} satisfying \eqref{binfty}. 

We start by showing that such critical points are smooth up to the
boundary, with uniform estimates at infinity.  To this aim, given
$t\in \R$ we denote \beq\label{ot} \Sigma_t^\pm:=\{(x,y)\in \Sigma:\,
x\gtrless t\}\,, \eeq and we recall that given an open set
$\Om\subset\R^2$ with Lipschitz boundary the trace space
$H^{1/2}(\pa\Om)$ of $H^1(\Om)$ may be equipped with the norm
$\|w\|^2_{H^{1/2}(\pa\Om)}:=\|w\|^2_{L^2(\pa\Om)}+
[w]^2_{\mathring{H}^{1/2}(\pa\Om)}$, where
$[w]^2_{\mathring{H}^{1/2}(\pa\Om)}$ stands for the squared Gagliardo
seminorm \beq\label{Gsemi}
[w]^2_{\mathring{H}^{1/2}(\pa\Om)}:=\int_{\pa\Om}
\int_{\pa\Om}\frac{|w(\mathbf{r})-w(\mathbf{r'})|^2}{|\mathbf{r}-
  \mathbf{r'}|^2}\, d\H^1({\mathbf r})d\H^1({\mathbf r'})\,.  \eeq
Moreover, with a slight abuse of notation, for any subset
$\Gamma\subset\pa\Om$ (and for $w\in H^{1/2}(\pa\Om)$) we will denote
$\|w\|^2_{H^{1/2}(\Gamma)}:=
\|w\|^2_{L^2(\Gamma)}+[w]^2_{\mathring{H}^{1/2}(\Gamma)}$, where
$[w]^2_{\mathring{H}^{1/2}(\Gamma)}$ is defined as in \eqref{Gsemi},
with $\pa\Om$ replaced by $\Gamma$.
  
\begin{lemma}\label{lm:elem} Let $\theta\in H^1_{l}(\Sigma) \cap
  L^\infty(\Sigma)$ be a solution of \eqref{wEL2}.  Then, up to
  choosing a representative, the following statements hold true:
\begin{itemize}
\item[a)] $\theta\in C^\infty(\overline\Sigma)$, {and} for {every}
  $k\in \N$ there exists a constant
  $C_k = {C_k(\gamma, h, \|\theta\|_\infty)} >0$ such that
  \beq\label{ofallorder} \|\theta\|_{C^k(\overline\Sigma)}\leq C_k\,;
  \eeq
\item [b)] if in addition $\theta$ satisfies \eqref{binfty}, then the
  convergence at infinity is uniform with respect to the $C^k$-norm
  for any $k\in \N${, i.e.,} \beq\label{inftypm}
  \lim_{t\to-\infty}\|\theta-\ell^-\|_{C^k(\overline\Sigma^-_t)}=0
  \quad\text{and}\quad
  \lim_{t\to+\infty}\|\theta-\ell^+\|_{C^k(\overline\Sigma^+_t)}=0\,.
  \eeq Moreover, if $h=0$, then $\ell^-, \ell^+\in \frac\pi2\Z$, while
  if $h>0$, then $\ell^-, \ell^+\in \pi\Z $.
	\end{itemize}
\end{lemma}
\begin{proof} In what follows, for all $t\in \R$ and $R>0$ we set
  $Q^t_R:=(t-R, t+R)\times (0,1)$; moreover, $C$ will denote a
  positive constant depending only on $R$ that may change from line to
  line.

  We first observe that by a standard Caccioppoli Inequality type
  argument, that is, testing \eqref{wEL2} with
  $\varphi=\eta^2 \theta$, where $\eta\in C^\infty$ is with compact
  support in $\overline\Sigma$, we may infer from the boundedness of
  $\theta$ that $\nabla\theta$ is uniformly locally bounded with
  respect to the $L^2$-norm. More precisely, for every $R>0$ there
  exists $C_1=C_1(\gamma, h, \|\theta\|_\infty, R)>0$ such that
  $\sup_{t\in \R}\|\theta\|_{H^1(Q^t_R)}\leq C_1$. In turn, by the
  Trace Theorem, see for instance \cite[Theorem 5.5]{necas}, we have
  $\|\theta\|_{H^{1/2}(\pa Q^t_R\cap \pa\Sigma)}\leq
  \|\theta\|_{H^{1/2}(\pa Q^t_R)}\leq C\|\theta\|_{H^1(Q^t_R)}\leq
  CC_1$ and, in turn, using the definition \eqref{Gsemi} of the
  Gagliardo seminorm we may check that
  $\|\sin (2\theta)\|_{H^{1/2}(\pa Q^t_R\cap \pa\Sigma)}\leq
  C\|\theta\|_{H^{1/2}(\pa Q^t_R\cap \pa\Sigma)}\leq CC_1$. Thus,
  \beq\label{C'(R)} \sup_{t\in \R}\|\gamma \sin
  (2\theta)\|_{H^{1/2}(\pa Q^t_R\cap \pa\Sigma)}\leq \gamma C C_1\,.
  \eeq
  
  Fix $t\in \R$ and a cut-off function $\zeta\in C^\infty_c(-3R, 3R)$,
  $0\leq\zeta\leq 1$, and $\zeta\equiv 1$ in $[-2R, 2R]$.  Let
  $\Omega^0\subset \R^2$ be a bounded domain with boundary of class
  $C^\infty$ such that $Q^0_{3R}\subset\Om^{0} {\subset \Sigma}$, and
  {let} $\Om^t:=\{(x,y):\, (x-t, y)\in \Om^0\}$. Finally, denote by
  $g$ the function defined for $\H^1$-a.e.  $(x,y)\in \pa\Om^t$ by
$$
g(x,y):=
\begin{cases}
  -\gamma\, \zeta(x-t)\sin(2\theta(x,y)) & \text{if }(x,y)\in
  \pa\Om^t\cap \pa\Sigma\,,\\ 
  0 & \text{otherwise.}
\end{cases}
$$
Using again \eqref{Gsemi}, one can check that
$g\in H^{1/2}(\pa\Om^t)$,
with \beq\label{prep-lift} \|g\|_{H^{1/2}(\pa\Om^t)} \leq \gamma C
\|\sin(2\theta)\|_{H^{1/2}(\pa Q^t_{2R}\cap\pa\Sigma)}\,, \eeq where
$C>0$ depends only on $\zeta$ and $\Om^0$ and thus, ultimately, only
on $R$. In turn, by \cite[Theorem~1.5.1.2]{grisvard} there exists a
lifting function $\tilde g\in H^2(\Om^t)$ such that
$\pa_\nu \tilde g=g$ on $\pa \Om^t $ and \beq\label{g} \|\tilde
g\|_{H^2 (\Om^t)}\leq C\|g\|_{H^{1/2}(\pa\Om^t)} \leq\gamma C'
\|\sin(2\theta)\|_{H^{1/2}(\pa Q^t_{2R}\cap\pa\Sigma)}\,, \eeq where
we used \eqref{prep-lift} (and, again, the constants $C, C'$ depend
only on $R$).  Since $\pa_\nu \tilde g=g=-\gamma\sin(2\theta)$ on
$ \pa Q^t_{2R}\cap\pa\Sigma$, integration by parts yields
$$
\int_{\Sigma}(\nabla \tilde g\cdot \nabla \varphi+ \Delta\tilde g\,
\varphi)\, d^2 r+\gamma\int_{\pa\Sigma}\sin(2\theta)\varphi\,
d\mathcal H^1=0 \qquad \forall \varphi\in H^1_l({\Sigma}) \text{ with
  supp$\,\varphi\subset \overline{Q^t_{2R}}$}.
$$
Subtracting the above identity from \eqref{wEL2} and setting
$w:=\theta-\tilde g$, we get
$$
\int_{\Sigma}\big(\nabla w\cdot \nabla \varphi+
(h\sin\theta-\Delta\tilde g)\, \varphi\big)\, d^2 r=0 \qquad \forall
\varphi\in H^1_l({\Sigma}) \text{ with
  supp$\,\varphi\subset \overline{Q^t_{2R}}$}\,,
$$
that is $w$ is a weak solution to
$$
\begin{cases}
   \Delta w=h\sin\theta-\Delta \tilde g & \text{in } Q^t_{2R}\,,\\
   \pa_\nu w= 0& \text{on } \pa Q^t_{2R}\cap\pa\Sigma\,.
   \end{cases}
$$   
Thus, by standard $H^2$-estimates (see for instance \cite{gilbarg})
and taking into account \eqref{C'(R)} and \eqref{g}, we get
\beq\label{E2} \|\theta\|_{H^2(Q^t_R)}\leq \|w\|_{H^2(Q^t_R)}+
\|\tilde g\|_{H^2(Q^t_{2R})}\leq C\big(\| h\sin\theta-\Delta \tilde
g\|_{L^2 (Q^t_{2R})}+\|w\|_{H^1 (Q^t_{2R})}+ \gamma C C_1\big)\leq
C_2\, \eeq for a suitable positive constant $C_2$ depending only on
$R$, $\|\theta\|_\infty$, $\gamma$, and $h$.

We can now start a bootstrap argument in order to obtain uniform
estimates also with respect to higher norms. Owing to \eqref{E2} and
to the fact that $\|\sin(2\theta)\|_{H^2(Q^t_R)} \leq M $, with
$M=M \big(\|\theta\|_{H^2(Q^t_R)}\big) $ (and thus ultimately
depending only $R$, $\|\theta\|_\infty$, $\gamma$, and $h$), by
applying the Trace Theorem again we can improve \eqref{C'(R)} to
obtain for all $t\in \R$
   $$
   \sup_{t\in \R}\|\sin (2\theta)\|_{H^{3/2}(\pa Q^t_R\cap
     \pa\Sigma)}\leq\gamma C M\,.
   $$ 
   Now, arguing as above and relying again on \cite[Theorem
   1.5.1.2]{grisvard} we may find a ``lifting'' function
   $\tilde g\in H^3(\Om^t)$ such that
   $\pa_\nu \tilde g=-\gamma\sin(2\theta)$ on
   $\pa Q^t_{2R}\cap \pa\Sigma $ and
  $$
  \|\tilde g\|_{H^3 (\Om^t)}\leq \gamma C
  \|\sin(2\theta)\|_{H^{3/2}(\pa Q^t_{2R}\cap\pa\Sigma)}\,.
  $$
  Thus, defining $w$ as before and arguing similarly, we clearly may
  improve estimate \eqref{E2} to obtain for every $R>0$
$$
\sup_{t\in \R} \|\theta\|_{H^3(Q^t_R)}\leq C_3
$$
for a suitable positive constant $C_3$ depending only on $R$,
$\|\theta\|_\infty$, $\gamma$, and $h$. We can now iterate this
argument to show that for every $k\in \N$ there exists a positive
constant $C_k$ depending only on $R$, $\|\theta\|_\infty$, $\gamma$,
and $h$ such that \beq\label{Ek} \sup_{t\in \R}
\|\theta\|_{H^k(Q^t_R)}\leq C_k \eeq for all $R>0$. In turn,
\eqref{Ek} combined with the Sobolev Embedding Theorem yields
\eqref{ofallorder}.

The uniform bounds \eqref{ofallorder}, together with the convergence
condition in Definition~\ref{def:dws} give \eqref{inftypm}. The latter
in particular implies that both $\Delta \theta$ and $\pa_\nu\theta$
vanish at infinity. Thus, from \eqref{limitEL2} we deduce that
$\sin(2\ell^\pm)=0$ and that also $\sin (\ell^\pm)=0$ when $h>0$.  The
last part of statement b) readily follows.
 \end{proof}


 In the next lemma we show that in the case $h=0${, or} $h>0$ and
 $F(\theta)<+\infty$, condition \eqref{wbinfty} is equivalent to
 \eqref{binfty}.

\begin{lemma}\label{lm:equiv} 
  Let
  $\theta\in C^2(\Sigma)\cap C^1(\overline\Sigma) \cap
  L^\infty(\Sigma)$ be a solution of \eqref{limitEL2} such that
  \eqref{wbinfty} holds. Assume that either $h=0$, or $h>0$ and
  $F(\theta)<+\infty$.  Then also \eqref{binfty} holds true.
\end{lemma}
\begin{proof}
  Consider first the case $h=0$. Let $\{\lambda_n\}$ be a sequence
  such that $\lambda_n\to+\infty$ and set
  $\theta_n:=\theta(\cdot+\lambda_n, \cdot)$. By statement a) of
  Lemma~\ref{lm:elem} we have that for every $k\in\N$ the sequence
  $\{\theta_n\}$ is uniformly bounded with respect to the $C^k$-norm
  on $\overline\Sigma$. Therefore, we may find a subsequence
  $\{\theta_{n_k}\}$ and a bounded function $\theta_\infty$ solving
  \eqref{limitEL2} such that $\theta_{n_k}\to\theta_{\infty}$ in $C^k$
  on the compact subsets of $\overline{\Sigma}$ for every $k\in
  \N$. Moreover, in view of \eqref{wbinfty} we also have
  $\theta_\infty=\ell^+$ on $\pa \Sigma$.  In particular,
  $\theta_\infty$ is a bounded harmonic function in $\Sigma$, which is
  constant on $\pa \Sigma$. It easily follows that
  $\theta_\infty\equiv \ell^+$. One way to see this is to extend the
  harmonic function $\theta_\infty-\ell^+$ to the whole plane by
  repeated odd reflections across the lines $\{x=j\}$, $j\in \Z$, thus
  getting an entire bounded harmonic function $w$, vanishing on such
  lines. Liouville's Theorem implies that $w\equiv 0$ in $\R^2$ and
  thus, in particular, $\theta_\infty\equiv \ell^+$ in $\Sigma$. In
  turn, this implies that
  $\theta(\lambda_{n_k}, y) \to \theta_\infty(0, y)=\ell^+$ as
  $k\to\infty$ for all $y\in [0,1]$. By the arbitrariness of
  $\{\lambda_n\}$ we have shown that the second condition in
  \eqref{binfty} is satisfied. A similar argument shows that also the
  first one holds true.
 
  Assume now that $h>0$ and $F(\theta)<+\infty$ and note that the
  latter condition immediately implies that both
  $\ell^-, \ell^+\in\pi\Z$. We may now run a similar argument as in
  the $h=0$ case.  Let $\{\lambda_n\}$, $\{\theta_n\}$ be as before
  and let $\theta_\infty$ be the limit (up to a subsequence) of
  $\theta_n$. One can show that in this case $\theta_\infty$ solves
$$
\begin{cases}
   \Delta \theta_\infty=h\sin\theta_\infty & \text{in }\Sigma\,,\\
   \pa_\nu\theta_\infty=0& \text{on }\pa \Sigma\,,\\
   \theta_\infty=\ell^+ & \text{on }\pa \Sigma\,.
   \end{cases}
 $$
 Even reflections with respect to $\pa \Sigma$ allow one to extend
 $\theta_\infty$ to a function $\tilde \theta_\infty$ defined on the
 ``tripled'' stripe $\tilde\Sigma:=\R\times (-1, 2)$ still solving the
 same equation
 $$
 \Delta \tilde \theta_\infty=h\sin\tilde\theta_\infty \quad\text{in }\tilde\Sigma\,.
 $$ 
 By classical results, see for instance \cite[Theorem
 6.8.2]{morrey66}, we infer that $\tilde \theta_\infty$ is analytic in
 $\tilde\Sigma$ and thus, in particular, $ \theta_\infty$ is analytic
 in $\Sigma$ up to the boundary. But then, owing to the overdetermined
 boundary conditions on $\pa \Sigma$, by the Cauchy-Kovalevskaya
 Theorem (see for instance \cite{evans}) it follows that
 $\theta_\infty\equiv \ell^+$ in a neighborhood of $\pa\Sigma$ and
 thus, by analyticity, everywhere in $\Sigma$. This establishes the
 second condition in \eqref{binfty} and the first one can be proven
 similarly.
\end{proof} 

We now start paving the way for the application of the sliding method
to our situation. We recall that owing to Lemma~\ref{lm:elem},
bounded weak solutions to \eqref{limitEL2} are in fact smooth
classical solutions and thus, in what follows, we will not distinguish
between weak and strong formulations.  We begin with the following
comparison principle for problem \eqref{limitEL2}, where we will be
using notation \eqref{ot}.

  \begin{lemma}\label{lm:cp}
    Let $t\in \R$ and let
    $\theta_1, \theta_2$ be domain wall solutions to \eqref{limitEL2} according
    to Definition~\ref{def:dws}, with $\theta_1\leq \theta_2$ on
    $\Gamma_t:=\{x=t\}\cap \Sigma$. Denote by $\ell^-_i$, $\ell^+_i$,
    $i=1, 2$, the boundary conditions at infinity of $\theta_i$
    according to \eqref{binfty} and assume also that
    $\ell^+_1\leq\ell^+_2$.  Assume also that there exists an interval
    $J {= (\theta^-, \theta^+)}$ such that \beq\label{inJ}
    \sup_{\Sigma_t^+}\theta_1<{\theta^+} \text{ and }
    \inf_{\Sigma_t^+}\theta_2>{\theta^-} \,, \eeq and
    $\theta\mapsto\sin(2\theta)$ is strictly increasing in $J$,
    {together with} $\theta \mapsto\sin(\theta)$ if $h>0$. Then,
    $\theta_1\leq\theta_2$ in $\overline \Sigma_t^+$. The same
    statement holds true with $\ell_i^+$ and $\Sigma_t^+$ replaced by
    $\ell_i^-$ and $\Sigma_t^-$, respectively.
  \end{lemma}
  \begin{proof} We prove the statement only for $\Sigma^+_t$, the
    other case being analogous. For any fixed $\e>0$ set
    $\varphi_\e:=(\theta_1-\theta_2-\e)^+\chi_{\Sigma_t^+}$ and note
    that from the assumptions $\theta_1\leq\theta_2$ on $\Gamma_t$ and
    $\ell^+_1\leq\ell^+_2$, taking into account part b) of
    Lemma~\ref{lm:elem}, we conclude that the function $\varphi_\e$ is
    in $H^1(\Sigma)$ with bounded support contained in
    $\overline\Sigma^+_t$. Testing \eqref{wEL2} for $\theta_i$ with
    $\varphi_\e$ and subtracting the two resulting equations we get
  \begin{multline*}
    \int_{\Sigma_t^+}|\nabla \varphi_\e|^2\, d^2 r+
    h\int_{\{\theta_1-\theta_2>\e\}\cap
      \Sigma^+_t}(\sin(\theta_1)-\sin(\theta_2))\varphi_\e\,
    d^2 r\\
    +\gamma\int_{\{\theta_1-\theta_2>\e\}\cap
      (\pa\Sigma^+_t\setminus\Gamma_t)}(\sin(2\theta_1)
    -\sin(2\theta_2))\varphi_\e\, d\mathcal H^1=0\,.
  \end{multline*}
  Note that {$\theta_1(\cdot)$, $\theta_2(\cdot)\in J$  in
    $\{\theta_1-\theta_2>\e\}\cap \Sigma^+_t$}, thanks to
  \eqref{inJ}. Using now the monotonicity assumptions on {
    $\sin(2\theta)$ and $\sin(\theta)$ for $\theta \in J$}, we may
  conclude from the above integral identity that
  $\nabla\varphi_\e\equiv 0$ and that $\theta_1-\theta_2\leq \e$, or
  equivalently $\varphi_\e=0$ on $\pa\Sigma_t^+$. Thus,
  $\varphi_\e\equiv 0$, that is, $\theta_1-\theta_2\leq \e$ in
  $\overline\Sigma_t^+$. The conclusion follows from the arbitrariness
  of $\e$.  \end{proof}

In the lemma below, we write down a version of the Strong Maximum
Principle which works {for}~\eqref{limitEL2}. Note that a similar
principle (and the argument behind) has been used already in the proof
of Theorem~\ref{th:k=2}.
 \begin{lemma}\label{lm:smp}
   Let $U\subset \R^2$ be a connected open set and let $\theta_1$,
   $\theta_2\in C^2(\Sigma) \cap C^1(\overline \Sigma)$ be solutions
   of \eqref{limitEL2} such that $\theta_1\leq\theta_2$ in
   $U\cap \Sigma$. Assume that
   $\theta_1(\bar x,\bar y)=\theta_2(\bar x,\bar y)$ for some point
   $(\bar x,\bar y) \in U\cap \overline \Sigma$. Then
   $\theta_1=\theta_2$ in $U\cap\overline \Sigma$.
 \end{lemma}
 
 \begin{proof}
   We can argue similarly as in the proof of
   Theorem~\ref{th:k=2}. Indeed, setting $u:=\theta_2-\theta_1$, we
   note that $u$ is smooth up to $U\cap \pa\Sigma$ and satisfies
   \beq\label{u3}
  \begin{cases}
    \Delta u=\tilde c u & \text{in }U\cap \Sigma\,,\\
    \pa_\nu u=-\gamma (\sin(2\theta_2)-\sin(2\theta_1)) & \text{on }
    U\cap \pa\Sigma\,,\\ 
    u\geq 0 & \text{in }U\cap \Sigma\,,\\
    u(\bar x, \bar y)=0\,,
  \end{cases}
  \eeq
  where now
  $$\tilde c:=
  \begin{cases}
    h\frac{\sin(\theta_2)-\sin(\theta_1)}{\theta_2-\theta_1} &
    \text{in  }U\cap \{\theta_2>\theta_1\}\,,\\ 
    h & \text{in }U\cap \{\theta_2=\theta_1\}\,.
  \end{cases}
  $$ 
  Notice that if $\bar y\in \{0,1\}$, then by Hopf's Lemma \cite[Lemma
  3.4]{gilbarg} we have $\pa_\nu u(\bar x, \bar y)\neq0$, which
  contradicts the Neumann boundary condition in \eqref{u3}.  Thus,
  necessarily $\bar y\in (0,1)$.  We may then invoque the Strong
  Maximum Principle \cite[Theorem 2.2]{pucci04} to conclude that
  $u\equiv 0$ and in turn $\theta_2=\theta_1$ in $U\cap \Sigma$.
  \end{proof} 	
  
%

  We continue now with some elementary considerations, showing in
  particular that only some specific values are admissible for
  $\ell_1$ and $\ell_2$.

  As a consequence of the Strong Maximum Principle and of the
  comparison Lemma~\ref{lm:cp} we have the following observation,
  which will be instrumental in the implementation of the sliding
  method.

\begin{lemma}\label{lm:sliding}
  Let
  $\theta_1, \theta_2$
  be domain wall solutions to \eqref{limitEL2} according to
  Definition~\ref{def:dws}, and denote by $\ell^-_i$, $\ell^+_i$,
  $i=1, 2$, the boundary conditions at infinity of $\theta_i$
  according to \eqref{binfty}. Assume that $\theta_1\leq \theta_2$ in
  $\Sigma$ and that $\ell^-_1>\ell^+_2$. Assume also that there exist
  two open intervals $J^+$, $J^-$ where $\theta\mapsto\sin(2\theta)$
  is strictly increasing and so is $\theta \mapsto\sin(\theta)$ if
  $h>0$, and such that $\ell^\pm_2\in J^\pm$.  Then, there exists
  $\lambda\in \R$ such that
  $\theta_1(\cdot+\lambda, \cdot)\equiv \theta_2$.
\end{lemma}
\begin{proof} Let us first show that it is impossible to have
  $\ell_2^+>\ell_1^+$ or $\ell_2^->\ell_1^-$. To this aim we argue by
  contradiction.

  Assume first that $\ell_2^\pm>\ell_1^\pm$. Since also
  $\ell^-_1>\ell^+_2$, there exists $\lambda\in \R$ such that
  $\theta_1(\cdot+\lambda, \cdot)\leq \theta_2$ and
  {$\theta_1(\bar x + \lambda, \bar y)=\theta_2(\bar x + \lambda,
    \bar y)$} for some point
  ${(\bar x, \bar y)} \in \overline \Sigma$.  Thus by
  Lemma~\ref{lm:smp} the two solutions coincide which contradicts our
  initial assumption $\ell_2^\pm>\ell_1^\pm$.
  
  Assume now that $\ell_2^->\ell_1^-$ but $\ell_2^+=\ell_1^+=:\ell^+$.
  Owing to Lemma~\ref{lm:elem}-b) and the fact that $\ell^+\in J^+$,
  we may choose $t^+$ such that \beq\label{tbar} \inf J^+<
  \inf_{\Sigma_{t^+}^+}\theta_2\leq \sup_{\Sigma_{t^+}^+}\theta_2<\sup
  J^+\,.  \eeq Set now
 $$
 \lambda_0:=\inf\{\lambda\leq0:\, \theta_1(\cdot+ \lambda, \cdot)\leq
 \theta_{2}\}\,.
$$
Note that thanks to the assumption $\ell_1^->\ell_2^+$ we have
$\lambda_0\in \R$.  Moreover, clearly
$\theta_1(\cdot+ \lambda_0, \cdot)\leq \theta_{2}$ and thus, in
particular, recalling \eqref{tbar}, we have \beq\label{tbar2}
\sup_{\Sigma_{ t^+}^+}\theta_1(\cdot+ \lambda_0, \cdot)<\sup J^+.
\eeq We claim that $\theta_1(\cdot+ \lambda_0, \cdot)$ and $\theta_2$
coincide at some point in $\overline\Sigma$.  Indeed if by
contradiction $\theta_1(\cdot+ \lambda_0, \cdot) < \theta_{2}$
everywhere, then, using also that $\ell_2^->\ell_1^-$, we have
$\min_{\overline\Sigma^-_{t^+}}(\theta_2-\theta_1(\cdot+ \lambda_0,
\cdot))>0$. By uniform continuity, recalling \eqref{tbar2}, we may
find $\e>0$ so small that \beq\label{piccolo}
\min_{\overline\Sigma^-_{t^+}}(\theta_2-\theta_1(\cdot+ \lambda_0-\e,
\cdot))>0 \quad\text{and}\quad \sup_{\Sigma_{t^+}^+}\theta_1(\cdot+
\lambda_0-\e, \cdot)<\sup J^+\,.  \eeq Recalling also \eqref{tbar}, we
are in a position to apply Lemma~\ref{lm:cp} to infer that
$\theta_1(\cdot+ \lambda_0-\e, \cdot)\leq \theta_2$ in
$\Sigma_{t^+}^+$ and in turn, thanks to the first condition in
\eqref{piccolo}, $\theta_1(\cdot+ \lambda_0-\e, \cdot)\leq \theta_2$
in $\Sigma$. This contradicts the minimality of
$\lambda_0$. Therefore, $\theta_1(\cdot+ \lambda_0, \cdot)$ and
$\theta_2$ must coincide at some point in $\overline\Sigma$ and thus
everywhere thanks to the Strong Maximum Principle. This again leads to
a contradiction. The case where $\ell_2^+>\ell_1^+$ but
$\ell_2^-=\ell_1^-$ is clearly analogous.

It remains to consider the case $\ell^\pm_1=\ell^\pm_2$. In this case
choose $t^+$ as before.  Arguing similarly as before and recalling
that $\ell^-_2\in J^-$, we may also find $t^-<t^+$ such that
\beq\label{tbar-} \inf J^-< \inf_{\Sigma_{t^-}^-}\theta_2\leq
\sup_{\Sigma_{t^-}^-}\theta_2<\sup J^-\,.  \eeq Let $\lambda_0$ be as
before. We are going to show that in this case
$\theta_1(\cdot+ \lambda_0, \cdot)$ and $\theta_2$ coincide at some
point in $\overline\Sigma$ and thus everywhere by Lemma~\ref{lm:smp}.
Indeed otherwise
 $$
 \min_{\overline\Sigma^+_{t^-}\cap
   \overline\Sigma^-_{t^+}}(\theta_2-\theta_1(\cdot+ \lambda_0,
 \cdot))>0\,.
 $$ 
 Then, recalling \eqref{tbar2} and noticing also that
 $\sup_{\Sigma_{ t^-}^-}\theta_1(\cdot+ \lambda_0, \cdot)\leq
 \sup_{\Sigma_{ t^-}^-}\theta_2<\sup J^-$ by \eqref{tbar-}, we may
 find $\e>0$ so small that \beq\label{piccolo2}
 \min_{\overline\Sigma^+_{t^-}\cap
   \overline\Sigma^-_{t^+}}(\theta_2-\theta_1(\cdot+ \lambda_0-\e,
 \cdot))>0\,,\,\,\sup_{\Sigma_{t^-}^-}\theta_1(\cdot+ \lambda_0-\e,
 \cdot)<\sup J^-\text{ and } \sup_{\Sigma_{t^+}^+}\theta_1(\cdot+
 \lambda_0-\e, \cdot)<\sup J^+\,.  \eeq Taking into account also
 \eqref{tbar} and \eqref{tbar2}, we may apply Lemma~\ref{lm:cp} to
 infer that $\theta_1(\cdot+ \lambda_0-\e, \cdot)\leq \theta_2$ in
 $\Sigma_{t^\pm}^\pm$ and in turn, thanks to the first condition in
 \eqref{piccolo}, $\theta_1(\cdot+ \lambda_0-\e, \cdot)\leq \theta_2$
 in $\Sigma$. This contradicts the minimality of $\lambda_0$ and the
 conclusion follows.\end{proof}

We are now ready to prove the main result of this section, showing
that domain wall solutions in the sense of Definition~\ref{def:dws}
are unique up to horizontal translations and addition of integer
multiples of $\pi$, and coincide with the global minimizer constructed
in Theorem~\ref{th:k=2}, which is in turn unique.

\begin{proof}[Proof of Theorem~\ref{th:uniqueness}] We only consider
  the case $h=0$, the other one being analogous. We recall that by
  Lemma~\ref{lm:elem} $\ell^-$, $\ell^+\in \frac\pi2\Z$, hence there
  are three possible cases: $\ell^--\ell^+>\pi$, $\ell^--\ell^+=\pi$,
  and $\ell^--\ell^+=\frac\pi2$.

  We start by showing that the first case cannot occur. Indeed, assume
  by contradiction that $\ell^--\ell^+>\pi$ and recall that
  $\tilde\theta:=\theta+\pi$ is also a domain wall solution thanks to
  Remark~\ref{rm:dws}-b). Moreover,
  $\ell^->\ell^++\pi=\tilde\ell^+$. Then, arguing as at the beginning
  of the proof of Lemma~\ref{lm:sliding} we may find $\lambda\leq 0$
  such that $\theta(\cdot+\lambda, \cdot)$ and $\theta+\pi$ coincide
  at some point in $\overline\Sigma$ and thus everywhere by the Strong
  Maximum Principle Lemma~\ref{lm:smp}. This is clearly impossible.

  Let us now assume $\ell^--\ell^+\leq \pi$.  First of all note that
  since $\ell^+\in \frac\pi2\Z$, upon replacing $\theta$ by
  $\theta+k\pi$ for a suitable $k\in \Z$, we may assume thanks to
  Remark~\ref{rm:dws}-b) that either $\ell^+=0$ or
  $\ell^+=-\frac\pi 2$. Let us consider first the case $\ell^+=0$ and
  thus $\ell^-\in\{\frac\pi 2, \pi\}$.  Note that by the Strong
  Maximum Principle (Lemma~\ref{lm:smp}) we may easily infer that
  $\theta<\theta(\cdot+\lambda, \cdot)+\pi$ for all $\lambda\in
  \R$. Indeed if not, it would be possible to find $\lambda_0\in \R$
  such that $\theta\leq\theta(\cdot+\lambda, \cdot)+\pi$, with the two
  functions coinciding at some point and therefore everywhere by
  Lemma~\ref{lm:smp}, which is clearly impossible. In turn,
  \beq\label{pi} \theta\leq
  \lim_{\lambda\to+\infty}\theta(\cdot+\lambda,
  \cdot)+\pi=\ell^++\pi=\pi\,, \eeq and in fact the inequality is
  strict thanks to Lemma~\ref{lm:smp} and the fact that the constant
  function $\pi$ is also a solution to \eqref{limitEL2}.

Now recall that $\theta_{min}${, the minimizer from Theorem
  \ref{th:k=2},} vanishes at ${x =} +\infty$ and converges to $\pi$
at ${x = } -\infty$. In particular, thanks to Lemma~\ref{lm:elem}
we have \beq\label{limthetamin}
\lim_{t\to-\infty}\|\theta_{min}-\pi\|_{L^\infty(\Sigma^-_t)}=0\quad\text{and}\quad
\lim_{t\to+\infty}\|\theta_{min}\|_{L^\infty(\Sigma^+_t)}=0\,; \eeq
moreover, $0<\theta_{min}<\pi$ in $\overline\Sigma$. Thus, we may find
$t^-<t^+$ such that \beq\label{tsegnato}
\frac34\pi<\theta_{min}<\pi\quad\text{in
}\overline\Sigma^-_{t^-}\qquad\text{and}\qquad
0<\theta_{min}<\frac\pi4\quad\text{in }\overline\Sigma^+_{t^+}\,.
\eeq Clearly, we also have that \beq\label{mtsegnato}
m:=\min_{\overline\Sigma^-_{t^+}}\theta_{min}>0\,.  \eeq Since by
Lemma~\ref{lm:elem} we also have
$$
\lim_{t\to+\infty}\|\theta\|_{L^\infty(\Sigma^+_t)}=0\,,
$$	
we may now find $\lambda>0$ so large that \beq\label{lam} -\frac \pi4
< -m<\theta(\cdot+\lambda, \cdot)<m<\frac\pi4 \quad\text{in }
\overline\Sigma^+_{t^-}\,, \eeq where $m$ is the constant in
\eqref{mtsegnato}. We claim that \beq\label{cl1} \theta(\cdot+\lambda,
\cdot)\leq \theta_{min} \quad\text{in }\Sigma\,.  \eeq Indeed,
\eqref{mtsegnato} and \eqref{lam} imply that the inequality holds in
$\Sigma^+_{t^-}\cap \Sigma^-_{t^+}$. It remains to show that the
inequality $\theta(\cdot+\lambda, \cdot)\leq \theta_{min} $ holds also
in $\Sigma^\pm_{t^\pm}$. Let us start with $\Sigma^+_{t^+}$. Recall
that $\theta(\cdot+\lambda, \cdot)< \theta_{min}$ on
$\{(x,y):x=t^+\}\cap \overline\Sigma$ thanks to \eqref{mtsegnato} and
\eqref{lam}. Note also that \eqref{lam}) implies
$\sup_{\Sigma^+_{t^+}}\theta(\cdot+\lambda, \cdot)<\frac\pi4$. As
clearly $\inf_{\Sigma^+_{t^+}}\theta_{min}=0$, we may apply
Lemma~\ref{lm:cp} with $\theta_1= \theta(\cdot+\lambda, \cdot)$,
$\theta_2= \theta_{min}$, $J=(-\frac\pi4, \frac\pi4)$, to infer
$\theta(\cdot+\lambda, \cdot)\leq \theta_{min} $ in $\Sigma^+_{t^+}$.
Concerning $\Sigma^-_{t^-}$, observe that
$\sup_{\Sigma^-_{t^-}}\theta(\cdot+\lambda, \cdot)\leq\pi$ and
$\inf_{\Sigma^-_{t^-}}\theta_{min}>\frac34\pi$ by \eqref{pi} and
\eqref{tsegnato}, respectively. Moreover,
$\theta(\cdot+\lambda, \cdot)< \theta_{min}$ on
$\{(x,y):x=t^-\}\cap \overline\Sigma$ thanks to \eqref{mtsegnato} and
\eqref{lam}. Thus we may apply again Lemma~\ref{lm:cp} with
$\theta_1$, $\theta_2$ as before and $J=(\frac34\pi, \frac54\pi)$ to
conclude that the inequality holds also in $\Sigma^-_{t^-}$ and thus
\eqref{cl1} is proven.

 We are now in a position to apply Lemma~\ref{lm:sliding} to deduce
 that there exists $\bar\lambda\in \R$ such that
 $\theta(\cdot+\bar \lambda, \cdot)= \theta_{min}$ in $\Sigma$.

 Finally, the case $\ell^+=-\frac\pi2$ can be dealt with similarly by
 finding $\lambda>0$ such that \eqref{cl1} holds and then by applying
 Lemma~\ref{lm:sliding} to conclude. The argument to show the
 existence of a such a $\lambda$ is similar as before, and in fact
 easier as we may take advantage of the fact that both limits at
 ${x = } \pm\infty$ of $\theta{(x, \cdot)}$ are strictly smaller
 than the corresponding limits of $\theta_{min}$. The details are left
 to the reader.
\end{proof}
 
We now collect several corollaries.  The first one is an immediate
consequence of Theorems~\ref{th:k=2}~and~\ref{th:uniqueness}.
  \begin{corollary}\label{cor:uniq}
    The minimum problem \eqref{mink} {with $k \in \N$} (see {also}
    Remark~\ref{rm:mink}) admits a solution if and only if $k=1$ in
    the case $h=0$, and {if and only if} $k=2$ in the case
    $h>0$. Moreover, the solution is unique and coincides, up to a
    translation in the $x$-direction, with the function $\theta_{min}$
    provided by Theorem~\ref{th:k=2}.
  \end{corollary}
  
  Setting $\check\theta_{min}(x,y):=\theta_{min}(-x,y)$, the previous
  corollary yields immediately the following result.
  \begin{corollary}\label{cor:mineps0}
    Any minimizer $m$ of \eqref{mineps0} coincides, up to {a
      translation} in the $x$-direction, with either {
      $(\cos\theta_{min}, \sin\theta_{min})$, or
      $(\cos\theta_{min}, -\sin\theta_{min})$, or
      $(\cos\check\theta_{min}, \sin\check\theta_{min})$, or
      $(\cos\check\theta_{min}, -\sin\check\theta_{min})$.}
  \end{corollary}

  The next corollary deals with symmetry and decay properties of the
  domain wall profile $\theta_{min}$.
 \begin{corollary}\label{cor:decay}
   In addition to the properties stated in Theorem~\ref{th:k=2}, the
   profile $\theta_{min}$ minimizing
   \eqref{mink} with $k = 1$ for $h = 0$, or $k = 2$ for $h > 0$
   satisfies:
   \begin{itemize}
   \item [a)] (symmetry) $\theta_{min}(x,y)=\theta_{min}(x,1-y)$ and
     $\theta_{min}(x,y)=k\pi-\theta_{min}(-x,y)$ for all
     $(x, y) \in \overline \Sigma$;
   \item [b)] (exponential decay at infinity) for every $m\in \N$
     there exist positive constants $\alpha_m$, $\beta_m$ such that
     $$
     \|\theta_{min}-k\pi\|_{C^m(\overline\Sigma^-_{-t})}\leq \alpha_m
     \mathrm{e}^{-\beta_m t}\quad\text{and}\quad
     \|\theta_{min}\|_{C^m(\overline\Sigma _t^+)}\leq \alpha_m
     \mathrm{e}^{-\beta_m t}
     $$ 
     for all $t>0$ sufficiently large.
   \end{itemize}
 \end{corollary}
      
 \begin{proof}
   Observing that $\theta_{min}(\cdot, 1-\cdot)$ is still a domain
   wall solution satisfying the normalization condition
   $\int_0^1\theta_{min}(0,y)\,dy=\frac{k\pi}2$, the first symmetry
   property follows at once from the uniqueness result of
   Theorem~\ref{th:uniqueness}. The second symmetry property is proven
   in a similar way, observing that $k\pi-\theta_{min}(-\cdot,\cdot)$
   is also a domain wall solution satisfying the same normalization
   condition.  This concludes the proof of part a) of the corollary.

   In order to prove the second part, we employ a barrier
   argument. Clearly, by the symmetry property established in part a)
   it is enough to show the exponential decay as $x \to +\infty$.  To
   this aim, we fix $\e_0>0$ so small that \beq\label{bar1}
   \sin(2\theta)\geq\theta\quad\text{for all }\theta\in (0, \e_0)\,,
   \eeq and choose $\bar t>0$ so large that \beq\label{bar2}
   0<\theta_{min}<\e_0\quad\text{in }\overline\Sigma^+_{\bar t}\,.
   \eeq Recall that this is possible due to the fact that
   $\|\theta_{min}\|_{L^\infty(\Sigma_t^+)}\to 0$ as $t\to+\infty$. We
   now define the barrier $\theta^+$ in $\Sigma^+_{\bar t}$ as
$$
\theta^+(x,y):=\e_0\psi(y)\mathrm{e}^{-\alpha(x-\bar t)}\,,
$$
where
$$
\psi(y):=1+{\frac{1}{2} \gamma y (1 - y), }
$$
and $\alpha=\alpha(\gamma)>0$ is a constant sufficiently small so that 
 \begin{align*}
   \Delta\theta^+(x,y)=
   \e_0 \mathrm{e}^{-\alpha(x-\bar t)}[\alpha^2 \psi(y)- \gamma]\leq 
   \e_0 \mathrm{e}^{-\alpha(x-\bar
   t)}\Big[\alpha^2\Big(1+\frac\gamma8\Big) - \gamma\Big]<0\,. 
 \end{align*}
With such a choice of $\alpha$, $\theta^+$ satisfies by construction
\beq\label{theta+}
\begin{cases}
  \Delta\theta^+<0 & \text{in }	 \Sigma^+_{\bar t}\,,\\
  \pa_\nu\theta^+=-\frac\gamma2\theta^+ & \text{on } \pa
  \Sigma^+_{\bar t}\cap\pa\Sigma\,,\\ 
  \theta^+=\e_0\psi\geq\e_0 & \text{on }\Gamma_{\bar t}\,.
\end{cases}
\eeq In particular, \beq\label{wtheta+} \int_{\Sigma^+_{\bar t}}\nabla
\theta^+\cdot \nabla \varphi\, d^2 r+\gamma\int_{\pa\Sigma^+_{\bar
    t}\cap\pa\Sigma}\frac{\theta^+}2\varphi\, d\mathcal H^1\geq 0 \eeq
for all non-negative $\varphi\in H^1(\Sigma^+_{\bar t})$ with bounded
support and vanishing on $\Gamma_{\bar t}$.  For any fixed $\eta>0$,
consider the test function
$\varphi_\eta:=(\theta_{min}-\theta^+-\eta)^+$ defined in
$\Sigma^+_{\bar t}$ and note that thanks to \eqref{bar2} and the last
condition in \eqref{theta+}, $\varphi_\eta=0$ on $\Gamma_{\bar t}$ so
that it can be extended by $0$ to the whole $\Sigma$. Moreover, by the
uniform convergence to $0$ of $\theta_{min}(x, \cdot)-\theta^+$ as
$x \to +\infty$, we have that $\varphi_\eta$ has bounded support in
$\overline\Sigma^+_{\bar t}$. Plugging $\varphi_\eta$ into
\eqref{wEL2}, with $\theta=\theta_{min}$, and also into
\eqref{wtheta+}, and subtracting the two resulting inequalities, we
get
\begin{multline*}
  \int_{\Sigma_{\bar t}^+}|\nabla \varphi_\eta|^2\, d^2 r +
  h\int_{\{\theta_{min}-\theta^+>\eta\}\cap \Sigma^+_{\bar
      t}}\sin(\theta_{min})\varphi_\eta\, d^2 r \\
  +\gamma\int_{\{\theta_{min}-\theta^+>\eta\}\cap (\pa\Sigma^+_{\bar
      t}\cap\pa \Sigma)} \left( \sin(2\theta_{min})-\frac{\theta^+}2
  \right) \varphi_\eta\, d\mathcal H^1\leq0\,.
  \end{multline*}
  Note that both $\sin(\theta_{min})$ and
  $\sin(2\theta_{min})-\frac{\theta^+}2$ are strictly positive in
  $\{\theta_{min}-\theta^+>\eta\} \cap \Sigma^+_{\bar t} $ (if
  nonempty), thanks to \eqref{bar1} and \eqref{bar2}. Thus for the
  above integral inequality to hold it is necessary that
  $\nabla\varphi_\eta\equiv 0$ in $\Sigma_{\bar t}^+$ and that the
  sets $\{\theta_{min}-\theta^+>\eta\}\cap \Sigma^+_{\bar t}$ and
  $\{\theta_{min}-\theta^+>\eta\}\cap (\pa\Sigma^+_{\bar t}\cap\pa
  \Sigma)$ have vanishing measures. Thus, $\varphi_\eta\equiv 0$, that
  is, $\theta_{\min}-\theta^+\leq \eta$ in $\Sigma_{\bar t}^+$. From
  the arbitrariness of $\eta$, we may conclude that
  $\theta_{\min}\leq \theta^+$ in $\Sigma_{\bar t}^+$ and thus
  \beq\label{expinfty} \|\theta_{min}\|_{L^\infty(\Sigma^+_t)}\leq
  \e_0\Big(1+\frac\gamma8\Big) \mathrm{e}^{\alpha \bar t}
  \mathrm{e}^{-\alpha t}\,.  \eeq for $t>\bar t$. The exponential
  decay with respect to any $C^m$-norm follows now from
  \eqref{expinfty} by an interpolation argument, taking into account
  that by Lemma~\ref{lm:elem}-a) for every $m \in \N$ there exists a
  constant $C_m>0$ such that
  $\|\theta_{min}\|_{C^m (\overline\Sigma^+_{\bar t})}\leq C_m$.
  \end{proof}

    \begin{proof}[Proof of Theorem \ref{t:exist}.]
      Finally, combining the results of Corollary \ref{cor:uniq} and
      Corollary \ref{cor:decay} yields the conclusion of Theorem
      \ref{t:exist}.
    \end{proof}

  \subsection{Limiting regimes} We now turn to the analysis of the
  minimizers of $F$ for $h = 0$ in the two extremes of the values of
  $\gamma$ {covered by Theorem \ref{th:glargesmall}}.

  \begin{proof}[Proof of item a) of Theorem \ref{th:glargesmall}] We
    show that as $\gamma \to 0$ we have
    $\theta_{min,\gamma}(x /\sqrt{\gamma},y) \to \pi - 2 \arctan
    (e^{2x})$ {locally uniformly in ${(x, y) \in } \Sigma$.}  Rescaling the {$x$
      coordinate} as $\tilde x = \sqrt{\gamma} x$ and defining
    $\tilde\theta(\tilde x, y) := \theta(x,y)$, we obtain
\begin{align}
  \label{eq:251}
  \tilde F_\gamma(\tilde\theta) :=   \frac{1}{\sqrt{\gamma}} \,
  F(\theta)=\frac12 {\int_0^1  \int_\R} \left( 
  |\partial_{\tilde x} \tilde\theta|^2 + \frac{1}{\gamma}
  |\partial_{y} \tilde\theta|^2 \right) d \tilde x\, d y  + \int_\R
  \left( \sin^2\tilde\theta(\tilde x,0) + \sin^2\tilde\theta(\tilde
  x,1) \right) d\tilde x\,.  
\end{align}
{For $\bar \theta \in H^1_{loc}(\R)$, we} can also define
${G(\bar \theta)}$ as
$$
G (\bar \theta) {:=} \int_\R \left( \frac12 |{\bar\theta}'|^2 +2
  \sin^2 {(\bar \theta)} \right) dx.
$$
Notice that if $\tilde \theta(x, y) = \bar \theta(x)$, then
$\tilde F_\gamma(\tilde \theta) = G(\bar \theta)$. Therefore, if
$\theta_{min, \gamma}$ is a minimizer of the energy $F(\theta)$ for a
fixed ${\gamma > 0}$ and $\theta_{min, \gamma}(0,\cdot)=\frac{\pi}{2}$
then it is clear that $\tilde F_\gamma(\tilde \theta_{min, \gamma})$ {is
  bounded independently of $\gamma$. This implies that}
${|\nabla \tilde \theta_{min, \gamma}|}$ is bounded in
${L^2} (\Sigma)$, and $\partial_y \tilde \theta_{min, \gamma} \to 0$
in $L^2(\Sigma)$ as $\gamma \to 0$. It follows that there is a
subsequence (not relabelled) such that
$\tilde \theta_{min, \gamma} \wto \theta_*$ weakly in
$H^1_{l}(\Sigma)$ and $\tilde \theta_{min, \gamma} \to \theta_*$ in
$L^2_{loc}(\pa\Sigma)$ (see, e.g., \cite{adams}) with
$\theta_* (x,y) = \bar \theta_* (x)$ for some
$\bar \theta_* \in H^1_{loc}(\R)$.
 
We observe that $\bar \theta_*$ is a
minimizer of the energy $G$ in the class
$$
{\mathcal{A}_1^{1d}} := \left\{\bar \theta\in H^1_{loc}(\R) :\,
|\bar \theta'|\in L^2(\R),\ \lim_{x\to+\infty} \bar \theta(x)=0, \
\lim_{x\to-\infty} \bar\theta(x)=\pi, \ \bar\theta(0)=\frac{\pi}{2}
\right\}.
$$ 
Indeed, for any $\bar \theta \in \mathcal{A}_1^{1d}$ and
$\theta(x,y) =\bar \theta(x)$ we have $\theta \in \mathcal A_1$ and
$$
G(\bar \theta)= \liminf_{\gamma \to 0} \tilde F_\gamma (\theta) \geq
\liminf_{\gamma \to 0} \tilde F_\gamma (\tilde\theta_{min,\gamma})
\geq G(\bar \theta_*).
$$
Therefore
$$
\bar\theta_* (x) = \pi - 2 \arctan (e^{2x})
$$
is the unique minimizer of $G$ in ${\mathcal{A}_1^{1d}}$ and we deduce
that $\tilde \theta_{min, \gamma} \to \theta_*$ in $H^1_l(\Sigma)$ for
the whole sequence.

{Finally, we note that by the strong convergence of
  $\tilde \theta_{min,\gamma}$ to $\theta_*$ in
  $L^2_{loc}(\partial \Sigma)$, monotonicity of
  $\tilde \theta_{min, \gamma}(\cdot, 0)$, and continuity and decay at
  infinity of $\theta_*$ we also have that
  $\tilde \theta_{min,\gamma} \to \theta_*$ uniformly in
  $\partial \Sigma$. Therefore, since $\theta_{min,\gamma}$ is
  harmonic in $\Sigma$, with the help of the representation
  \begin{align*}
    \tilde \theta_{min,\gamma} (x, y) = \int_{-\infty}^\infty P_\gamma(x -
    x', y) \, \tilde \theta_{min,\gamma} (x', 0) \, dx'
  \end{align*}
  from \eqref{eq:29}, where
  $P_\gamma(x, y) := \gamma^{-1/2} P(\gamma^{-1/2} x, y)$ and
  $P(x, y)$ is the Poisson kernel given in \eqref{eq:33}, the
  assertion easily follows by observing that
  $\| P_\gamma (\cdot, y) \|_{L^1(\R)} = 1$ and $P_\gamma(\cdot, y)$
  approaches a Dirac delta-function for every $y \in (0,1)$ as
  $\gamma \to 0$, together with uniform bounds on the derivatives of
  \begin{align*}
    \theta_{*, \gamma} (x, y) := \int_{-\infty}^\infty P_\gamma(x -
    x', y) \, \bar \theta_* (x') \, dx'
  \end{align*}
  away from $\partial \Sigma$.  }
\end{proof}

\begin{proof}[Proof of item b) of Theorem \ref{th:glargesmall}] We
  first show that as $\gamma \to \infty$, we have
  $\theta_{min,\gamma}(x, 0) \to \bar \theta_0(x)$ for all $x \in \R$,
  where
  \begin{align}
    \label{eq:32}
    \bar\theta_0(x) :=
    \begin{cases}
      \pi, & x < 0, \\
      {\pi \over 2}, & x = 0, \\
      0, & x > 0.
    \end{cases}
  \end{align}
  To see this, for $0 \leq \eps < \frac12$ consider a test function
  \begin{align}
    \label{eq:27}
    \theta_\eps(x, y) := {\pi \over 2} - \arctan \left( {\sinh (\pi (1
    - 2 \eps)
    x) \over \sin (\pi ((1 - 2 \eps) y + \eps))} \right)
  \end{align}
  Notice that $\theta_\eps \in C^\infty(\overline\Sigma)$ for all
  $0 < \eps < \frac12$ and is harmonic in $\Sigma$. Furthermore, in
  this range of $\eps$ we have $\theta_\eps(x, \cdot) \to 0$
  exponentially as $x \to +\infty$ together with all its derivatives,
  and $\theta_\eps(x, \cdot) \to \pi$ exponentially as
  $x \to -\infty$. In particular, $\theta_\eps \in \mathcal A_1$ for
  $0 < \eps < \frac12$, and using symmetry of $\theta_\eps$ we have
  \begin{align}
    \label{eq:28}
    F(\theta_\eps) = \int_{-\infty}^\infty \int_0^{1/2}
    |\nabla \theta_\eps(x, y)|^2 dy \, dx + 2 \gamma \int_{-\infty}^\infty
    \sin^2 \theta_\eps (x, 0) \, dx \\ \notag
    =  \int_{-\infty}^\infty \left( \frac{\pi}{2} - \theta_\eps(x, 0) 
    \right) \partial_y
    \theta_\eps(x, 0) \, dx + 2 \gamma \int_{-\infty}^\infty
    \sin^2 \theta_\eps (x, 0) \, dx,  
  \end{align}
  where to go to the second line we integrated by parts.

  By an explicit computation we get
  \begin{align}
    \theta_\eps(x, 0)
    & = {\pi \over 2} - \arctan \left( {\sinh (\pi (1
      - 2 \eps)
      x) \over \sin (\pi \eps)} \right), \\
    \partial_y \theta_\eps(x, 0)
    & = -\frac{2 \pi  (1-2 \eps) \cos (\pi  \eps)
      \sinh (\pi  (1-2 \eps) x)}{\cos (2 \pi \eps)-\cosh (2 \pi  (1-2
      \eps) x)}, \\
    \sin^2 \theta_\eps(x, 0)
    & = \frac{1}{\csc ^2(\pi  \eps) \sinh ^2(\pi
      (1-2 \eps) x)+1} . 
  \end{align}
  In particular, as $\eps \to 0$ there holds
    \begin{align}
      \theta_\eps(\eps x, 0)
      & \simeq {\pi \over 2} - \arctan x, \\
      \eps \partial_y \theta_\eps(\eps x, 0)
      & \simeq  {x \over 1 + x^2}, \\
      \sin^2 \theta_\eps(\eps x, 0)
      & \simeq \frac{1}{1 + x^2} . 
    \end{align}
    Therefore, by standard asymptotic techniques for integrals we
    obtain as $\eps \to 0$:
    \begin{align}
      \label{eq:30}
      F(\theta_\eps) \simeq \pi \log \eps^{-1} + 2 \pi
      \gamma \eps,
    \end{align}
    and choosing $\eps = \gamma^{-1}$ yields
    \begin{align}
      \label{eq:31}
      2 \gamma \int_{-\infty}^\infty \sin^2 \theta_{min,\gamma}(x, 0) \, dx
      \leq F(\theta_{min,\gamma}) \leq F(\theta_\eps) \leq 2 \pi \log
      \gamma,
    \end{align}
    for all $\gamma$ sufficiently large. Thus in view of monotonicity
    of $\theta_{min,\gamma}$ we have
    $\theta_{min,\gamma}(x, 0) \to \bar\theta_0(x)$ for all $x \in \R$
    as $\gamma \to \infty$.  Furthermore, this convergence is locally
    uniform in ${\overline\R} \setminus \{0\}$.

    Notice that $\theta_0$ defined in \eqref{eq:27} is the harmonic
    extension of $\bar \theta_0$ from $\partial \Sigma$ to
    $\Sigma$. Furthermore, by direct computation
    \begin{align}
      \label{eq:34}
      \theta_0(x, y) = \int_{-\infty}^\infty P(x - x', y) \, \bar
      \theta_0(x') \, dx',
    \end{align}
    where $P(x, y)$ is the Poisson kernel defined in
    \eqref{eq:33}. Notice that
    $P(x, y) \simeq {y \over \pi (x^2 + y^2)}$ for $|x|, |y| \ll 1$,
    and $P(\cdot, y)$ decays exponentially at infinity for all
    $y \in (0,1)$. Therefore, by the representation
    \begin{align}
      \label{eq:35}
      \theta_{min,\gamma}(x, y) = \int_{-\infty}^\infty P(x - x', y) \,
      \theta_{min,\gamma}(x', 0) \, dx' 
    \end{align}
    from \eqref{eq:29} and locally uniform convergence of
    $\theta_{min,\gamma}(x, 0)$ to $\bar\theta_0(x)$ in
    ${\overline\R} \setminus \{0\}$, we conclude that
    {$\theta_{min,\gamma} \to \theta_0$ locally uniformly in
      $\Sigma$} as $\gamma \to \infty$.
\end{proof}

\section{The reduced two-dimensional micromagnetic model}
\label{sec:mag}

We now turn to the analysis of the relationship between the minimizers
of the reduced micromagnetic model introduced in \eqref{epsilone} and
those of the thin film limit model in \eqref{limitingE}.  In what
follows it is understood that both $E_\e$ and $E_0$ are defined for
any function in $L^2_{loc}(\Sigma; \mathbb S^1)$ simply by setting
them equal to $+\infty$ outside $\Sp$ and
$H^1_{l}(\Sigma; \mathbb S^1)$, respectively. Note that
$\{m\in L^2_{loc}(\Sigma; \mathbb S^1):\, E_\e(m)<+\infty\}$ is a
strict subset of $H^1_{l}(\Sigma; \mathbb S^1)$, and the same is true
for $E_0$.

In what follows, we assume that, if not otherwise specified, $C$ is a
positive constant that might depend only on $\gamma$, $h$ and
$\|\eta'\|_{\infty}$. {We also denote by $\mathscr F(f)$ the
  Fourier transform of {$f \in L^2(\R^2)$, defined as}
\begin{align}
  \label{eq:16}
  \mathscr F(f)(\mathbf k) := \int_{\mathbb R^2} e^{-i \mathbf k \cdot
  \mathbf r} f(\mathbf r) \, d^2 r,
\end{align}
for $f \in {L^1(\mathbb R^2) \cap L^2(\R^2)}$.
}

We start with several simple lemmas which will be useful in handling
an {\em unbounded} domain $\Sigma$. We provide proofs for the reader's
convenience. Recall that $[ w ]_{\mathring{H}^{1/2}(\R)}$ refers to
the Gagliardo seminorm of $w$ defined in \eqref{Gsemi}.

\begin{lemma}
  \label{lem:trace}
  There exists $C > 0$ such that for all $w \in {H}^1(\Sigma)$ and all
  $y \in [0,1]$ there holds:
  \begin{align}
    [ w(\cdot, y)]^2_{\mathring{H}^{1/2}{(\R)}} + \| w(\cdot, y) \|_{L^2(\R)}^2 
    & \leq C
      (\|\nabla w \|_{L^2{(\Sigma)}}^2 + \|w
      \|_{L^2{(\Sigma)}}^2), \label{eq:trace12} \\
    \|w \|_{L^2(\Sigma)}^2
    & \leq C (\|\nabla
      w \|_{L^2(\Sigma)}^2+\|w(\cdot, y) \|^2_{L^2(\R)}), \label{eq:poin} 
  \end{align}
  where $w(\cdot, y)$ is understood in the sense of trace.
\end{lemma}

\begin{proof}
  By a reflection with respect to the lines $y = 0$ and $y = 1$
  followed by a multiplication by a smooth cutoff function $\phi(y)$
  that vanishes outside $[-2,2]$, we may extend $w$ to a function
  $\tilde w \in H^1(\R^2)$ such that $w = \tilde w$ in $\Sigma$ and
  $\| \tilde w \|_{H^1(\R^2)} \leq C \| w \|_{H^1(\Sigma)}$ for some
  universal $C > 0$. Therefore, by a density argument we may assume
  that $w \in C^\infty_c(\R^2)$ throughout the rest of the proof.

  To prove \eqref{eq:trace12}, without loss of generality we may
  assume that $y = 0$.  {Letting $\hat w := \mathscr F(w)$} and
  using the Fourier inversion formula, we get
$$
w(x,0) = \frac{1}{(2\pi)^2} \int_{\R^2} e^{i k_1 x } \hat w(k_1, k_2)
\, dk_1 \, dk_2 = \frac{1}{(2\pi)^2} \int_\R e^{i k_1  x}\left( \int_\R \hat
  w(k_1, k_2) \, dk_2 \right) \, dk_1.
$$
Therefore, the one-dimensional Fourier transform $\hat v(k)$ of
$v(x) := w(x, 0)$ equals
$$
\hat v(k) := \int_\R e^{-i k x} w(x, 0) \, dx = \frac{1}{2\pi} \int_\R
\hat w(k, s) \, ds.
$$
Using Cauchy-Schwarz inequality, we thus obtain
$$
\left| \hat v(k) \right|^2 = \frac{1}{(2\pi)^2} \left| \int_\R \hat w(k,
  s)\, ds \right|^2 \leq \frac{1}{(2\pi)^2} \int_\R \frac{ds}{1+ k^2 +
  s^2} \int_\R |\hat w(k, s)|^2 (1+ k^2 + s^2) \, ds.
$$
In turn, using the fact that
$ \int_\R \frac{ds}{1+ k^2 + s^2} = \frac{\pi}{\sqrt{1+ k^2}}$ we
deduce that
$$
(1+ |k|)\left| \hat v(k) \right|^2 \leq 2 \sqrt{1+ k^2}\left| \hat
  v(k) \right|^2 \leq \frac{1}{2\pi} \int_\R |\hat w(k, s)|^2 (1+ k^2
+ s^2) \, ds.
$$
Finally, integrating the above inequality in $k$ and using the Fourier
representations of the $H^1(\R^2)$ and $H^{1/2}(\R)$ norms
{\cite{lieb-loss}} we obtain the desired inequality.

We now turn to \eqref{eq:poin}. By Young's and Jensen's inequalities,
for every $x \in \mathbb R$ and $y' \in [0,1]$ we have
$$
|w(x,y')|^2 = \left| w(x,y) + \int_y^{y'} \partial_s w(x,s)\, ds
\right|^2 \leq 2 |w(x,y)|^2 + 2 \int_0^1 |\partial_s w(x,s)|^2 \, ds.
$$
Therefore, integrating over $x$ and $y'$ yields \eqref{eq:poin}.
\end{proof}

\begin{lemma}\label{lm:tab} For any $a, b>0$ we have
$$
\int_0^{\infty}\frac{e^{-a\sqrt{x^2+b^2}}}{\sqrt{x^2+b^2}}\, dx=K_0(ab)\,,
$$
where $K_0(z)$ is the modified Bessel function of the second kind of
order zero.
\end{lemma}
\begin{proof}
  The identity follows from the integral representation
  \cite[8.432-1]{gradshteyn} of $K_0(z)$ by the change of variable
  $x=b\sinh t$.
\end{proof}

\begin{lemma}\label{lm:tab2}  For any $a>0$ we have
$$
\mathscr{F}\Big(\frac{e^{-a|\mathbf r|}}{2 \pi |\mathbf
  r|}\Big)(\mathbf k)=\frac{1}{\sqrt{a^2+|\mathbf k|^2}}\,.
$$
\end{lemma}
\begin{proof} Denoting by $J_0(z)$ the Bessel function of the first
  kind of order zero, recall that for every $t\in \R$ we have
$$
J_0(t)=\frac{1}{2\pi}\int_0^{2\pi}e^{-i t \cos\theta}\, d\theta\,,
$$   
see \cite[8.411]{gradshteyn}. Therefore,
\begin{align*}
  \mathscr{F}\Big(\frac{e^{-a|\mathbf r|}}{2 \pi |\mathbf
  r|}\Big)(\mathbf k) 
  &=\frac{1}{2\pi}\int_0^{\infty} \left( \int_0^{2\pi}e^{-i r|\mathbf k|
    \cos\theta}\, d\theta\right) e^{-ar}\,dr\\ 
  &=\int_0^{\infty}J_0(r|\mathbf k|)\,\, e^{-ar}\, dr=
    \frac{1}{\sqrt{a^2+|\mathbf k|^2}}\,,
\end{align*}
where the last equality follows from \cite[6.611-1]{gradshteyn}.
\end{proof}

{We now proceed towards the proof of Theorem \ref{th:micromin}.} We
first establish the following result.
\begin{proposition}\label{prop:biscottina} There exists $\e_0>0$ and
  $C>0$ depending only on $\|\eta'\|_\infty$ such that for all
  $\e\in (0, \e_0)$ and {$m\in \Sp$}, the following inequality
  holds: \beq \label{biscottina2}
\begin{split}
  \frac{1}{|\ln \e|} {\int_{\R^2}
    \frac{|\mathscr{F}\big({\Div(\eta_\e m)\big)}|^2}{2 \pi |\mathbf
      k|} \, d^2 k} &\geq 2(1-\beta) \Big(\int_\R m^2_{2} (x,0)\, dx +
  \int_\R m^2_{2} (x,1)\, dx\Big)
  \\
  &- \frac{C}{\beta|\ln \e|}(\|\nabla m\|^2_{L^2(\Sigma)} +
  \|m_{2}\|^2_{L^2(\Sigma)})
\end{split}
\eeq for all $\beta\in (0, 1)$. 
\end{proposition}

\begin{proof} 
  We first note that extending $m$ by zero outside $\Sigma$ we have
  $m\eta_\e \in H^1_{loc} (\R^2; \R^2)$. Furthermore, due to our
  assumptions on $m$ we get
  $\Div\, (\eta_\e m) = \eta_\e \Div\, m + \eta_\e' m_2 \in L^2
  (\R^2)$ and, therefore, its Fourier transform makes sense in
  $L^2(\R^2)$ {\cite{lieb-loss}}. We next fix $0<a \leq 1$ to
  obtain \beq \label{Fdiv} \int_{\R^2}
  \frac{|\mathscr{F}\big({\Div(\eta_\e m)\big)}|^2}{|\mathbf k|} \,
  {d^2 k \over (2 \pi)^2} \geq \int_{\R^2}
  \frac{|\mathscr{F}\big({\Div(\eta_\e m)\big)}|^2}{\sqrt{|\mathbf
      k|^2 + a^2}} \, {d^2 k \over (2 \pi)^2} .  \eeq Thus, using
  Lemma~\ref{lm:tab2}, we have {\cite[Theorem 5.8]{lieb-loss}}
  \beq \label{expriesz} {\int_{\R^2}
    \frac{|\mathscr{F}\big({\Div(\eta_\e m)\big)}|^2}{2 \pi |\mathbf
      k|} \, d^2 k} \geq \int_{\R^2} \int_{\R^2} \Div(\eta_\e
  m)(\mathbf r) \, \Div(\eta_\e m)(\mathbf r') \frac{e^{-a |\mathbf r
      - \mathbf r'|}}{|\mathbf r - \mathbf r'|} \, d^2 r \, d^2
  r'. \eeq The above trick allows us to control the behavior of the
  expression under the integral at infinity and significantly
  simplifies the subsequent analysis of the magnetostatic energy,
  essentially reducing it to the analysis on compact domains.

  We now define \beq \mathcal{K}_a(\mathbf r-\mathbf r') :=
  \frac{e^{-a|\mathbf r-\mathbf r '|}}{| \mathbf r - \mathbf r'|} \eeq
  and proceed to write the integral in the right-hand side of
  \eqref{expriesz} as \beq\label{perterza} \int_{\R^2} \int_{\R^2}
  \Div(\eta_\e m)(\mathbf r) \, \Div(\eta_\e m)(\mathbf r')
  \mathcal{K}_a(\mathbf r -\mathbf r ') \, d^2 r \, d^2 r' = I_1+2
  I_2+I_3, \eeq where \beq \begin{split} I_1 & := \int_{\R^2}
    \int_{\R^2}\eta_\e({\mathbf r}) \Div(m)(\mathbf r) \,
    \eta_\e({\mathbf r}') \Div(m)(\mathbf r') \mathcal{K}_a(\mathbf r
    -\mathbf r
    ') \,  d^2 r \, d^2 r',   \\
    I_2 & := \int_{\R^2} \int_{\R^2} \eta_\e({\mathbf r})
    \Div(m)(\mathbf r) \, (\nabla \eta_\e\cdot m)(\mathbf r')
    \mathcal{K}_a(\mathbf r
    -\mathbf r ') \,  d^2 r \, d^2 r',  \\
    I_3 & := \int_{\R^2} \int_{\R^2} (\nabla \eta_\e\cdot m)(\mathbf
    r) \, (\nabla \eta_\e\cdot m)(\mathbf r') \mathcal{K}_a(\mathbf r
    -\mathbf r ') \, d^2 r \, d^2 r' .
\end{split}
\eeq Using the Fourier representation and Young's inequality, one can
see that 
$$
-\frac{1}{\beta} I_1 - \beta I_3\leq 2 I_2 \leq \frac{1}{\beta} I_1
+\beta I_3,
$$
for any $\beta>0$. Therefore, we have \beq\label{vallissimo}
(1-\beta^{-1}) I_1+ (1-\beta) I_3\leq \int_{\R^2} \int_{\R^2}
\Div(\eta_\e m)(\mathbf r) \, \Div(\eta_\e m)(\mathbf r')
\mathcal{K}_a(\mathbf r -\mathbf r ') \, d^2 r \, d^2 r' \leq
(1+\beta^{-1}) I_1+ (1+\beta) I_3\,.  \eeq Using Young's inequality
for convolutions, we can estimate \beq I_1 \leq
\|\mathcal{K}_a\|_{L^1(\R^2)}\| \Div \, m\|^2_{L^2(\Sigma)} \leq
\frac{4\pi}{a} \| \nabla m\|^2_{L^2(\Sigma)}.  \eeq In order to
estimate $I_3$ we write \beq\label{vallissimissimo} I_3 =
J_1+2J_2+J_3, \eeq where \beq \begin{split} J_1 & :=\frac{1}{\e^2}
  \int_{\R\times[0,\e]}\int_{\R\times[0, \e]} \eta'(y/\e) m_2 (\mathbf
  r) \, \eta'(y'/\e) m_2(\mathbf r') \mathcal{K}_a(\mathbf r
  -\mathbf r ') \, d^2 r \, d^2 r', \\
  J_2 & := \frac{1}{\e^2} \int_{\R\times[0,\e]}\int_{\R\times[1-\e,1]}
  \eta'(y/\e) m_2 (\mathbf r) \, \eta'(y'/\e) m_2(\mathbf r')
  \mathcal{K}_a(\mathbf r
  -\mathbf r ') \, d^2 r \, d^2 r', \\
  J_3 & := \frac{1}{\e^2}
  \int_{\R\times[1-\e,1]}\int_{\R\times[1-\e,1]} \eta'(y/\e) m_2
  (\mathbf r) \, \eta'(y'/\e) m_2(\mathbf r') \mathcal{K}_a(\mathbf r
  -\mathbf r ') \, d^2 r \, d^2 r'.
\end{split}
\eeq We would like to show that $J_2$ is negligible compared to $J_1$
and $J_3$. Using Young's inequality for convolutions, it is
straightforward to see that for $\e$ sufficiently small  \beq\label{J2}
\begin{split}
  J_2 &\leq \frac{C}{\e^2} \int_0^\e \int_{1-\e}^1\int_{\R}\int_{\R}
  |m_2 (x,y)| \, |m_2 (x', y')| \frac{e^{-a|x-x'|}}{\sqrt{|x-x'|^2
      +1/2}} \, dx \, dx' dy \, dy'\,  \\
  &\leq \frac{C}{a\e^2} \int_0^\e \| m_2 (\cdot, y) \|_{L^2(\R)} \,
  dy\int_{1-\e}^1\| m_2 (\cdot, y') \|_{L^2(\R)} \, dy'.
\end{split}
\eeq Hence by Lemma \ref{lem:trace} we have \beq J_2 \leq \frac{C}a
(\|\nabla m_2\|^2_{L^2{(\Sigma)}} + \|m_2\|^2_{L^2{(\Sigma)}} ).  \eeq

It is clear that the integrals $J_1$ and $J_3$ are similar. Therefore,
we provide an estimate for $J_1$ only. We write
\beq\label{vallissimissimissimo}
\begin{split}
  J_1 &=H_1+H_2:=\frac{1}{\e^2} \int_{\R\times(0,\e)}\int_{\R\times(0,
    \e)} \eta'(y/\e) m_2 (x,y) \, \eta'(y'/\e) m_2(x,y')
  \mathcal{K}_a(\mathbf r -\mathbf r') \, \, {d^2 r \, d^2 r'}  \\
  &+ \frac{1}{\e^2}\int_{\R\times(0,\e)}\int_{\R\times(0, \e)}
  \eta'(y/\e) m_2 (x,y) \, \eta'(y'/\e) (m_2(x',y') - m_2(x,y'))
  \mathcal{K}_a(\mathbf r -\mathbf r') \, {d^2 r \, d^2 r'} .
\end{split}
\eeq We now estimate $H_2$ as follows: 
\beq\label{H2}
\begin{split}
  H_2&\leq \frac{C}{\e^2} \int_0^\e \int_0^\e\int_{\R}\int_{\R} |m_2
  (x,y)| {e^{-a|x-x'|}}\, \frac{ |m_2(x',y') - m_2(x,y')| }{|x-x'|}\,
  dx\, dx'\, dy\, dy'\\
  &\leq \frac{C}{\e^2} \int_0^\e \int_0^\e \left(\int_{\R}\int_{\R}
    |m_2 (x,y)|^2 {e^{-2a|x-x'|}}\, dx\, dx'\right)^{\frac12}
  [m_2(\cdot, y')]_{{\mathring{H}}^{\frac12}{(\R)}}\, dy\, dy'\\
  &\leq \frac{C}{\e^2 {\sqrt{a}} } \int_0^\e \|m_2(\cdot,
  y)\|_{L^2{(\R)}}\, dy \int_0^\e[m_2(\cdot,
  y')]_{{\mathring{H}}^{\frac12}{(\R)}}\, dy'\,, 
\end{split}
\eeq where to obtain the second line we used Cauchy-Schwarz
inequality. Using again Lemma \ref{lem:trace} and Young's inequality,
from \eqref{H2} we may conclude that
$$
H_2\leq \frac{C}{\sqrt{a}} (\|\nabla m_2\|^2_{L^2{(\Sigma)}} +
\|m_2\|^2_{L^2{(\Sigma)}} )\,.
$$

Concerning $H_1$, integrating first in $x'$ and using
Lemma~\ref{lm:tab}, we get \beq\label{splitH1}
\begin{split}
  H_1 &=\frac2{\e^2}\int_0^\e\int_0^\e \int_{\R} \eta'(y/\e) m_2 (x,y)
  \,  \eta'(y'/\e) m_2(x,y') K_0(a|y-y'|)\,dx\, dy\, dy' \\ 
  &= 2\int_0^1\int_0^1 \int_{\R} \eta'(y) m_2 (x,\e y) \,  \eta'(y')
  m_2(x,\e y') K_0(a\e |y-y'|)\,dx\, dy\, dy'\\ 
  & = 2 \int_0^1\int_0^1 \int_{\R} \eta'(y) \eta'(y') m^2_2 (x,0) \,
  K_0(a\e |y-y'|)\,dx\, dy\, dy'+2 H_{1,1}\,,
\end{split}
\eeq where \beq\label{H11} H_{1,1}:=\int_0^1\int_0^1 \int_{\R}
\eta'(y) \eta'(y') (m_2 (x,\e y)-m_2 (x,0)) \, (m_2 (x,\e y')+m_2
(x,0)) K_0(a\e |y-y'|)\,dx\, dy\, dy' \,.  \eeq Note that for all $\e$
sufficiently small and $t \in (0,a\e)$ we have
$K_0(t) \leq 2 |\ln(t)|$ and hence
\beq\label{H11est0}
\begin{split}
  H_{1,1}&\leq2 \int_0^1\int_0^1 \int_{\R} |m_2 (x,\e y)-m_2 (x,0)| \,
  |m_2 (x,\e y')+m_2 (x,0)| \, |\ln(a\e|y-y'|)|\, \,dx\, dy\, dy' \\
  &\leq C|\ln (a\e)| \int_{\R} \left(\int_0^1 |m_2 (x,\e y)-m_2
    (x,0)|^2 \, \,dy\right)^{\frac12}\, \left(\int_0^1 |m_2 (x,\e
    y')+m_2 (x,0)|^2 \, \,dy'\right)^{\frac12}\ dx\\
    &+ C\int_{\R} \left(\int_0^1 |m_2 (x,\e y)-m_2
    (x,0)|^2 \, \,dy\right)^{\frac12}\, \left(\int_0^1 |m_2 (x,\e
    y')+m_2 (x,0)|^2 \, \,dy'\right)^{\frac12}\ dx,
\end{split}
\eeq where for the last line we used Young's inequality for
convolutions. It is clear that for $\e$ small enough we can absorb the
expression in the last line to the expression in the second line
above. Moreover, for a.e $x \in \R$ we can estimate
$$
\int_0^1 |m_2 (x,\e y)-m_2
    (x,0)|^2 \, \,dy \leq \e \int_0^1 |\nabla m_2(x,y)|^2\, dy
$$
and
$$
\int_0^1 |m_2 (x,\e y')+m_2 (x,0)|^2 \, \,dy' \leq C \left(
  |m_2(x,0)|^2 +\e \int_0^1 |\nabla m_2(x,y)|^2\, dy \right).
$$
Therefore, using Cauchy-Schwarz and Young's inequalities, by
Lemma~\ref{lem:trace} we obtain \beq\label{H11est}
\begin{split}   
  H_{1,1} &\leq C |\ln (a\e)|\sqrt{\e}\int_{\R}
  \|\nabla m_2(x,\cdot)\|_{L^2(0,1)}\, \left(|m_2(x,0)|^2 +\e
    \|\nabla m_2(x,\cdot)\|_{L^2(0,1)}^2 \right)^{\frac12}\ dx \\ 
  & \leq C |\ln (a\e)|\sqrt{\e} (\|\nabla m_2\|^2_{L^2(\Sigma)} +
  \|m_2\|^2_{L^2(\Sigma)})\,.
\end{split}
\eeq 
Now we note
that, for $\e$ small enough and $t \in (0,a\e)$, we have
$|K_0(t) + \ln(t)| \leq C $ and we get
\begin{align*}
  \left| \int_0^1\int_0^1 \int_{\R} \eta'(y) \eta'(y') m^2_2 (x,0) \,
  K_0(a\e |y-y'|)\,dx\, dy\, {\, dy'} \right.
  & \left.- |\ln \e| \int_\R m^2_2
    (x,0)\, dx \right| \\
  &\leq
    C(|\ln
    a| +1)
    (\|\nabla
    m_2\|^2_{L^2{(\Sigma)}}
    +
    \|m_2\|^2_{L^2{(\Sigma)}}). 
\end{align*}
Finally, combining the above estimates we obtain \eqref{biscottina2},
and we establish the proposition.
\end{proof}

\begin{corollary}\label{cor:biscottina}
  Assume {$m_\e \in \Sp$} and
  $\limsup_{\e \to 0} E_\e(m_\e)<+\infty$. Then
\begin{itemize}
\item $\limsup_{\e\to 0} \|m_{2, \e}\|^2_{L^2(\pa \Sigma)} < \infty$;
\item $\limsup_{\e\to 0} \|m_{2, \e}\|^2_{L^2(\Sigma)} < \infty$.
\end{itemize}
\end{corollary}
\begin{proof}
  Using Proposition~\ref{prop:biscottina} with $\beta=\frac12$ and
  inequality \eqref{eq:poin}, we have
\begin{align*}
  E_\e(m_\e)
  &\geq  {\frac{\gamma}{2}} \Big(\int_\R m^2_{2, \e} (x,0)\, dx +
    \int_\R m^2_{2, \e} 
    (x,1)\, dx\Big) - \frac{C \gamma}{|\ln \e|}(\|\nabla m_{
    \e}\|^2_{L^2{(\Sigma)}} + 
    \|m_{2, \e}\|^2_{L^2{(\Sigma)}}) 
  \\
  &\geq {\frac{\gamma}{2}}  \left(1- \frac{2CC'}{|\ln \e|} \right)  \|m_{2,
    \e}\|^2_{L^2(\pa \Sigma)} - \frac{C {\gamma}  (1+C')}{|\ln
    \e|}\|\nabla m_{\e}\|^2_{L^2{(\Sigma)}}. 
\end{align*}
Recalling that by our assumption $\{\nabla m_\e\}$ is bounded in
$L^2 {(\Sigma)}$ independently of $\e$, we obtain
$$
\limsup_{\e\to 0} \|m_{2, \e}\|^2_{L^2(\pa \Sigma)} < \infty.
$$
The second conclusion now follows again by \eqref{eq:poin}.
\end{proof}

We now prove the $\liminf$ and $\limsup$ inequalities for the
magnetostatic energy term.
\begin{proposition}\label{prop:biscottina-2}
  Assume that {$m_\e \in \Sp$} and that
  $\limsup_{\e\to 0}E_\e(m_\e)<+\infty$. If $m_\e \wto m$ weakly in
  $H^1_{l}(\Sigma; \mathbb S^1)$ then \beq\label{seconda} \liminf_{\e
    \to 0} \frac{1}{|\ln \e|} {\int_{\R^2}
    \frac{|\mathscr{F}\big({\Div(\eta_\e m)\big)}|^2}{2 \pi |\mathbf
      k|} \, d^2 k}\geq 2 \int_\R m^2_2(x,0)\, dx \ +\ 2 \int_\R
  m^2_2(x,1) \, dx\,.  \eeq Moreover, for any
  ${m\in H^1_{l}(\Sigma; \mathbb S^1)}$ with $E_0 (m) <+\infty$
  {such that} the set
  $\{\mathbf r \in \Sigma:\, {m_2}(\mathbf r)\neq 0\}$ {is}
  essentially bounded we have \beq\label{terza} \limsup_{\e \to 0}
  \frac{1}{|\ln \e|} {\int_{\R^2}
    \frac{|\mathscr{F}\big({\Div(\eta_\e m)\big)}|^2}{2 \pi |\mathbf
      k|} \, d^2 k} \leq 2 \int_\R m^2_2(x,0)\, dx \ +\ 2 \int_\R
  m^2_2(x,1) \, dx\,.  \eeq
\end{proposition}
\begin{proof}
  Using Proposition~\ref{prop:biscottina}, we can take the limit as
  $\e \to 0$ in \eqref{biscottina2}. Employing
  Corollary~\ref{cor:biscottina} and the fact that
$$
\liminf_{\e\to 0} \Big(\int_\R m^2_{2, \e} (x,0)\, dx + \int_\R
m^2_{2, \e} (x,1)\, dx\Big)\geq \int_\R m^2_{2} (x,0)\, dx + \int_\R
m^2_{2} (x,1)\, dx,
$$
we obtain
$$
\frac{1}{|\ln \e|} {\int_{\R^2}
  \frac{|\mathscr{F}\big({\Div(\eta_\e m)\big)}|^2}{2 \pi |\mathbf k|}
  \, d^2 k} \geq 2(1-\beta) \Big(\int_\R m^2_{2} (x,0)\, dx + \int_\R
m^2_{2} (x,1)\, dx\Big).
$$
Finally, taking the limit as $\beta \to 0$ we obtain \eqref{seconda}.

We are left with showing the {second} part of the statement.  We note
that by our assumptions on $m$ there exists $R> {1}$ such that
\beq\label{BRuno} \{\nabla m\neq 0\}\subset Q_R, \eeq where
$Q_R:=(-R, R)\times (0,1) \subset \Sigma$. Moreover, since by
assumption $m_2 \in L^2(\Sigma)$, we have \beq\label{BR} m_2 = 0
\text{ a.e.  in } \Sigma\setminus Q_R\,.  \eeq We start by splitting
the magnetostatic energy as in \eqref{perterza}, with $\mathcal{K}_a$
replaced by the original kernel
{$\mathcal{K}_0(\mathbf r)=\frac{1}{|\mathbf r|}$ after passing to
  the limit $a \to 0$}.  With the same notation for $I_1$, $I_2$,
$I_3$ (and taking $a \to 0$), it is straightforward to see that the
second inequality in \eqref{vallissimo} still holds.

Using the Young's inequality for convolutions, we can estimate
\beq\label{I1bis} I_1 \leq \|\mathcal{K}_0\|_{L^1(\Sigma\cap
  Q_{2R})}\| {\Div \, m}\|^2_{L^2(\Sigma)} {\leq C \| \nabla
  m\|_{L^2(\Sigma)}^2} \,, \eeq {for some $C > 0$ depending only on
  $R$.} We now proceed by splitting $I_3$ as
$$
I_3 = J_1+2J_2+J_3,
$$
with the same notation as in \eqref{vallissimissimo} (and with
$a=0$). Using \eqref{BR}, the estimate in \eqref{J2} (with $a=0$) may
be replaced by
$$
J_2\leq \frac{1}{\e^2} \Bigl\|
\tfrac{1}{\sqrt{|\cdot|^2+\frac12}}\Bigr\|_{L^1({-2R, 2R})}
\int_0^\e \| m_2 (\cdot, y) \|_{L^2{(\Sigma)}} \, dy\int_{1-\e}^1\|
m_2 (\cdot, y') \|_{L^2{(\Sigma)}} \, dy'\,,
$$
and by \eqref{eq:trace12}, we obtain \beq\label{J2bis} J_2\leq {C}
|\ln R| (\|\nabla m_2\|^2_{L^2{(\Sigma)}} + \|m_2\|^2_{L^2{(\Sigma)}}
)\,.  \eeq

Taking into account \eqref{BR}, we can split $J_1$ as
 
 \beq\label{vallissimissimissimobis}
\begin{split}
  J_1 &=H_1+H_2:=\frac{1}{\e^2} \int_{(-R, R)\times(0,\e)}\int_{(-R,
    R)\times(0, \e)} \eta'(y/\e) m_2 (x,y) \,  \eta'(y'/\e) m_2(x,y')
  \frac{1}{|\mathbf r -\mathbf r'|} \, \, {d^2 r \, d^2 r'}  \\
  &+ \frac{1}{\e^2} \int_{(-R, R)\times(0,\e)}\int_{(-R, R)\times(0,
    \e)} \eta'(y/\e) m_2 (x,y) \, \eta'(y'/\e) (m_2(x',y') -
  m_2(x,y')) \frac{1}{|\mathbf r -\mathbf r'|} \, {d^2 r \, d^2 r'} .
\end{split}
\eeq We can estimate $H_2$ as in \eqref{H2}, with $a=0$ but taking
advantage of the fact that \eqref{BR} holds, to get
 $$
 H_2\leq \frac{C\sqrt{R}}{\e^2} \int_0^\e \|m_2(\cdot, y)\|_{L^2(\R)}\, dy
 \int_0^\e[m_2(\cdot, y')]_{{\mathring{H}}^{\frac12}{(\R)}}\,
 dy'.
 $$
 In turn, using \eqref{eq:trace12} we obtain \beq \label{H2bis}
 H_2\leq {C\sqrt{R}} (\|\nabla m_2\|^2_{L^2{(\Sigma)}} +
 \|m_2\|^2_{L^2{(\Sigma)}} )\,.  \eeq
 
 Concerning $H_1$, by integrating first with respect to $x'$ over
 $(-R, R)$, we can argue similarly to \eqref{splitH1} and write
 \beq\label{splitH1bis} H_1 = H_1'+ H_{1,1}:= \int_0^1\int_0^1
 \int_{-R}^R \eta'(y) \eta'(y') m^2_2 (x,0) \,
 \int_{-R}^R\frac{dx'}{\sqrt{|x-x'|^2+\e^2|y-y'|^2}}\,dx\, dy\, dy'+
 H_{1,1}\,, \eeq where $H_{1,1}$ is defined as in \eqref{H11}, with
 $K_0(a\e |y-y'|)$ replaced by
 $ \int_{-R}^R\frac{dx'}{\sqrt{|x-x'|^2+\e^2|y-y'|^2}}$ and with the
 integral in $dx$ running over $(-R, R)$ instead of $\R$. Observe that
$$
\int_{-R}^R\frac{dx'}{\sqrt{|x-x'|^2+\e^2|y-y'|^2}} \leq
2\int_{0}^{2R}\frac{ds}{\sqrt{s^2+\e^2|y-y'|^2}}\,.
$$
By computing explicitly the {right-hand side}, one can easily see
that there exists a constant $C=C(R)>0$ such that for $\e$ small
enough
$$
\int_{-R}^R\frac{dx'}{\sqrt{|x-x'|^2+\e^2|y-y'|^2}} \leq
C|\ln(\e|y-y'|)|\,.
$$ 
With this estimate at hand, we can now argue similarly to
\eqref{H11est} to obtain \beq\label{H11estbis} H_{1,1}\leq C |\ln
\e|\sqrt{\e} (\|\nabla m_2\|^2_{L^2} + \|m_2\|^2_{L^2})\,.  \eeq

It remains to {estimate} $H_1'$. To this aim, we observe that for
any fixed $\de\in (0, R)$  we have
\begin{align*}
  \int_{-R}^R\frac{dx'}{\sqrt{|x-x'|^2+\e^2|y-y'|^2}}
  & {\leq} \int_{(-R,
    R)\cap\{|x-x'|>\de\}}\frac{dx'}{\sqrt{|x-x'|^2+\e^2|y-y'|^2}}+
    2\int_0^\de\frac{ds}{\sqrt{s^2+\e^2|y-y'|^2}}\\ 
  &\leq \frac{2R}\de+ 2\int_0^\de\frac{ds}{\sqrt{s^2+\e^2|y-y'|^2}},
\end{align*}
from which we easily deduce that
 \beq\label{H1primo}
\begin{split}
  H_1' &\leq \frac{C(R)}\de\|m_2(\cdot, 0)\|^2_{L^2}+ 2
  \int_0^1\int_0^1 \int_{-R}^R \eta'(y) \eta'(y') m^2_2 (x,0) \,
  \int_0^\de\frac{ds}{\sqrt{s^2+\e^2|y-y'|^2}}\,dx\, dy\, dy'\\
  &\leq \frac{C'(R)}\de\|m_2(\cdot, 0)\|^2_{L^2}+ 2 \int_0^1\int_0^1
  \int_{-R}^R \eta'(y) \eta'(y') m^2_2 (x,0) |\ln(\e|y-y'|)|\,dx\,
  dy\, dy'\\
  &\leq \Big(\frac{C'(R)}\de+ {C''}\Big)\|m_2(\cdot, 0)\|^2_{L^2}+
  2|\ln \e| \int_{-R}^R m^2_2 (x,0) \,dx\,,
\end{split}
\eeq 
provided that $\delta$ is sufficiently small. Note that the second inequality can
be easily obtained by computing explicitly the innermost integral and
by estimating the result (for $\de$ sufficiently small) with
$ |\ln(\e|y-y'|)|+\tilde C$ for a suitable $\tilde C>0$, while the
third inequality can be obtained by integrating out
$|\ln|y-y'||$. Combining \eqref{I1bis} and
\eqref{J2bis}--\eqref{H1primo} and the completely analogous estimates
for $J_3$, we obtain \eqref{terza}.
\end{proof}

As a straightforward consequence of \eqref{biscottina2} in
Proposition~\ref{prop:biscottina} we have the following
corollary.
\begin{corollary}\label{cor:biscottina2}
  For any $M>0$, let $\e_0>0$ be as in
  Proposition~\ref{prop:biscottina}. Then, for all $\e\in (0, \e_0)$
  and {$m\in \Sp$} such that $E_\e (m) \leq M$ we have
$$
E_0(m)\leq C {M},
$$
for a suitable $C>0$ independent of $\e$, $M$ and $m$.
\end{corollary}


Before analysing the asymptotic behavior of $E_\e$ as $\e\to 0$, let
us show that for $\e>0$ small enough $E_\e$ admits a global minimizer
in the class of magnetizations with nontrivial ``winding''.

\begin{proposition}\label{prop:mineps}
  There exists $\e_1>0$ such that for all $\e\in (0, \e_1)$ the
  following problem: \beq\label{mineps} \min\left\{E_\e(m):\,
    m=(\cos\theta, \sin\theta)\in {\Sp} \text{ with }\theta \text{
      satisfying \eqref{lemma1} for some }k_1,\, k_2\in \Z,\, k_1\neq
    k_2\right\}\footnote{Note that here we are not fixing $k_1$ and
    $k_2$, but we are also minimizing with respect to $k_1$ and $k_2$,
    with $k_1\neq k_2$.}  \eeq admits a solution.
\end{proposition}
\begin{proof} Denote by $i_\e$ the infimum of the problem in
  \eqref{mineps} and observe that 
$$
M:={1 + } \sup_{\e\in (0, \frac12)}i_\e <+\infty\,.
$$
Indeed, it is enough to consider a fixed test function
$m=(\cos\theta, \sin\theta)\in {\Sp}$ such that {the set}
$\{{\mathbf r} \in \Sigma:\, {m_2}({\mathbf r})\neq 0\}$ is
bounded and $\theta$ satisfies the proper boundary conditions at
infinity. By Proposition~\ref{prop:biscottina-2} we easily get
$$
i_\e\leq E_\e(m)\leq C\,.
$$

Let $\e_1:= \e_0({M})$, where $\e_0({M})$ is {as} in in
Proposition~\ref{prop:biscottina}, and {let} $\e\in (0, \e_1)$.
{If} $m_n=(\cos\theta_n, \sin\theta_n)$ {is} a minimizing
sequence for \eqref{mineps}, by Corollary \ref{cor:biscottina2}
we have
$$
F(\theta_n)\leq C
$$
for every $n$ {large enough,} with $C$ independent of $n$.  Set
\beq\label{defbar} \bar\theta_n(x):=\int_0^1\theta(x, y)\, dy\,.  \eeq
By shifting and flipping the $\theta_n$'s if needed, in view also of
Lemma~\ref{lm:zagara} we may assume that
$$
\lim_{x\to-\infty}\bar\theta_n(x)=k_n\pi\text{ and }
\lim_{x\to+\infty}\bar\theta_n(x)=0
$$ 
for some $k_n\in {-\mathbb N}$.

Observe also that
$$
\sup_{n}\int_{\R }|\bar\theta'_n|^2\, dx<+\infty \text{ and }
\sup_{n}\int_{\Sigma}|\nabla\theta_n|^2\, dx<+\infty\,.
$$
By replacing $\theta_n$ with $\theta_n(\cdot-\tau_n, \cdot)$,
$\bar \theta_n$ with $\bar \theta_n(\cdot-\tau_n)$ and not renaming
the minimizing sequence, we can use the continuity of
$\bar \theta_n(x)$ and conditions at infinity to make sure that
$\bar \theta_n(0) =-\frac{\pi}{2}$.  It follows that
$$
|\bar \theta_n(x) - \bar \theta_n(0)| = \left| \int_0^x
  \bar\theta_n'(s) \, ds \right| \leq \sqrt{x}\, \| \bar\theta_n'
\|_{L^2(\R)} \leq C \sqrt{x}.
$$
Therefore, we have that $\bar \theta_n$ is bounded in
$L^2_{loc} (\R)$. Employing the Poincare inequality we deduce that
$\theta_n$ is bounded in $L^2_{loc}(\Sigma)$.  Thus we may apply
\cite[Lemma 1]{doering14} to deduce that there exists
$\theta_\infty\in H^1_{l}(\Sigma)$ and a subsequence (not relabelled)
such that $\theta_n\wto \theta_\infty$ weakly in $H^1_{l}(\Sigma)$,
$\bar\theta_n\wto \bar\theta_\infty$ weakly in $H^1_{l}(\R)$, and
\beq\label{felice} \bar\theta_\infty(0) = -{\pi \over 2}, \quad
\limsup_{x\to-\infty}\bar\theta_\infty(x)\leq -\frac{\pi}2\quad \text{
  and } \quad \liminf_{x\to+\infty}\bar\theta_\infty(x)\geq
-\frac{\pi}2\,, \eeq Furthermore, testing $\theta_n$ with
$\phi(x, y) = \psi(x)$, where $\psi \in C^\infty_c(\R)$, and passing
to the limit, it is easy to see that
$\bar\theta_\infty(x) = \int_0^1 \theta_\infty(x, y) \, dy$ for a.e.
$x \in \R$. In addition, using weak lower semicontinuity of the energy
$F$ we also have
 $$
 F(\theta_\infty)\leq C.
 $$
 In turn, by Lemma~\ref{lm:zagara} there exist $j_1, j_2\in \Z$ such
 that $\theta_\infty$ satisfies \eqref{lemma1}, with $k_1$, $k_2$
 replaced by $j_1$, $j_2$, respectively, and
 $$
 \lim_{x\to-\infty}\bar\theta_\infty(x)=j_1\pi, \qquad
 \lim_{x\to+\infty}\bar\theta_\infty(x)=j_2\pi\,.
 $$
 Taking into account \eqref{felice}, it is clear that $j_1\neq
 j_2$. It is now easy to check that
 $m_\infty:=(\cos\theta_\infty, \sin\theta_\infty)$ is a solution to
 \eqref{mineps}.  
\end{proof}

We are now ready to state the main $\Gamma$-convergence result showing
that \eqref{limitingE} is the limiting energy of \eqref{epsilone}.

\begin{theorem}\label{th:gamma} {Let $\gamma > 0$ and $h \geq 0$, and
    let $E_\e$ and $E_0$ be defined by \eqref{epsiloneF} and
    \eqref{limitingE}, respectively, on
    $L^2_{loc}(\Sigma; \mathbb S^2)$. Then the following two
    statements are true:}
\begin{itemize}
\item[(i)] \emph{($\Gamma$-$\liminf$ inequality)} Let
  {$m_\e \in \Sp$} and $m_\e\to m$ strongly in
  $L^2_{loc}(\Sigma; {\R^2})$ {as $\e \to 0$}. Then
  \beq\label{liminf} \liminf_{\e\to 0} E_\e(m_\e)\geq E_0(m)\,.  \eeq
\item[(ii)] \emph{($\Gamma$-$\limsup$ inequality)} Let
  {$m\in H^1_{l}(\Sigma; \mathbb S^1)$} be such that
  $E_0(m)<+\infty$. {Then there exists $m_\e \in \Sp$ such that
    $m_\e \to m$ in $L^2_{loc}(\Sigma; \R^2)$ as $\e \to 0$ and}
$$
\limsup_{\e\to 0} E_\e(m_\e)\leq E_0(m)\,.
$$
{Furthermore, if $\theta$ and $\theta_\e$ are such that
  $m = (\cos \theta, \sin \theta)$ and
  $m_\e = (\cos \theta_\e, \sin \theta_\e)$, then for every $\e$
  sufficiently small we have
$$
\lim_{x\to - \infty}\|\theta_\e(x,\cdot)- k_1 \pi \|_{L^{2}(0,1)}=0
\quad \text{and} \quad \lim_{x\to - \infty}\|\theta_\e(x,\cdot)- k_2
\pi \|_{L^{2}(0,1)}=0 ,
$$
where $k_1, k_2 \in \mathbb Z$ are as in \eqref{lemma1}.  }
\end{itemize}
\end{theorem}

\begin{proof}
  Let us first prove the $\Gamma$-liminf inequality. If
  $\liminf_{\e \to 0} E_\e(m_\e)=+\infty$ there is nothing to
  prove. Hence we may assume without loss of generality that (after
  passing to a subsequence)
$$
\liminf_{\e {\to 0}} E_\e(m_\e)=\lim_{\e {\to 0}} E_\e(m_\e)<+\infty.
$$
Then, in particular,
${\lim\sup}_{\e {\to 0}} \|\nabla m_\e\|_{L^2 {(\Sigma)}}<+\infty$ and
thus $m_\e\wto m \in H^1_{l}(\Sigma; \mathbb S^1)$ weakly in
$H^1_{l}(\Sigma; {\R^2})$. Inequality \eqref{liminf} then follows from
the Proposition~\ref{prop:biscottina-2} (see \eqref{seconda}) and from
the lower semicontinuity of the local terms in the energies.

Let us now establish the upper bound.  Let $m$ and $\theta$ be as in
the second part of the statement. {Then by Lemma \ref{lem:trace} we
  have $m \in \Sp$, and} by Lemma~\ref{lm:zagara} there exists
{$k_1, k_2 \in \Z$ such that \eqref{lemma1} holds true.}  Now, {arguing as
  in the proof of \eqref{finalclaim}} one can construct a sequence
$\{\theta_n\}$ with the following properties:
\begin{itemize}
\item[i)] for every $n$ there exists $M_n>0$ such that 
$$
\theta_n(x, y)= {k_1 \pi \text{ if $x\leq -M_n$ and } \theta_n(x, y)=k_2 \pi
\text{ if $x\geq M_n$, }}
$$
\item[ii)] {$m_n\to m \in L^2_{loc}(\Sigma; \R^2)$},
\item[iii)] setting $m_n:=(\cos\theta_n, \sin\theta_n)$, {we have}
$$
E_0(m_n)=F(\theta_n)\to E_0(m)=F(\theta)\text{ as }n\to\infty\,.
$$
\end{itemize}
Therefore, by a standard diagonal argument it is enough to prove the
upper bound under the following additional assumption: there exists
$M>0$ such that
$$
{\theta(x, y)= k_1 \pi \text{ if $x\leq -M$ and } \theta(x, y)=k_2 \pi \text{ if
  $x\geq M$.}}
$$
Under such an assumption, the conclusion follows simply by taking
$m_\e= m$ for all $\e$ and observing that
$$
\limsup_{\e\to 0}E_\e(m)\leq E_0(m)\,,
$$
 thanks to Proposition~\ref{prop:biscottina-2} (see \eqref{terza}). 
\end{proof}

\begin{corollary}\label{cor:infkeps}
Let $k \in \Z \setminus\{0\}$.
Then
$$
\lim_{\e \to 0} \inf_{{m \in \mathcal A_k^0}} E_\e({m}) =
\inf_{\theta \in \mathcal A_k} F (\theta),
$$
{where
  $\mathcal A_k^0 := \{ m \in \Sp \, : \, m = (\cos \theta, \sin
  \theta), \ \theta \in \mathcal A_k \}$.}
\end{corollary}
\begin{proof}
  For simplicity of the presentation we provide the proof for
  $k \in \N$ only. Using Theorem~\ref{th:gamma}, we know that for any
  fixed $\bar \theta \in \mathcal A_k$ such that
  $F(\bar \theta)<+\infty$ we may find
  {$m_\e = (\cos \theta_\e, \sin \theta_\e) \in \mathcal A_k^0$}
  such that $\theta_\e\wto \bar \theta$ weakly in $H^1_{l}(\Sigma)$
  and $E_\e ({m_\e}) \to F(\bar \theta)$. Thus,
$$
\limsup_{\e \to 0} \inf_{m \in \mathcal A_k^0} E_\e({m}) \leq
\lim_{\e \to 0} E_\e({m_\e}) = F(\bar \theta).
$$
By the arbitrariness of $\bar \theta\in \mathcal A_k$, we obtain
$$
\limsup_{\e \to 0} \inf_{\theta \in \mathcal A_k} E_\e(\cos\theta,
\sin\theta) \leq \inf_{\theta \in \mathcal A_k} F(\theta).
$$

For the reverse inequality, let
{$m_\e = (\cos \theta_\e, \sin \theta_\e) \in \mathcal A_k^{0}$}
be a sequence such that \beq\label{infimizing} \lim_{{\e \to 0}}
E_{\e}({m_\e}) = \liminf_{\e\to 0}\inf_{m \in \mathcal A_k^0}
E_\e({m}).  \eeq Then for {$\e$ small} enough we may use
Corollary~\ref{cor:biscottina} to get
$$
{\limsup_{\e \to 0}} \big(\|\nabla \theta_{\e}\|_{L^2(\Sigma)} +
\|\sin\theta_{\e}\|_{L^2(\Sigma)} +
\|\sin\theta_{\e}\|_{L^2(\partial\Sigma)}\big)<+\infty\,.
$$
On the other hand, for any $\beta\in (0,1)$, using inequality
\eqref{biscottina2} from Proposition~\ref{prop:biscottina} and
denoting by $C$ a positive constant independent of $j$ and $\beta$
(that may change from inequality to inequality) we have
\begin{multline*}
  E_{\e}({m_\e}) \geq F(\theta_{\e}) - \frac{C}{\beta |\ln
    {\e}| } \left(\|\nabla \theta_{\e}\|^2_{L^2(\Sigma)} + \|
    \sin\theta_{\e}\|^2_{L^2(\Sigma)}
  \right) - C \beta \| \sin\theta_{\e} \|^2_{L^2 (\partial \Sigma)}  \\
  \geq \inf_{\theta \in \mathcal A_k} F(\theta) - \frac{C}{\beta |\ln
    {\e}| } - C \beta.
\end{multline*}
Taking the limit as ${\e \to 0}$ and recalling \eqref{infimizing},
we obtain
$$
\liminf_{\e \to 0} \inf_{m \in \mathcal A_k^0} E_\e({m}) \geq
\inf_{\theta \in \mathcal A_k} F(\theta) - C \beta.
$$
The conclusion then follows from the arbitrariness of $\beta$.
\end{proof}

\begin{corollary}\label{cor:mincon} Let $\e \to 0$ and let ${\{m_\e\}}$
  be a sequence of minimizers for problem \eqref{mineps}. Then, after
  suitable translations in the $x$-variable and up to a subsequence
  (not relabelled), we have
  {$m_\e\to m_0 \in H^1_{l}(\Sigma; \mathbb S^1)$ strongly} in
  $H^1_{l}(\Sigma{; \R^2})$, where $m_0$ is a solution to
  \eqref{mineps0}.  Moreover,
$$
\lim_{\e \to 0} E_\e(m_\e) = E_0(m_0).
$$ 
\end{corollary}

\begin{proof}
  Note that by Corollary~\ref{cor:biscottina} we have \beq
  \limsup_{\e\to 0} E_0(m_\e)<+\infty\,.  \eeq In turn, by
  Lemma~\ref{lm:zagara}, without loss of generality, we may associate
  to each $m_\e$ a phase function $\theta_\e$ satisfying
  \eqref{thetau}.

We may now argue exactly as in Proposition~\ref{prop:mineps} (with
$\theta_\e$ in place of $\theta_n$) to deduce the existence of
$\theta_0\in H^1_{l}(\Sigma)$ and of $j_1$, $j_2\in \Z$, $j_1\neq j_2$
such that \eqref{lemma1} holds with $\theta$, $k_1$, $k_2$ replaced by
$\theta_0$, $j_1$, $j_2$, respectively, and
$$
\theta_\e\wto \theta_0\quad\text{weakly in
}H^1_{l}(\Sigma)\,,
$$
up to a subsequence (not relabelled).  Set now
$m_0:=(\cos\theta_0, \sin\theta_0)$. The fact that $m_0$ is a solution
of \eqref{mineps0} and the convergence of energies follows from a
standard $\Gamma$-convergence argument in view of
Theorem~\ref{th:gamma}. In turn, the convergence of energies implies
strong convergence of $m_\eps$ in $H^1_{l}(\Sigma; \R^2)$. 
 \end{proof}


 Corollary~\ref{cor:mincon} combined with
 Corollaries~\ref{cor:mineps0}, \ref{cor:infk0} and \ref{cor:infkeps}
 easily yields that for $\e$ small enough the minimization in
 \eqref{mineps} is achieved by at most single winding. Precisely,
 we have:
  \begin{corollary}\label{cor:mineps}
    There exists $\e_1>0$ such that for $\e\in (0, \e_1)$ any
    minimizer $m_\e=(\cos\theta_\e,\sin\theta_\e)$ of \eqref{mineps}
    is such that $\theta_\e$ satisfies \eqref{lemma1} for some
    $k_1(\e), k_2(\e)\in \Z$, with $|k_1(\e)-k_2(\e)|=k$, where $k=1$
    if $h=0$ or $k=2$ if $h>0$. {Moreover, after suitable
      translations we have
      \begin{align}
        \label{eq:42}
        \text{sgn}(k_1(\e) - k_2(\e)) (\theta_\e - k_2(\e) \pi) \to
        \theta_{min} \ \text{strongly in} \ H^1_{l}(\Sigma) \
        \text{as} \ \eps \to 0, 
      \end{align}      
      where $\theta_{min}$ is the unique (up to translations)
      minimizer from Theorem \ref{t:exist}.}
  \end{corollary}
\begin{proof}
  We provide a proof for $h=0$ only. Let $\e>0$ be small enough and
  $m_\e=(\cos\theta_\e, \sin\theta_\e)$ be a minimizer of
  \eqref{mineps}. Using Corollary~\ref{cor:biscottina2} and
  Lemma~\ref{lm:zagara}, we know that there exist
  $k_1(\e), k_2(\e) \in \N$ such that \beq
  \lim_{x\to-\infty}\|\theta_\e(x,\cdot)-k_1(\e)\pi\|_{L^{2}(0,1)}=0
  \text{ and }
  \lim_{x\to+\infty}\|\theta_\e(x,\cdot)-k_2(\e)\pi\|_{L^{2}(0,1)}=0\,.
  \eeq Employing Corollary~\ref{cor:mincon}, we also know that (after
  a suitable translation) $m_\e \to m_0$ strongly in
  $H^1_{l} (\Sigma)$ {for a subsequence of $\e \to 0$}, where $m_0$
  is a minimizer of \eqref{mineps0}. We want to show that
  $|k_1(\e) - k_2(\e)| \to 1$. Assume this is not the case, then there
  exists a further subsequence $\e_k \to 0$ such that either: (a)
  $|k_1(\e_k) - k_2(\e_k)| \to n \in \Z_+ \setminus \{1\}$ or (b)
  $|k_1(\e_k) - k_2(\e_k)| \to \infty$.

  In case (a), we see that there exists $\e_1>0$ such that for all
  $\e_k <\e_1$ we have $|k_1(\e_k) - k_2(\e_k)| = n$ and therefore
  (after a suitable shift of $\theta_{\e_k}$ by $k_2(\e_k) \pi$) we
  obtain $m_{\e_k} (\cos \theta_{\e_k}, \sin \theta_{\e_k})$ with
  $\theta_{\e_k} \in \mathcal A_n$. Since $m_{\e_k}$ minimizes
  \eqref{mineps}, we cannot have $n=0$. Furthermore, since $m_{\e_k}$
  is a minimizer of \eqref{mineps} we obtain
  $\lim_{k \to \infty} E_{\e_k}(m_{\e_k}) = \lim_{k \to \infty}
  \inf_{m \in \mathcal A_n} E_{\e_k} (m)$. Using
  Corollary~\ref{cor:infkeps} and Corollary~\ref{cor:infk0}, we obtain
  that
$$
\lim_{k \to \infty} E_{\e_k}(m_{\e_k}) = \inf_{\theta \in \mathcal
  A_n} F(\theta)= n F(\theta_{min}) > F(\theta_{min})>0,
$$
contradicting the convergence of energies in
Corollary~\ref{cor:mincon}.

In case (b), we assume without loss of generality that $k_2(\e)=0$ and
$k_1(\e)>0$.  We note that since $k_1(\e) \to \infty$, the function
$\bar\theta_{\e_k} = \min \{ n \pi, \theta_{\e_k}\}$ for any fixed
$n \in \N$ yields $F(\theta_{\e_k}) \geq F(\bar \theta_{\e_k})$. Using
Corollary~\ref{cor:mincon}, we know that $E_{\e_k} (m_{\e_k})<C$ and
therefore by employing Proposition~\ref{prop:biscottina} and
Corollary~\ref{cor:biscottina} we obtain
$$
\liminf_{k \to \infty} E_{\e_k} (m_{\e_k}) \geq (1-\beta) \liminf_{k
  \to \infty} F(\theta_{\e_k}) \geq (1-\beta) \liminf_{k \to \infty}
F(\bar\theta_{\e_k})
$$
for any $\beta \in (0,1)$. Finally, taking $n=2$ and noting that
$\bar\theta_{\e_k} \in \mathcal A_n$ we deduce, using
Corollary~\ref{cor:infk0}, that
$\liminf_{k \to \infty} E_{\e_k} (m_{\e_k}) \geq 2(1-\beta)
F(\theta_{min}) > F(\theta_{min})$ for $\beta$ small enough, and again
we have a contradiction with the convergence of energies in
Corollary~\ref{cor:mincon}.

{Finally, strong convergence of
  $\text{sgn}(k_1(\e) - k_2(\e)) (\theta_\e - k_2(\e) \pi)$ to
  $\theta_{min}$ in $H^1_{l}(\Sigma)$ follows from
  Corollary~\ref{cor:mincon} and uniqueness of the minimizer of the
  limit problem. }
\end{proof}

{
  \begin{proof}[Proof of Theorem \ref{th:micromin}]
    To conclude, the assertion of Theorem \ref{th:micromin} is an
    immediate consequence of Corollary \ref{cor:mineps}.
  \end{proof}
}
 
 \appendix

 \section{Poisson kernel, Dirichlet-to-Neumann map and Dirichlet
   energy on a strip}
 \label{s:app}

 Here we provide the details of the computation that leads to
 \eqref{eq:9} and \eqref{eq:35} for the convenience of the reader (see
 also, e.g., \cite{widder61}). We start by noting that the symmetry of
 minimizers with respect to the $y = \frac12$ line follows by a
 standard reflection argument (see also Corollary
 \ref{cor:decay}). Hence we may assume that
 \begin{align}
   \label{eq:13}
   \mathcal F(\theta) = \int_{\mathbb R} \int_0^{\frac12} |\nabla \theta|^2 dy
   \, dx + 2 \gamma \int_{\mathbb R} \sin^2 \bar \theta \, dx,
 \end{align}
 with $\theta$ satisfying
 \begin{align}
   \label{eq:14}
   \theta(x, 0) = \bar \theta(x) \qquad \text{and} \qquad
   \partial_y \theta\left(x, \tfrac12 \right) = 0 \qquad \forall x \in
   \mathbb R.  
 \end{align}

 We next minimize $\mathcal F$ in \eqref{eq:13} with respect to
 $\theta$ satisfying the first of \eqref{eq:14} with a fixed
 $\bar\theta \in C^\infty(\mathbb R)$ obeying \eqref{eq:36}. This
 amounts to choosing $\theta$ to be the harmonic extension of
 $\bar \theta$ in $\R \times (0, \frac12)$ that satisfies the boundary
 conditions in \eqref{eq:14}. Notice that by standard elliptic
 regularity theory we have $\theta \in C^m(\R \times [0,1])$ for every
 $m \in \mathbb N$ under our assumption on $\bar \theta$, and
 $\theta(x, y)$ decays exponentially to the respective limits together
 with all its derivatives as $x \to \pm \infty$.

 Let
 \begin{align}
   \label{eq:17}
   \hat \theta(k, y) := \int_{\mathbb R} e^{-i k x} \theta(x, y) \, dx
 \end{align}
 be the one-dimensional Fourier transform of $\theta$ in the
 $x$-variable, understood in the sense of tempered distributions. The
 function $\hat \theta(k, y)$ solves
 \begin{align}
   \label{eq:18}
   \partial^2_y  \hat \theta (k, y) - k^2 \hat \theta (k, y) =
   0, \qquad \partial_y \hat \theta \left( k, \tfrac12 \right) = 0
   \qquad \forall (k, y) \in \R \times 
   (0,\tfrac12),
 \end{align}
 where we noted that the regularity and decay of $\theta$ allows us to
 interchange the order of differentiation and an application of the
 Fourier transform distributionally. The solution of the above
 boundary value problem in terms of the boundary data is
 \begin{align}
   \label{eq:19}
   \hat \theta (k, y) = {\cosh ( k ( \tfrac12 - y)
   ) \over \cosh (\tfrac12 k)} \, \hat \theta (k, 0),
 \end{align}
 and upon inverting the Fourier transform we can write
 \begin{align}
   \label{eq:29}
   \theta(x, y) = \int_{\R} P(x - x', y) \, \bar \theta(x') \, dx', 
 \end{align}
 where
 \begin{align}
   \label{eq:33}
   P(x, y) := \frac{2 \cosh (\pi  x) \sin (\pi  y)}{\cosh (2 \pi  x) -
   \cos (2 \pi  y)} 
 \end{align}
 is the Poisson kernel for $\R \times (0,\frac12)$ with Neumann
 boundary condition at $y = \frac12$. Notice that by direct inspection
 the formula in \eqref{eq:29} remains valid if, for example,
 $\bar\theta \in C(\R) \cap L^\infty(\R)$ (see also \cite{widder61}).

 By square integrability of $|\nabla \theta(\cdot, y)|$ for $\theta$
 given by \eqref{eq:29} and Fubini theorem, we conclude that the
 Plancherel identity holds for every $y \in [0,\frac12]$ in the
 gradient squared term. Therefore, by \eqref{eq:19} we can write
 \begin{align}
   \label{eq:15}
   \mathcal F(\theta) = -\int_{\mathbb R} \hat K(k) |\hat \theta(k,
   0)|^2 \, {dk \over 2 \pi}  + 2 \gamma 
   \int_{\mathbb R} \sin^2 \bar \theta(x) \, dx,
 \end{align}
 where 
 \begin{align}
   \label{eq:20}
   \hat K(k) := -k \tanh(\tfrac12 k).  
 \end{align}
 Notice that by \eqref{eq:19} we have
 $\partial_y \hat \theta(k, 0) = \hat K(k) \hat \theta(k, 0)$, i.e.,
 $\hat K$ is the Fourier symbol of the Dirichlet-to-Neumann map at
 $y = 0$.
 
 To obtain a real space representation of \eqref{eq:15}, we regularize
 $\hat K(k)$ for $0 < \e < \frac12$:
 \begin{align}
   \label{eq:37}
   \hat K_\e(k) := k \left( \sinh (k \e) - \tanh ( \tfrac12 k) \cosh (k \e)
   \right),
 \end{align}
 and note that $0 > \hat K_\e(k) \searrow \hat K(k)$ as
 $\e \searrow 0$. Also, in view of the fact that
 $\hat K_\e(k) \hat \theta(k, 0) = \partial_y \hat \theta(k, \e)$ and
 \eqref{eq:29}, the inverse Fourier transform of $\hat K_\e(k)$ reads
 \begin{align}
   \label{eq:21}
   K_\e(x) & = \frac{2 \pi  \cos (\pi  \e) \cosh (\pi  x) (\cos (2 \pi
             \e)+\cosh (2 \pi  x)-2)}{(\cos (2 
             \pi  \e)-\cosh (2 \pi  x))^2}. 
 \end{align}
 In particular, passing to the limit $\e \to 0$ yields
 $K(x) = \lim_{\e \to 0} K_\e(x)$, where $K$ is given by
 \eqref{eq:10}.

 Lastly, the expression for the energy in \eqref{eq:9} follows from
 \eqref{eq:15} by an appropriate limiting argument with the help of an
 observation that
 \begin{align}
   \label{eq:22}
   \int_{\mathbb R} \hat K_\e (k)
   & |\hat \theta(k, 0)|^2 \,
     {dk \over 2
     \pi} \\
   & = \int_{\mathbb R} \int_{\mathbb R} K_\e(x - x')
     \bar\theta(x) \bar\theta(x') \, dx \, dx' = -\frac12  \int_{\mathbb
     R} \int_{\mathbb R} K_\e(x - x') (\bar\theta(x) -
     \bar\theta(x'))^2 \, dx \, dx',   \notag
 \end{align}
 where we noted that $\int_{\R} K_\eps(x) \, dx = \hat K_\e(0) = 0$
 and, hence, the function
 $v_\eps(x) := \int_{\R} K_\e(x - x') \, \bar\theta(x') \, dx'$ is
 smooth and exhibits exponential decay as $|x| \to \infty$. Indeed, we
 can pass to the limit in the left-hand side of \eqref{eq:22} by
 monotone convergence theorem to obtain the first term in the
 right-hand side of \eqref{eq:15}. At the same time, as can be easily
 seen we have $|K_\e(x)| \leq K(x)$ for all $0 < \e <
 \frac12$. Therefore, we can pass to the limit in the right-hand side
 of \eqref{eq:22} with the help of the dominated convergence theorem.


\bibliographystyle{plain}
\bibliography{../../nonlin,../../mura}

\begin{thebibliography}{10}

\bibitem{adams}
R.~A. Adams and J.~J.~F. Fournier.
\newblock {\em Sobolev spaces}.
\newblock Pure and Applied Mathematics. Academic Press, 2nds edition, 2003.

\bibitem{allwood05}
D.~A. Allwood, G.~Xiong, C.~C. Faulkner, D.~Atkinson, D.~Petit, and R.~P.
  Cowburn.
\newblock Magnetic domain-wall logic.
\newblock {\em Science}, 309:1688--1692, 2005.

\bibitem{apalkov16}
D.~Apalkov, B.~Dieny, and J.~M. Slaughter.
\newblock Magnetoresistive random access memory.
\newblock {\em Proc. IEEE}, 104:1796--1830, 2016.

\bibitem{arrieta99}
J.~M. Arrieta, A.~N. Carvalho, and A.~Rodr{\'{\i}}guez-Bernal.
\newblock Parabolic problems with nonlinear boundary conditions and critical
  nonlinearities.
\newblock {\em J. Differential Equations}, 156:376--406, 1999.

\bibitem{bader10}
S.~D. Bader and S.~S.~P. Parkin.
\newblock Spintronics.
\newblock {\em Ann. Rev. Cond. Mat. Phys.}, 1:71--88, 2010.

\bibitem{berestycki91}
H.~Berestycki and L.~Nirenberg.
\newblock On the method of moving planes and the sliding method.
\newblock {\em Bol. Soc. Brasil. Mat. (N.S.)}, 22:1--37, 1991.

\bibitem{berestycki92}
H.~Berestycki and L.~Nirenberg.
\newblock Traveling fronts in cylinders.
\newblock {\em Ann. Inst. H. Poincar\'e Anal. Non Lin\'eaire}, 9:497--572,
  1992.

\bibitem{bourgain01}
J.~Bourgain, H.~Brezis, and P.~Mironescu.
\newblock Another look at {Sobolev} spaces.
\newblock In E.~Rofman J.~L.~Menaldi and A.~Sulem, editors, {\em Optimal
  Control and Partial Differential Equations}, A volume in honour of A.
  Bensoussan's 60th birthday, pages 439--455. IOS Press, 2001.

\bibitem{bourgain05}
J.~Bourgain, H.~Brezis, and P.~Mironescu.
\newblock Lifting, degree, and distributional {J}acobian revisited.
\newblock {\em Comm. Pure Appl. Math.}, 58:529--551, 2005.

\bibitem{cabre05}
X.~Cabr{\'e} and J.~Sol{\`a}-Morales.
\newblock Layer solutions in a half-space for boundary reactions.
\newblock {\em Comm. Pure Appl. Math.}, 58:1678--1732, 2005.

\bibitem{cmy:mmnp17}
K.-S. Chen, C.~B. Muratov, and X.~Yan.
\newblock Layer solutions for a one-dimensional nonlocal model of
  {Ginzburg-Landau} type.
\newblock {\em Math. Model. Nat. Phenom.}, 12:68--90, 2017.

\bibitem{cm:non13}
M.~Chermisi and C.~B. Muratov.
\newblock One-dimensional {N\'eel} walls under applied external fields.
\newblock {\em Nonlinearity}, 26:2935--2950, 2013.

\bibitem{consul96}
N.~C\`onsul.
\newblock On equilibrium solutions of diffusion equations with nonlinear
  boundary conditions.
\newblock {\em Z. Angew. Math. Phys.}, 47:194--209, 1996.

\bibitem{dennis02}
C.~L. Dennis, R.~P. Borges, L.~D. Buda, U.~Ebels, J.~F. Gregg, M.~Hehn,
  E.~Jouguelet, K.~Ounadjela, I.~Petej, I.~L. Prejbeanu, and M.~J. Thornton.
\newblock The defining length scales of mesomagnetism: {A} review.
\newblock {\em J. Phys. -- Condensed Matter}, 14:R1175--R1262, 2002.

\bibitem{desimone06r}
A.~DeSimone, R.~V. Kohn, S.~M\"uller, and F.~Otto.
\newblock Recent analytical developments in micromagnetics.
\newblock In G.~Bertotti and I.~D. Mayergoyz, editors, {\em The Science of
  Hysteresis}, volume~2 of {\em Physical Modelling, Micromagnetics, and
  Magnetization Dynamics}, pages 269--381. Academic Press, Oxford, 2006.

\bibitem{dmrs:sima20}
G.~Di~Fratta, C.~B. Muratov, F.~N. Rybakov, and V.~V. Slastikov.
\newblock Variational principles of micromagnetics revisited.
\newblock {\em SIAM J. Math. Anal.}, 52:3580--3599, 2020.

\bibitem{dms21}
G.~Di~Fratta, C.~B. Muratov, and V.~V. Slastikov.
\newblock Reduced energy for thin ferromagnetic films with perpendicular
  anisotropy.
\newblock (In preparation).

\bibitem{dinezza12}
E.~Di~Nezza, G.~Palatucci, and E.~Valdinoci.
\newblock Hitchhiker's guide to the fractional {S}obolev spaces.
\newblock {\em Bull. Sci. Math.}, 136:521--573, 2012.

\bibitem{doering14}
L.~D\"oring, R.~Ignat, and F.~Otto.
\newblock A reduced model for domain walls in a reduced model for domain walls
  in soft ferromagnetic films at the cross-over from symmetric to asymmetric
  wall types.
\newblock {\em J. Eur. Math. Soc.}, 16:1377--1422, 2014.

\bibitem{evans}
L.~C. Evans and R.~L. Gariepy.
\newblock {\em Measure Theory and Fine Properties of Functions}.
\newblock CRC, Boca Raton, revised edition, 2015.

\bibitem{fukami09}
S.~{Fukami}, T.~{Suzuki}, K.~{Nagahara}, N.~{Ohshima}, Y.~{Ozaki}, S.~{Saito},
  R.~{Nebashi}, N.~{Sakimura}, H.~{Honjo}, K.~{Mori}, C.~{Igarashi},
  S.~{Miura}, N.~{Ishiwata}, and T.~{Sugibayashi}.
\newblock Low-current perpendicular domain wall motion cell for scalable
  high-speed {MRAM}s.
\newblock In {\em 2009 Symposium on VLSI Technology}, pages 230--231, 2009.

\bibitem{gaididei17}
Y.~Gaididei, A.~Goussev, V.~P. Kravchuk, O.~V. Pylypovskyi, J.~M. Robbins,
  V.~Slastikov D.~D.~Sheka, and S.~Vasylkevych.
\newblock Magnetization in narrow ribbons: curvature effects.
\newblock {\em J. Phys. A: Mat. Theor.}, 50:385401, 2017.

\bibitem{gilbarg}
D.~Gilbarg and N.~S. Trudinger.
\newblock {\em Elliptic Partial Differential Equations of Second Order}.
\newblock Classics in Mathematics. Springer, Berlin, 2001.

\bibitem{gradshteyn}
I.~Gradshteyn and I.~Ryzhik.
\newblock {\em Table of integrals, series, and products}.
\newblock Elsevier/Academic Press, Amsterdam, 7th edition, 2007.

\bibitem{grisvard}
P.~Grisvard.
\newblock {\em Elliptic problems in nonsmooth domains}, volume~24 of {\em
  Monographs and Studies in Mathematics}.
\newblock Pitman (Advanced Publishing Program), Boston, MA, 1985.

\bibitem{harutyunyan14}
D.~Harutyunyan.
\newblock Scaling laws and the rate of convergence in thin magnetic films.
\newblock {\em J. Math. Anal. Appl.}, 420:1744--1761, 2014.

\bibitem{harutyunyan16}
D.~Harutyunyan.
\newblock On the existence and stability of minimizers in ferromagnetic
  nanowires.
\newblock {\em J. Math. Anal. Appl.}, 434:1719--1739, 2016.

\bibitem{heinze}
S.~Heinze.
\newblock A variational approach to traveling waves.
\newblock Technical Report~85, Max Planck Institute for Mathematical Sciences,
  Leipzig, 2001.

\bibitem{hubert}
A.~Hubert and R.~Sch\"afer.
\newblock {\em Magnetic Domains}.
\newblock Springer, Berlin, 1998.

\bibitem{ignat10}
R.~Ignat and H.~Kn{\"u}pfer.
\newblock Vortex energy and {$360^\circ$} {N}\'eel walls in thin-film
  micromagnetics.
\newblock {\em Comm. Pure Appl. Math.}, 63:1677--1724, 2010.

\bibitem{ignat17}
R.~Ignat and R.~Moser.
\newblock N{\'e}el walls with prescribed winding number and how a nonlocal term
  can change the energy landscape.
\newblock {\em J. Differential Equations}, 263:5846--5901, 2017.

\bibitem{jang12}
Y.~Jang, S.~R. Bowden, M.~Mascaro, J.~Unguris, and C.~A. Ross.
\newblock Formation and structure of 360 and 540 degree domain walls in thin
  magnetic stripes.
\newblock {\em Appl. Phys. Lett.}, 100:062407, 2012.

\bibitem{klaui08b}
M.~Kl{\"a}ui.
\newblock Head-to-head domain walls in magnetic nanostructures.
\newblock {\em J. Phys. -- Condensed Matter}, 20:313001, 2008.

\bibitem{klaui04}
M.~Kl{\"a}ui, C.~A.~F. Vaz, J.~A.~C. Bland, L.~J. Heyderman, F.~Nolting,
  A.~Pavlovska, E.~Bauer, S.~Cherifi, S.~Heun, and A.~Locatelli.
\newblock Head-to-head domain-wall phase diagram in mesoscopic ring magnets.
\newblock {\em Appl. Phys. Lett.}, 85:5637--5639, 2004.

\bibitem{kmn:arma19}
H.~Kn\"upfer, C.~B. Muratov, and F.~Nolte.
\newblock Magnetic domains in thin ferromagnetic films with strong
  perpendicular anisotropy.
\newblock {\em Arch. Rat. Mech. Anal.}, 232:727--761, 2019.

\bibitem{knupfer20}
H.~Kn\"upfer and W.~Shi.
\newblock {$\Gamma$}-limit for two-dimensional charged magnetic zigzag domain
  walls.
\newblock arXiv:2005.02857, 2020.

\bibitem{kohn05arma}
R.~V. Kohn and V.~V. Slastikov.
\newblock Another thin-film limit of micromagnetics.
\newblock {\em Arch. Ration. Mech. Anal.}, 178:227--245, 2005.

\bibitem{kuehn06}
K.~K\"uhn.
\newblock Scaling laws of domain walls in magnetic nanowires.
\newblock Technical Report~58, Max Planck Institute for Mathematical Sciences,
  2006.

\bibitem{kuehn07}
K.~K\"uhn.
\newblock {\em Reversal modes in magnetic nanowires}.
\newblock PhD thesis, Max Planck Institute for Mathematics in the Sciences,
  2007.

\bibitem{kunz09}
A.~Kunz.
\newblock Field induced domain wall collisions in thin magnetic nanowires.
\newblock {\em Appl. Phys. Lett.}, 94:132502, 2009.

\bibitem{kurzke06}
M.~Kurzke.
\newblock Boundary vortices in thin magnetic films.
\newblock {\em Calc. Var. Partial Differential Equations}, 26:1--28, 2006.

\bibitem{landau8}
L.~D. Landau and E.~M. Lifshitz.
\newblock {\em Course of Theoretical Physics}, volume~8.
\newblock Pergamon Press, London, 1984.

\bibitem{laufenberg06}
M.~Laufenberg, D.~Backes, W.~B{\"u}hrer, D.~Bedau, M.~Kl{\"a}ui,
  U.~R{\"u}diger, C.~A.~F. Vaz, J.~A.~C. Bland, L.~J. Heyderman, F.~Nolting,
  S.~Cherifi, A.~Locatelli, R.~Belkhou, S.~Heun, and E.~Bauer.
\newblock Observation of thermally activated domain wall transformations.
\newblock {\em Appl. Phys. Lett.}, 88:052507, 2006.

\bibitem{li01}
S.~P. Li, D.~Peyrade, M.~Natali, A.~Lebib, Y.~Chen, U.~Ebels, L.~D. Buda, and
  K.~Ounadjela.
\newblock Flux closure structures in cobalt rings.
\newblock {\em Phys. Rev. Lett.}, 86(6):1102--1105, Feb 2001.

\bibitem{lieb-loss}
E.~H. Lieb and M.~Loss.
\newblock {\em Analysis}.
\newblock American Mathematical Society, Providence, RI, 2010.

\bibitem{lms:jns20}
R.~G. Lund, C.~B. Muratov, and V.~V. Slastikov.
\newblock Edge domain walls in ultrathin exchange-biased films.
\newblock {\em J. Nonlinear Sci.}, 30:1165--1205, 2018.

\bibitem{lms:non18}
R.~G. Lund, C.~B. Muratov, and V.~V. Slastikov.
\newblock One-dimensional in-plane edge domain walls in ultrathin ferromagnetic
  films.
\newblock {\em Nonlinearity}, 31:728--754, 2018.

\bibitem{manipatruni18}
S.~Manipatruni, D.~E. Nikonov, and I.~A. Young.
\newblock Beyond {CMOS} computing with spin and polarization.
\newblock {\em Nature Phys.}, 14:338--343, 2018.

\bibitem{mcmichael97}
R.~D. McMichael and M.~J. Donahue.
\newblock Head to head domain wall structures in thin magnetic strips.
\newblock {\em {IEEE} Trans. Magn.}, 33:4167--4169, 1997.

\bibitem{morrey66}
C.~B. Morrey, Jr.
\newblock {\em Multiple integrals in the calculus of variations}.
\newblock Die Grundlehren der mathematischen Wissenschaften, Band 130.
  Springer-Verlag New York, Inc., New York, 1966.

\bibitem{mo:jcp06}
C.~B. Muratov and V.~V. Osipov.
\newblock Optimal grid-based methods for thin film micromagnetics simulations.
\newblock {\em J. Comp. Phys.}, 216:637--653, 2006.

\bibitem{mo:ieeetm09}
C.~B. Muratov and V.~V. Osipov.
\newblock Bit storage by $360^\circ$ domain walls in ferromagnetic nanorings.
\newblock {\em IEEE Trans. Magn.}, 45:3207--3209, 2009.

\bibitem{ms:prsla17}
C.~B. Muratov and V.~V. Slastikov.
\newblock Domain structure of ultrathin ferromagnetic elements in the presence
  of {Dzyaloshinskii-Moriya} interaction.
\newblock {\em Proc. R. Soc. Lond. Ser. A}, 473:20160666, 2017.

\bibitem{nakatani05}
Y.~Nakatani, A.~Thiaville, and J.~Miltat.
\newblock Head-to-head domain walls in soft nano-strips: a refined phase
  diagram.
\newblock {\em J. Magn. Magn. Mater.}, 290-291:750--753, 2005.

\bibitem{necas}
Jind\v{r}ich Ne\v{c}as.
\newblock {\em Direct methods in the theory of elliptic equations}.
\newblock Springer Monographs in Mathematics. Springer, Heidelberg, 2012.

\bibitem{palatucci13}
G.~Palatucci, O.~Savin, and E.~Valdinoci.
\newblock Local and global minimizers for a variational energy involving a
  fractional norm.
\newblock {\em Annali di Matematica}, 192:673--718, 2013.

\bibitem{parkin08}
S.~S.~P. Parkin, M.~Hayashi, and L.~Thomas.
\newblock Magnetic domain-wall racetrack memory.
\newblock {\em Science}, 320:190--194, 2008.

\bibitem{pucci04}
P.~Pucci and J.~Serrin.
\newblock The strong maximum principle revisited.
\newblock {\em J. Differ. Equations}, 196:1--66, 2004.

\bibitem{ross_patent05}
C.~A. Ross and F.~J. Castano.
\newblock Magnetic memory elements using $360^\circ$ walls.
\newblock US Patent 6,906,369 B2, 2005.

\bibitem{sharard12}
M.~{Sharad}, C.~{Augustine}, G.~{Panagopoulos}, and K.~{Roy}.
\newblock Spin-based neuron model with domain-wall magnets as synapse.
\newblock {\em IEEE Trans. Nanotechnol.}, 11:843--853, 2012.

\bibitem{slastikov12}
V.~V. Slastikov and C.~Sonnenberg.
\newblock Reduced models for ferromagnetic nanowires.
\newblock {\em IMA J. Appl. Math.}, 77:220--235, 2012.

\bibitem{stepanova}
M.~Stepanova and S.~Dew, editors.
\newblock {\em Nanofabrication: Techniques and Principles}.
\newblock Springer-Verlag, Wien, 2012.

\bibitem{tchernyshyov05}
O.~Tchernyshyov and G.-W. Chern.
\newblock Fractional vortices and composite domain walls in flat nanomagnets.
\newblock {\em Phys. Rev. Lett.}, 95:197204, 2005.

\bibitem{thiaville09}
A.~Thiaville and Y.~Nakatani.
\newblock Chapter 6 - micromagnetics of domain-wall dynamics in soft
  nanostrips.
\newblock In T.~Shinjo, editor, {\em Nanomagnetism and Spintronics}, pages
  231--276. Elsevier, Amsterdam, 2009.

\bibitem{toland97}
J.~F. Toland.
\newblock The {Peierls-Nabarro} and {Benjamin-Ono} equations.
\newblock {\em J. Funct. Anal.}, 145:136--150, 1997.

\bibitem{widder61}
G.~N. Widder.
\newblock Functions harmonic in a strip.
\newblock {\em Proc. Amer. Math. Soc.}, 12:67--72, 1961.

\bibitem{zhang15}
J.~Zhang, S.~A. Siddiqui, P.~Ho, J.~A. Currivan-Incorvia, L.~Tryputen, E.~Lage,
  D.~C. Bono, M.~A. Baldo, and C.~A. Ross.
\newblock 360$^\circ$ domain walls: Stability, magnetic field and electric
  current effects.
\newblock {\em New J. Phys.}, 18:053028, 2015.

\bibitem{zhu06rev}
J.-G. Zhu and C.~Park.
\newblock Magnetic tunnel junctions.
\newblock {\em Materials Today}, 9:36--45, 2006.

\bibitem{zhu03}
X.~Zhu and J.-G. Zhu.
\newblock A vertical {MRAM} free of write disturbance.
\newblock {\em IEEE Trans. Magn.}, 39:2854--2856, 2003.

\end{thebibliography}
 
\end{document}